\documentclass[reqno,12pt,oneside]{amsart}
\usepackage[utf8]{inputenc}
\usepackage{amssymb,amsthm,amsfonts}
\usepackage{amsmath}
\usepackage{tikz}
\usepackage{graphicx}
\usepackage{psfrag}
\usepackage{cancel}
\usepackage{hyperref}
\usepackage{multicol}
\usepackage{comment}
\newtheorem{theorem}{Theorem}[section]
\newtheorem{definition}[theorem]{Definition}
\newtheorem{lemma}[theorem]{Lemma}
\newtheorem{proposition}[theorem]{Proposition}
\newtheorem{corollary}[theorem]{Corollary}
\newtheorem{remark}[theorem]{Remark}

\newtheorem*{mtheorem}{Main Theorem} 
\numberwithin{equation}{section}
\usepackage[top=3cm,bottom=2.5cm,left=2.5cm, right=3cm,marginparwidth=0cm]{geometry}

\title[]{Dynamics of asymptotically holomorphic polynomial-like maps}

\thanks{
This work has been partially supported by 
``Projeto Tem\'atico Din\^amica em Baixas Dimens\~oes''; FAPESP Grants: 2016/25053-8, 
BPE 2016/25970-0, Sprint Project  2017/50139; ERC AdG grant no 339523 RGDD; EU Marie-Curie IRSES Brazilian-European partnership in Dynamical Systems (FP7-PEOPLE-2012-IRSES 318999 BREUDS) and  EU Marie-Sklodowska-Curie ITN Critical Transitions in Complex Systems (H2020-MSCA-ITN-2014 643073 CRITICS)}

\author{Trevor Clark} 
\address{Trevor Clark, Imperial College London, London, UK}
\email[]{t.clark@imperial.ac.uk}

\author{Edson de Faria} 
\address{Edson de Faria, Instituto de Matem\'{a}tica e Estat\'{i}stica, USP, S\~{a}o Paulo, SP, Brazil}
\email[]{edson@ime.usp.br}

\author{Sebastian van Strien} 
\address{Sebastian van Strien, Imperial College London, London, UK}
\email[]{s.van-strien@imperial.ac.uk}

\begin{document}
 \begin{abstract}
  The purpose of this paper is  to initiate 
  a theory concerning the dynamics of asymptotically holomorphic polynomial-like maps. 
 Our maps arise naturally as  deep renormalizations of asymptotically holomorphic extensions of $C^r$ ($r>3$)
 unimodal maps that are infinitely renormalizable of bounded type. Here we prove a version of the 
 Fatou-Julia-Sullivan theorem and a topological straightening theorem in this setting. 
In particular, these maps do not have wandering domains and their Julia sets are locally connected. 
\end{abstract}

\maketitle

\section{Introduction}

Over the last decades many remarkable results were obtained  for rational maps of the Riemann sphere, and 
somewhat surprisingly it turned out that quite a few of these have an analogue
in the case of smooth interval maps. For example, the celebrated Julia-Fatou-Sullivan structure theorem
for rational maps establishes the absence of wandering domains, showing that 
each component of the Fatou set is eventually periodic, and moreover gives a simple
classification of the possible dynamics on a periodic component of the Fatou set, see \cite{sullivan0}. 
For smooth interval maps analogous results were obtained, starting with Denjoy's results for $C^2$ circle diffeomorphisms dating back to  1932.  We now know that $C^2$ interval or circle maps cannot have wandering intervals provided all their critical points
are non-flat,  proved  in increasing generality in \cite{Gu,Ly1,BL,dMvS,MMS,demelovanstrien,vSV}.
Interestingly, although the statements for the Julia-Fatou-Sullivan structure theorem for rational maps and 
the generalised Denjoy theorems for interval and circle maps are analogous,  the proofs use entirely different ideas. 
In the former case, they rely on the Measurable Riemann Mapping Theorem (MRMT) while in the latter case the proofs rely on 
real bounds coming from $C^2$ distortion estimates together with arguments relating to the order structure of the real line. 

However, overall, not only the results but also the techniques used in the fields of holomorphic dynamics and interval dynamics 
have become increasingly intertwined over the last decades. Indeed, within the literature of real one-dimensional 
dynamics a growing number of 
results 
are obtained under the additional assumption that
the maps are real analytic 
rather than smooth. The reason for this is that a real analytic map (obviously)
has a complex extension to a small neighbourhood in $\mathbb C$ of the dynamical interval,
and therefore many tools from complex analysis can be applied to such a real map. 
For instance, many results in the theory of renormalization of interval maps are either 
not known in the smooth category, or were only obtained with a significant amount of additional effort. 
Specifically, the Feigenbaum-Coullet-Tresser conjectures were first obtained using computer 
supported proofs, {\em e.g.}  \cite{Lanford} and later using conceptual
proofs for real analytic unimodal interval maps in
\cite{sullivan,McM,Ly3,AL},
for real analytic circle homeomorphisms  with critical points in \cite{edsonwelington1,edsonwelington2,yam,KT}, and for certain 
multimodal maps in \cite{Smania1,Smania2,Smania3,Smania4}. All these later results heavily use  complex analytic 
machinery, and in particular rely on the complex analytic extensions
of interval maps. 

Within the literature on holomorphic dynamics one sees a similar development: 
many conjectures about iterations of general polynomials are only
solved in the context of polynomials with real coefficients.
%
%
%
An example of such a conjecture is density of hyperbolicity
which is unsolved in the general case but  was proved for real
quadratic maps independently by Lyubich and Graczyk - Swiatek
and in the general case by Kozlovski, Shen and van Strien, see \cite{GS2,Ly1,KSS1,KSS2}.
These results heavily rely on the existence of so-called real and complex bounds, \cite{LvS,GS1,LyY,Shen,CvST}
but such complex bounds do not hold for general non-real polynomials or rational maps. 
Indeed they 
hold for non-renormalizable polynomial maps \cite{Yoc, Hub, KvS} but in general not 
for non-real infinitely renormalizable quadratic maps, see for example \cite{Mil2,Sor}.

Of course there are plenty of results on renormalization and towards density of hyperbolicity
in the setting of non-real polynomials \cite{Ly2,Ka,KL, KL2, InSh, ChSh} and similarly there are plenty impressive results
on interval maps which do not use complex tools, on for example
invariant measures, thermodynamic formalism and stochastic stability. 
Nevertheless it is fair to say that
%
a growing number of results within the field of real one-dimensional dynamics crucially rely on complex analytic tools, 
and vice versa many results about polynomial maps are only known when these preserve  the real line. 

When studying real one-dimensional maps, it is unnatural to restrict attention to maps
which are real analytic.
Indeed, in certain cases renormalization results for real analytic interval maps can be extended to $C^3$ or $C^4$ maps.
This was done using a functional analytic approach in \cite{deFariadeMeloPinto} for unimodal interval maps
and heavily exploiting what is known for real analytic circle homeomorphisms in \cite{GdM}. A purely real approach which gives 
existence of periodic points of the renormalization operator for unimodal maps of the form $g(|x|^\ell)$, $\ell>1$,  
 was obtained  by Martens \cite{Martens}. 

The purpose of this paper is to initiate a theory for $C^{3+}$ interval maps showing that these have extensions
to the complex plane with properties analogous to those of real polynomial maps. 
Thus the eventual aim of this theory is to show that $C^{3+}$ maps can be treated with techniques which 
are very similar to the complex analytic techniques which were so fruitful in the case of polynomial 
and real-analytic maps. 

In this paper we will establish the first cornerstone of this theory by showing that one has a {\em Julia-Fatou-Sullivan} type description for
such maps in a very important situation, namely for infinitely renormalizable maps  of bounded type.

Let us be more precise and consider a $C^r$ map $f\colon I\to \mathbb R$. Such a map $f$
has an extension to  a $C^{r}$ map $F\colon \mathbb C \to \mathbb C$ 
which is {\em asymptotically holomorphic of order $r$}, {\em i.e.,} $\frac{\partial}{\partial\bar{z}}F(z)=0$ when $\mathrm{Im}\, z=0$  
and  $\frac{\partial}{\partial\bar{z}}F(z)=O(| \mathrm{Im}\,{z} |^{r-1})$ uniformly, see \cite{GSS}. 
The notion of asymptotically holomorphic maps goes back at least to \cite{Car}. In dynamics this notion was used in 
\cite{Ly1b}, \cite{sullivan},  \cite{CvST},  \cite{GdM}, \cite{AK}, \cite{CvS} (see also \cite{GaSu1}, \cite{GaSu2} for 
related material on the more restrictive 
notion of uniformly asymptotically conformal (UAC) map). 
Note that $F$ is not conformal outside the real line, and so in principle periodic points
can be of saddle type. Even if a periodic point is repelling,
in general the linearization at such a point will not be conformal. 
It follows that $F$ cannot be quasiconformally conjugate to a
polynomial-like map (the pullbacks of a  small circle in a small
neighbourhood of a non-conformal repelling point 
become badly distorted, but this is not the case in a 
small neighbourhood of a conformal repelling point).
For this reason, the absence of wandering domains for $F$ cannot be obtained via Sullivan's Nonwandering Domains Theorem \cite{sullivan}. 

%

\begin{mtheorem} \label{ThmMain}
Let $f\in C^{3+\alpha}$ ($\alpha>0$) be a unimodal, infinitely renormalizable interval map of bounded type whose 
critical point has criticality given by an even integer $d$. Then every $C^{3+\alpha}$ 
extension $F$ of $f$ to a map defined on a neighborhood of the interval in the complex plane is such that there exist a 
sequence of domains $U_n\subset V_n\subset \mathbb{C}$ containing the critical point
of $f$ and iterates $q_n$ with the following properties. 
\begin{enumerate}
\item The map $G:=F^{q_n}\colon U_n\to V_n$ is a degree $d$, quasi-regular polynomial-like map.
\item For large enough $n$, each periodic point in the filled Julia set 
 $\mathcal K_{G}:=\{z\in U_n; G^i(z)\in U_n\,\, \forall i\ge 0\}$ 
 is repelling. 
 \item The Julia $\mathcal J_G:= \partial \mathcal K_G$ and filled-in Julia set of $G$ coincide, {\em i.e.,}
$\mathcal J_G=\mathcal K_G$.
\item The map $G$ is topologically conjugate to a polynomial
mapping in a neighbourhood of its Julia set. In particular,
$G$  has no wandering domains.
\item The Julia set $\mathcal J_{G}$ is locally connected.
\end{enumerate}
\end{mtheorem}

A more precise statement of this theorem can be found in Corollary~\ref{CorMain} where
we use the notion of {\it controlled AHPL-maps\/}, see  Definition~\ref{def:control}.
We expect a similar result to hold in much greater generality, for example for general $C^{3+\alpha}$ asymptotically holomorphic interval maps with finitely many critical points of integer order. 

Our plan is to build on the results in this paper to prove absence of invariant line fields
for asymptotically holomorphic maps extending the methods of \cite{McM}. 
In addition,  rather than using functional analytic tools 
as in \cite{deFariadeMeloPinto}, we plan to prove renormalization results for $C^r$ maps
through the McMullen tower construction directly following the ideas in \cite{McM}, or more ambitiously 
following the approach of Avila-Lyubich \cite{AL}. 
Thus our ultimate goal is to establish a closer analogy between real and complex one-dimensional dynamics
along the lines suggested in the table below.

\bigskip 
{\small 
\begin{tabular}{ |  l  |  l | l |    }
\hline \hline 
setting & {\em real polynomials on the complex plane}  & {\em $C^3$ asympt. hol. maps} \\ 
\hline 
analogy & Julia-Fatou-Sullivan Theory & Yes (this paper) \\
& McMullen tower construction & ? \\
& Schwarz contraction  & ? \\ 
& Hyperbolicity of Renormalization & ? \\ 
& Deformation theory (through MRMT) & ? \\ 
\hline 
\end{tabular}}
\bigskip 

\subsection{Object of study}\label{sec:genpoly}

We shall study the dynamics of certain {\it quasi-regular maps\/} in the complex plane that are generalizations of
standard (holomorphic) polynomial-like maps, as defined by Douady-Hubbard in \cite{DH}. 
Such generalized polynomial-like maps arise as deep renormalizations of 
unimodal interval maps that admit an {\it asymptotically holomorphic\/} extension to a complex neighborhood of their 
real domain. Let $\varphi: U\to V$ be a $C^1$ map between two domains in the complex plane, and assume that $U\cap \mathbb{R}\neq \O$. 
We say that $\varphi$ is {\it asymptotically holomorphic of order $r>1$\/} if $\varphi$ is quasi-regular and its complex dilatation $\mu_\varphi$ 
satifies $|\mu_\varphi(z)|\leq C|\mathrm{Im}\,{z}|^{r-1}$ for all $z\in U$ and some constant $C>0$ 
(in particular, $\mu_\varphi$ vanishes on the real axis, {\it i.e.,\/} 
$\varphi$ is conformal there).  As mentioned above, every $C^r$ map of the real line admits an extension to 
a neighborhood of the real axis which is asymptotically holomorphic of order $r$. 
(The notion of asymptotically holomorphic maps can even be defined for maps which are merely 
quasiconformal on $\mathbb C$.  It can be shown that if such a map  is asymptotically holomorphic of order $r$ then its restriction 
to the real line is actually $C^r$, see \cite{AH,Dyn}.)

We may now formally define the class of dynamical systems we intend to study. 
Please note that in what follows we only consider maps having a unique critical point of finite even order $d\geq 2$. 

\begin{definition} \label{AHPL-map}
Let $U,V\subset \mathbb{C}$ be Jordan domains symmetric about the real axis, and suppose $U$ is compactly contained in $V$.
A $C^r$ ($r\geq 3$) map $f:U\to V$ is said to be an {\emph{asymptotically holomorphic polynomial-like map\/}}, or AHPL-map for short, if 
\begin{enumerate}
 \item[(i)] $f$ is a degree $d\geq 2$ proper branched covering map of $U$ onto $V$, branched at a unique critical point $c\in U\cap \mathbb{R}$ 
 of criticality given by $d$;
 \item[(ii)] $f$ is symmetric about the real axis, {\it i.e.,\/} $f(\overline{z})=\overline{f(z)}$ for all $z\in U$;
 \item[(iii)] $f$ is asymptotically holomorphic of order $r$. 
\end{enumerate}
\end{definition}

It follows from the well-known {\it Stoilow Factorization Theorem\/} (see \cite[Cor.~5.5.3]{AIM}) that an AHPL-map $f$ as above can be written as 
$f=\phi\circ g$, where $g:U\to V$ is a (holomorphic) polynomial-like map and $\phi:V\to V$ is a $C^r$ quasiconformal diffeomorphism 
which is also asymptotically holomorphic of order $r$. 

Just as in the case of standard polynomial-like maps, we define the {\it filled-in Julia set\/} of an AHPL-map $f:U\to V$ 
to be the closure of the set of points which never escape under iteration, namely
\[
 \mathcal{K}_f\;=\;\bigcap_{n\geq 0} f^{-n}(V)\;=\; \bigcap_{n\geq 0} f^{-n}(\overline{U})\ .
\]
This is a compact, totally $f$-invariant subset of $U$. Its boundary $J_f=\partial \mathcal{K}_f$ is called the {\it Julia set\/} of 
$f$. By simple analogy with the case of holomorphic polynomial-like maps, there are natural questions to be asked about AHPL-maps and 
their Julia sets, to wit:
\begin{enumerate}
 \item[(1)] Are the (expanding) periodic points dense in $J_f$?
 \item[(2)] When is $J_f$ locally connected? 
 \item[(3)] What is the classification of stable components of $\mathcal{K}_f\setminus J_f$? 
 \item[(4)] Can $f$ have non-wandering domains? 
 \item[(5)] Is there a (topological) straightening theorem for AHPL-maps? 
\end{enumerate}
These questions do not have obvious answers. For instance, in the holomorphic case, the first question has an affirmative 
answer whose proof is easy thanks to Montel's theorem -- a tool which is not useful here. Likewise, in the holomorphic case 
question (4) has a negative answer thanks to Sullivan's non-wandering domains theorem, whose proof uses quasiconformal deformations 
of $f$ in a way that is not immediately available here, because in general the iterates of an AHPL-map are not uniformly 
quasiconformal. 

Rather than studying very general AHPL-maps, in this paper we will restrict our attention to those which can be {\it renormalized\/}, 
in fact infinitely many times. The definition of {\it
  renormalization\/} in the present context is the same as the one for polynomial-like mappings:
an AHPL-map $f$
is {\it renormalizable\/} if there exists a topological disk $D$ containing the critical point of $f$ and an integer $p>1$ 
so that $D$ is compactly contained in $f^p(D)$ 
and $f^p:D\to f^p(D)$ is again an AHPL-map. 
Thanks to a theorem proved in \cite{CvST}, every sufficiently deep renormalization of an asymptotically 
holomorphic map whose restriction to the real line is an infinitely renormalizable map (in the usual real sense) is an 
(infinitely renormalizable) AHPL-map with {\em a priori} bounds. 

One of our goals in the present paper is to provide answers to (some of) the above questions {\it under the assumption that 
the AHPL-map $f$ is infinitely renormalizable of bounded type.\/}
Another goal will be to prove $C^2$ {\em a priori} bounds for the renormalizations of such an $f$, under the same bounded type assumption.

\subsection{Summary} Here is a brief description of the contents of this paper. 
We start by revisiting the real bounds for $C^3$ unimodal maps in \S \ref{sec:revisitbounds}. 
In \S \ref{sec:c2bounds}, we prove that the successive renormalizations of a $C^3$ infinitely renormalizable 
AHPL-map of bounded type are uniformly bounded in the $C^2$ topology, and that such bounds are {\it beau\/} in the sense of 
Sullivan. In proving these bounds, we employ as a tool the {\it matrix form of the chain rule\/} for the 
{\it second\/} derivative of a composition of maps. This tool does not seem to have been used at all in the 
literature on low-dimensional dynamics. The key ingredient that allows us to prove our Main Theorem 
is a result that, roughly speaking, states that (a deep
renormalization of) an AHPL-map is an {\it infinitesimal expansion\/} of
the hyperbolic metric on its co-domain minus the real axis. This is the main result in \S \ref{sec:control}, namely 
Theorem \ref{control}.

In \S \ref{sec:hyperbolic} we introduce techniques which are crucial in establishing Theorem \ref{control},
namely Proposition~\ref{prop52} and Theorem \ref{exphypmetric}. 
Specifically, we give a bound for the hyperbolic Jacobian of a $C^2$ quasiconformal map in terms of its local quasiconformal distortion in two situations: for maps with small dilatation and for maps which are asymptotically holomorphic. 
These bounds are 
applied to the diffeomorphic part of our AHPL-map, which therefore needs to be at least $C^2$ with good bounds. 
This is the main reason why we need the $C^2$ bounds developed in \S \ref{sec:c2bounds}. 
This infinitesimal expansion of the hyperbolic metric has several consequences, {\it e.g.,} the fact that 
every periodic point of (a sufficiently deep renormalization of) 
an AHPL-map is expanding -- once again, see Theorem \ref{control}. 

Finally, in \S \ref{localconnectivity}, we go further and construct puzzle pieces for such AHPL-maps, and show 
with the help of Theorem \ref{control}, that the puzzle pieces containing any given point of the Julia set 
of an infinitely renormalizable AHPL-map shrink around that point. This implies that 
the Julia set of such a map is always locally connected.
Even more, as a consequence, such a map is in fact topologically conjugate 
to an actual (holomorphic) polynomial-like map and therefore does not have wandering domains.

\section{Revisiting the real bounds}\label{sec:revisitbounds}
In this section we will recall some basic facts about renormalization of real unimodal maps. 

\subsection{Renormalization of unimodal maps}\label{sec:realbounds}

We need to recall some definitions and a few facts concerning the {\it renormalization theory\/} of interval maps.
Let us consider a $C^3$ {\it unimodal\/} map $f:I\to I$ defined on the interval $I=[-1,1]\subset \mathbb{R}$, 
with its unique critical point at $0$ and corresponding critical value at $1$,  {\it i.e.\/},
with $f'(0)=0$ and $f(0)=1$.  From the viewpoint of renormalization, to be defined below, there is no loss of generality 
in assuming that $f$ is \emph{even}, {\it i.e.,} that $f(-x)=f(x)$ for all $x \in I$. We also assume that the 
critical point of $f$ has {\it finite even order $d\geq 2$\/}. Hence we oftentimes refer to $f$ as a {\it $d$-unimodal map\/}. 

We say that such an $f$ is {\it renormalizable\/} 
if there exist an integer $p=p(f)>1$ and $\lambda=\lambda(f)=f^p(0)$ such that $f^p|[-|\lambda|,|\lambda|]$ is unimodal and 
maps $[-| \lambda|,|\lambda|]$ into itself. Taking $p$ the smallest possible, we define the \emph{first renormalization\/} of $f$ 
to be the map $Rf:I\to I$ given by
\begin{equation}
Rf(x)\;=\;\frac{1}{\lambda}\,f^p(\lambda x)\ .
\end{equation}
The intervals $\Delta_j=f^j([-|\lambda|,|\lambda|])$, for $0\leq j\leq p-1$, have pairwise disjoint interiors, 
and their relative order inside $I_0$ determines a {\it unimodal} permutation $\theta$ of $\{0,1,\ldots,p-1\}$. 
Thus, renormalization consists of a first return map to a small 
neighbourhood of the critical point rescaled to unit size via a linear rescale. 

It makes sense to ask whether $Rf$ is also renormalizable, since $Rf$ is certainly a normalized unimodal map. If 
the answer is yes then one can define $R^2f=R(Rf)$, and so on.  In particular, it may be the case that the unimodal map $f$ is 
{\it infinitely renormalizable\/}, in the sense that the entire sequence of renormalizations $f, Rf, R^2f, \ldots, R^nf, \ldots$ is well-defined. 

We assume from now on that $f$ is infinitely renormalizable. Let us denote by $P(f)\subseteq I$ the closure 
of the forward orbit of the critical point under $f$ (the {\it post-critical set\/} of $f$). The set $P(f)$ is a Cantor set 
with zero Lebesgue measure, see below. 
It can be shown also that $P(f)$ is  
the {\it global attractor\/} of $f$ both from the \emph{topological} and \emph{metric} points of view.


Note that for each $n\ge 0$, we can write 
\[R^n f (x) = \frac{1}{\lambda_n}  f^{q_n} (\lambda_n x)\ ,
\]
where 
$q_0=1$, $\lambda_0=1$, $q_n=\prod_{i=0}^{n-1} p(R^if)$ and $\lambda_n=\prod_{i=0}^{n-1}  \lambda (R^if)=f^{q_n}(0)$. 
The positive integers $a_i=p(R^if)\geq 2$ are called the {\it renormalization periods\/} of $f$, and the 
$q_n$'s are the {\it closest return times\/} of the orbit of the critical point. 
Note that $q_{n+1}=a_nq_n=\prod_{i=0}^{i=n}a_i \geq 2^{n+1}$; in particular, the sequence 
$q_n$ goes to infinity at least exponentially fast.

It will be important to consider the {\it renormalization intervals of $f$ at level $n$\/}, namely  
$\Delta_{0,n}=[-|\lambda_n|,|\lambda_n|]
\subset I_0$, and
$\Delta_{i,n}=f^i(\Delta_{0,n})$ for
$i=0,1, \ldots, q_n-1$.
The collection
${\mathcal C}_n=\{\Delta_{0,n}, \ldots, \Delta_{q_n-1,n}\}$
consists of pairwise disjoint intervals.
Moreover,
$\bigcup \{\Delta:\Delta\in {\mathcal C}_{n+1}\} \subseteq \bigcup
\{\Delta:\Delta\in {\mathcal C}_n\}$ for all $n \geq 0$ and we have
$$
P(f)=
\bigcap _{n=0}^{\infty}
\bigcup_{i=0}^{q_n-1}
\Delta_{i,n} \ .
$$
Once we know that $\max_{0\leq i\leq q_n-1} |\Delta_{i,n}|\to 0$ as $n\to \infty$, it follows that $P(f)$ is, indeed, a Cantor set. 
This (and much more) follows from the so-called {\it real a priori bounds\/} proved by Sullivan in \cite{sullivan}. The following form 
of the real bounds is not the most general, but it will be quite sufficient for our purposes.
We say that an infinitely renormalizable map $f$ as above has {\it combinatorial type bounded by $N$\/} if its remormalization periods 
are bounded by $N$, {\it i.e.,} $a_n\leq N$ for all $n\in \mathbb{N}$.  

\begin{theorem}[Real Bounds] \label{realbounds}
Let $f:I\to I$ be a $C^3$ unimodal map as above, and suppose that $f$ is infinitely renormalizable with 
combinatorial type bounded by $N>1$. Then there exist constants $K_f>0$ and $0<\alpha_f< \beta_f<1$ such that the following holds 
for all $n\in \mathbb{N}$. 
\begin{enumerate}
 \item[(i)] If $\Delta\in \mathcal{C}_{n+1}$, $\Delta^*\in \mathcal{C}_n$ and $\Delta\subset \Delta^*$, then 
$\alpha_f|\Delta^*|\leq |\Delta| \leq \beta_f|\Delta^*|$. 
 \item[(ii)] For all $1\leq i<j\leq q_n-1$ and each $x\in \Delta_{i,n}$, we have
 \[
  \frac{1}{K_f}\frac{|\Delta_{j,n}|}{|\Delta_{i,n}|}\;\leq\; |(f^{j-i})'(x)|\;\leq\;  K_f \frac{|\Delta_{j,n}|}{|\Delta_{i,n}|}\ .
 \]
 \item[(iii)] We have $\|R^nf\|_{C^1(I)}\;\leq\; K_f$. 
\end{enumerate}
Moreover, there exist positive constants $K=K(N)$, $\alpha=\alpha(N)$, $\beta=\beta(N)$, with $0<\alpha < \beta<1$,  
and $n_0=n_0(f)\in \mathbb{N}$ such that, for all $n\geq n_0$, 
the constants $K_f$, $\alpha_f$ and $\beta_f$ in (i), (ii) and (iii) above can be replaced by $K$, $\alpha$ and $\beta$, respectively.  
\end{theorem}

For a complete proof of this theorem, see \cite{demelovanstrien}. In informal terms, the theorem states three things. 
First, that the post-critical set $P(f)$ 
of an infinitely renormalizable $d$-unimodal map with bounded combinatorics is a Cantor set with bounded geometry. Second, 
that the successive renormalizations of such a map are uniformly bounded in the $C^1$ topology. Third, that the bounds on the 
geometry of the Cantor set and on the $C^1$ norms of the renormalizations become universal at sufficiently deep levels (such bounds 
are called {\it beau\/} by Sullivan in \cite{sullivan} -- see also \cite{demelovanstrien}). 

Further analysis of the non-linearity of renormalizations yields the following consequence of the real bounds.

\begin{corollary}[$C^2$ real bounds]\label{c2realbounds}
 Under the assumptions of Theorem \ref{realbounds}, the successive renormalizations of $f$ are uniformly 
 bounded in the $C^2$ topology, and the bound is {\emph{beau}} in the sense of Sullivan. 
\end{corollary}

The following consequence of the real bounds, namely Lemma \ref{Sn} below, 
is adapted from \cite[Lemma A.5, page 379]{edsonwelington1}, and also from 
\cite[\S 2.1]{dFG}. 

Let $f:I\to I$ be a $C^3$ unimodal map as defined above, and suppose $f$ is infinitely renormalizable 
with renormalization periods bounded by $N$. 
For each $n\geq 1$, let $\mathcal{C}_n=\{\Delta_{i,n}:\, 0\leq i\leq q_n-1\}$ denote the 
collection of renormalization intervals of $f$ 
at level $n$. 
For each $n \geq 1$, we define
$$S_n=\sum_{\mathcal{C}_n\ni \Delta\neq \Delta_{0,n}}\frac{|\Delta|}{d(c,\Delta)}\,,$$
where $d(c,\Delta)$ denotes the Euclidean distance between $\Delta \subset I$ and the critical point $c=0$.
Roughly speaking, the result states that the for each infinitely renormalizable unimodal map of bounded type, the 
sequence $\left\{S_n\right\}_{n \geq 1}$ 
is bounded, and the bound is {\it beau\/} in the sense of Sullivan.

\begin{lemma}\label{Sn} 
There exists a constant $B_1=B_1(N)>0$ with the following property. For each infinitely renormalizable unimodal 
map $f$ of combinatorial type bounded by $N$, there exists $n_1=n_1(f)\in \mathbb{N}$ such that, for all 
$n\geq n_1$, we have $S_n\leq B_1$. 
\end{lemma}

\begin{proof} The desired bound can be proved by a recursive estimate. Note that we can write 
\begin{equation}\label{Sn-eq1}
 S_{n+1} \;=\; \sum_{\mathcal{C}_{n+1}\ni J\subset \Delta_{0,n}\setminus \Delta_{0,n+1}} \frac{|J|}{d(c,J)}
 + \sum_{\mathcal{C}_n\ni \Delta\neq \Delta_{0,n}}\left( \sum_{\mathcal{C}_{n+1}\ni J\subset \Delta} \frac{|J|}{d(c,J)}  \right)
\end{equation}
Now, since $d(c,J) > \frac{1}{2}|\Delta_{0,n+1}|$ for each $J\in \mathcal{C}_{n+1}$, we certainly have
\begin{equation}\label{Sn-eq2}
 \sum_{\mathcal{C}_{n+1}\ni J\subset \Delta_{0,n}\setminus \Delta_{0,n+1}} \frac{|J|}{d(c,J)}\;\leq\; 2\frac{|\Delta_{0,n}|}{|\Delta_{0,n+1}|}\ .
\end{equation}
From the real bounds, Theorem \ref{realbounds}, we know that there exists a constant $0<\alpha=\alpha(N)<1$ such that 
$|\Delta_{0,n}|\leq \alpha^{-1}|\Delta_{0,n+1}|$ for 
all sufficiently large $n$. 
For each $\Delta\in \mathcal{C}_n$, let $J_1,J_2,\ldots, J_{a_n}\in \mathcal{C}_{n+1}$ 
be all the intervals at level $n+1$ which are contained in $\Delta$. Then, again from the real bounds, 
we have $\sum_{i=1}^{a_n}|J_i|\leq \beta|\Delta|$, where 
$0<\beta=\beta(N)<1$, provided the renormalization level $n$ is sufficiently large. Moreover, $d(c,J_i)\geq d(c,\Delta)$ for all $i$. 
Hence we have, for all $n$ sufficiently large, 
\begin{align}\label{Sn-eq3}
 \sum_{\mathcal{C}_n\ni \Delta\neq \Delta_{0,n}}\left( \sum_{\mathcal{C}_{n+1}\ni J\subset \Delta} \frac{|J|}{d(c,J)}  \right) 
 \;&\leq\; \sum_{\mathcal{C}_n\ni \Delta\neq \Delta_{0,n}}\left(\frac{\sum_{\mathcal{C}_{n+1}\ni J\subset \Delta}|J|}{d(c,\Delta)}\right) 
 \nonumber\\
 &\leq\; \beta \sum_{\mathcal{C}_n\ni \Delta\neq \Delta_{0,n}}\frac{|\Delta|}{d(c,\Delta)}\;=\; \beta S_n\ .
\end{align}
Putting  \eqref{Sn-eq2} and \eqref{Sn-eq3} back into \eqref{Sn-eq1}, we deduce that there exists $n_0=n_0(f)$ such that 
$S_{n+1}\leq \beta S_n +\alpha^{-1}$ for all $n\geq n_0$. By induction, it follows that 
$S_{n_0+k}\leq \beta^k S_{n_0} + \alpha^{-1}(1+\beta+\cdots + \beta^{k-1})$ for all $k\geq 0$. Since $\beta<1$, this shows that 
the sequence $(S_n)_{n\geq 1}$ is bounded, and eventually universally so. 
\end{proof}

What we will need is in fact a consequence of this lemma. Given $f$ as in Lemma \ref{Sn}, write for all $n\geq 1$
\begin{equation}
 S_n^*\;=\; \sum_{i=1}^{q_n-1} \frac{|\Delta_{i,n}|^2}{|\Delta_{i+1,n}|} [d(c,\Delta_{i,n})]^{d-2}\, 
\end{equation}
where $d$ is the order of $f$ at the critical point $c$.

\begin{lemma}\label{Snbis} 
There exists a constant $B_2=B_2(N)>0$ with the following property. For each infinitely renormalizable unimodal 
map $f$ of combinatorial type bounded by $N$, there exists $n_2=n_2(f)\in \mathbb{N}$ such that, for all 
$n\geq n_2$, we have $S_n^*\leq B_2$. 
\end{lemma}

\begin{proof}
 Since $f$ has a critical point of order $d$ at $c$, we have $|f'(x)|\geq C_0|x-c|^{d-1}$ for all $x\in I$, for some $C_0=C_0(f)>0$. 
 Replacing, if necessary, $f$ by $R^kf$ for sufficiently large $k$, we can assume that $C_0$ depends in fact only on $N$. 
 Now, for each $i$ we can write $|\Delta_{i+1,n}|/|\Delta_{i,n}|=|f'(x_{i,n})|$ for some $x_{i,n}\in \Delta_{i,n}$, by the mean-value theorem. 
 Hence, using that $|x_{i,n}-c|\geq d(c,\Delta_{i,n})$, we have
 \[
  \frac{|\Delta_{i,n}|^2}{|\Delta_{i+1,n}|} [d(c,\Delta_{i,n})]^{d-2} \;=\;  \frac{|\Delta_{i,n}|}{|f'(x_{i,n})|} [d(c,\Delta_{i,n})]^{d-2} \;\leq \]
\[ \leq   \; C_0^{-1} \frac{|\Delta_{i,n}|}{|x_{i,n}-c|} 
  \;\leq\; C_0^{-1} \frac{|\Delta_{i,n}|}{d(c,\Delta_{i,n})}
 \]
This shows that $S_n^*\leq C_0^{-1}S_n$ for all (sufficiently large) $n$, and the desired result follows from Lemma \ref{Sn}.  
\end{proof}

\section{The $C^2$ bounds for AHPL-maps}\label{sec:c2bounds}

In this section we prove that the successive renormalizations of an infinitely renormalizable AHPL-map of bounded combinatorial type 
are uniformly bounded in the $C^2$ topology, and the bound are {\it beau\/}. Such bounds will be required when we study
the diffeomorphic part of a AHPL-map. 

The main result of this section can be stated more precisely as follows.

\begin{theorem}\label{c2bounds}
 Let $f:U\to V$ be an infinitely renormalizable, $C^3,$ AHPL-map of combinatorial type bounded by $N\in \mathbb{N}$, and 
 let $R^n(f):U_n\to V_n$, $n\geq 1$, be the sequence of renormalizations of $f$. There exists a constant $C_f>0$ such that 
 $\|R^n(f)\|_{C^2(U_n)}\leq C_f$. Moreover, there exist $C=C(N)>0$ and $m=m(f)\in \mathbb{N}$ such that 
 $\|R^n(f)\|_{C^2(U_n)}\leq C$ for all $n\geq m$. 
\end{theorem}

The proof will use the real bounds as formulated in \S \ref{sec:realbounds}, Lemma \ref{Snbis}, as well as the complex bounds established in 
\cite{CvST}, in the form stated in 
\S \ref{sec:complexbounds} below. In fact, the complex bounds are essential even to make sure that the renormalizations $R^nf$ appearing in Theorem 
\ref{c2bounds} are well-defined AHPL-maps (see Remark \ref{rem2} below). 

\subsection{The complex bounds}\label{sec:complexbounds}

We conform with the notation introduced earlier when dealing with infinitely renormalizable interval maps, and 
with AHPL-maps. 

\begin{theorem}[Complex bounds]\label{complexbounds}
Let $f: U\to V$ be an AHPL-map and suppose that $f|_I:I\to I$ is an infinitely renormalizable quadratic unimodal map 
with combinatorial type bounded by $N$. There exist $C=C(N)>1$ and $n_3=n_3(f)\in \mathbb{N}$ such that the following 
statements hold true for all $n\geq n_3$.
\begin{enumerate}
 \item[(i)] For each $0\leq i\leq q_n-1$ there exist Jordan domains $U_{i,n}, V_{i,n}$, with piecewise smooth boundaries 
 and symmetric about the real axis, such that $\Delta_{i,n}\subset U_{i,n}\subset V_{i,n}$, the $V_{i,n}$ are pairwise disjoint,
 and we have the sequence of surjections
 \[
  U_{0,n}\xrightarrow{\;f\,} U_{1,n}\xrightarrow{\;f\,}\cdots \xrightarrow{\;f\,} U_{q_n-1,n}\xrightarrow{\;f\,} V_{0,n} 
  \xrightarrow{\;f\,} V_{1,n}\xrightarrow{\;f\,}
  \cdots\xrightarrow{\;f\,} V_{q_n-1,n} \ .
 \]
 \item[(ii)] For each $0\leq i\leq q_n-1$,  $f_{i,n}=f^{q_n}|_{U_{i,n}}:U_{i,n}\to V_{i,n}$ 
 is a well-defined AHPL-map with critical point at $f^i(c)$.

 \item[(iii)] We have $\mod(V_{i,n}\setminus U_{i,n})\geq C^{-1}$ and $\mathrm{diam}(V_{i,n})\leq C |\Delta_{i,n}|$, for all $0\leq i\leq q_n-1$.
 \item[(iv)] The map $f_{i,n}:U_{i,n}\to V_{i,n}$ has a Stoilow decomposition $f_{i,n}=\phi_{i,n}\circ g_{i,n}$ such 
 that $K(\phi_{i,n})\leq 1+C|\Delta_{0,n}|$, for each $0\leq i\leq q_n-1$. 
\end{enumerate}
\end{theorem}

This theorem is a straightforward consequence of (a special case of) the complex bounds proved in \cite{CvST}. 

\begin{remark}\label{rem2}
 For each $n\geq 1$, consider the linear map $\Lambda_n(z)=|\Delta_{0,n}|z$, and consider the Jordan domais 
 $U_n=\Lambda_n^{-1}(U_{0,n})\subset \mathbb{C}$ 
 and $V_n=\Lambda_n^{-1}(V_{0,n})\subset \mathbb{C}$. Note that $I\subset U_n\subset V_n$. 
 We define $R^nf:U_n\to V_n$ by 
 $R^nf=\Lambda_n^{-1}\circ f_{0,n}\circ \Lambda_n$. This is the $n$-th renormalization of $f$ that appears in the statement 
 of Theorem \ref{c2bounds}. Note that the complex bounds given by this theorem guarantee 
 that $\mathrm{diam}(V_n)\asymp |I|$; in particular, the $C^0$ norms $\|R^nf\|_{C^0(U_n)}$ 
 are uniformly bounded (by a \emph{beau\/} constant). 
\end{remark}

\subsection{Digression on the chain rule}
 
Let $\phi:U\to \mathbb{R}^n$ be a $C^2$ map defined on an open set $U\subset \mathbb{R}^n$. In matrix form, the 
second derivative $D^2\phi$ of $\phi$ is a $n\times n^2$ matrix obtained by the juxtaposition of the Hessian matrices of each of the $n$ 
scalar components of $\phi$. For instance, in dimension $n=2$, the second derivative of a map $\phi=u+iv$ is 
given by the $2\times 4$ matrix 
$D^2\phi= \left[\begin{matrix} u_{xx} & u_{xy} & v_{xx} & v_{xy}\\u_{yx} & u_{yy} & v_{yx} & v_{yy} \end{matrix}\right]$ obtained by adjoining 
the Hessian matrices of the two components of $\phi$. 

Now, if $U,V, W\subseteq \mathbb{R}^n$ are open sets with $V\subseteq W$, and if $\psi:U\to V$ and $\phi:W\to \mathbb{R}^n$ are both $C^2$, then 
the composition $\phi\circ \psi$ is $C^2$, and 
\begin{equation}\label{chainrule}
 D^2(\phi\circ \psi)\;=\; D^2\phi\circ \psi\cdot D\psi\otimes D\psi + D\phi\circ \psi\cdot D^2\psi\ .
\end{equation}
This is the {\it chain rule for the second derivative of a composition in matrix form\/}. 
Here, we denote by $A\otimes B$ the tensor (or Kronecker) product 
of two square matrices $A,B$ of the same size; thus, in our case $D\psi\otimes D\psi$ is a square $n^2\times n^2$ matrix. 
For a proof of this formula, see \cite{manton}.

We will need in fact a formula for the second derivative of an (arbitrarily high) iterate of a given map. 
We formulate it as a lemma.{\footnote{We use the abbreviation $A^{\otimes m}=A\otimes A\otimes \cdots \otimes A$ ($m$ times). }}

\begin{lemma}\label{lemchainrule}
 Let $\phi:U\to \mathbb{R}^n$, $U\subseteq \mathbb{R}^n$ open, be a $C^2$ map. Then for each $k\geq 0$ we have 
 \[
  D^2\phi^k\;=\; D^2\phi\circ \phi^{k-1}\cdot (D\phi^{k-1})^{\otimes 2} + \sum_{j=1}^{k-1} D\phi^{k-j}\circ \phi^j\cdot D^2\phi\circ \phi^{j-1}
  \cdot (D\phi^{j-1})^{\otimes 2}\ ,
 \]
wherever the $k$-th iterate $\phi^k$ is defined.
\end{lemma}

\begin{proof}
 This easily established from \eqref{chainrule} by induction (write $\phi^{k+1}=\phi\circ \phi^k$ for the induction step). 
\end{proof}

Of course, in this paper we will only need these formulas in dimension $n=2$.

\subsection{Proof of Theorem \ref{c2bounds}}

Here we prove our first main result, namely Theorem \ref{c2bounds}. It is natural to divide the proof 
into two steps: in the first step we bound the $C^1$ norms of renormalizations, and in the second step 
we bound the $C^2$ norms. Throughout the proof, we shall successively denote by $C_0,C_1,C_2,\ldots$ positive constants that 
are either absolute or depend only on the constants given by the real and complex bounds. Also, 
in the estimates to follow
we use the {\it operator norm\/} on matrices; to wit, we define 
$\|A\|=\sup_{|v|=1}|Av|$ (here, $|v|$ denotes the euclidean norm of the vector $v$). 
This norm has the advantage of being {\it sub-multiplicative\/}, 
which is to say that $\|AB\|\leq \|A\|\cdot \|B\|$ whenever the product $AB$ is well-defined.
It also satisfies $\|A\otimes B\|\leq \|A\|\cdot \|B\|$.

\subsubsection*{Bounding the $C^1$ norms}
First we prove that the sucessive renormalizations of $f$ are uniformly bounded in the $C^1$ topology, with {\it beau\/} bounds.
We will prove a bit more than what is required. Let us fix $n\in \mathbb{N}$ so large that the real and complex bounds 
given by Theorem \ref{realbounds} and Theorem \ref{complexbounds} hold true for $R^nf$. We divide our argument into 
a series of steps. 
\begin{itemize}
\item[(i)] Replacing $f$ by a sufficiently 
high renormalization we may assume, using Corollary \ref{c2realbounds}, that the $C^2$ norm of $f|_{I}$ 
is bounded by a {\it beau\/} 
constant (that depends only on $N$). In particular, there exists an open complex neighborhood $\mathcal{O}$ 
of the dynamical interval 
$I\subset \mathbb{R}$, with $\mathcal{O}\subseteq U$, such that $\|f\|_{C^2(\mathcal{O})}\leq C_0$. 
And, because the critical point $c$ has order $d$, we may also assume that $\|Df(y)\|\leq C_0|y-c|^{d-1}$ and 
$\|D^2f(y)\|\leq C_0|y-c|^{d-2}$ for all $y\in \mathcal{O}$. 

\item[(ii)] We may assume that $n$ is so large that $V_{i,n}\subset \mathcal{O}$ for all $i$. This is possible because, by the complex bounds 
(Theorem \ref{complexbounds}), $\mathrm{diam}(V_{i,n})\asymp |\Delta_{i,n}|$, and therefore the $V_{i,n}$ shrink exponentially fast
as $n\to \infty$, by the real bounds. 

\item[(iii)] Let $j,k$ be positive integers such that $1\leq j<j+k\leq q_n$. Then for each $x\in \Delta_{j,n}$ we have, by Theorem 
\ref{realbounds},
\begin{equation}\label{c1bounds1}
C_1^{-1}\frac{|\Delta_{j+k,n}|}{|\Delta_{j,n}|}\;\leq\; \|Df^k(x)\|=|(f^k)'(x)|\;\leq\; C_1\frac{|\Delta_{j+k,n}|}{|\Delta_{j,n}|}\ .
\end{equation}
\item[(iv)] Given $x\in \Delta_{j,n}$ and $y\in U_{j,n}$, let us write $x_i=f^i(x)$, $y_i=f^i(y)$ for all $i=0,1,\ldots,k$. 
By step (i), and since $f$ has a critical point at $c$ of order $d$, 
we have 
\begin{equation}\label{c1bounds2}
\frac{ \|Df(x_i) - Df(y_i)\|}{ [d(c,\Delta_{i+j,n})]^{d-2}}  \;\leq\; C_2 |x_i-y_i| \;\leq\; C_3|\Delta_{i+j,n}| \ ,
\end{equation}
for $i=0,1,\ldots, k-1$. From \eqref{c1bounds2} we obviously have
\begin{equation}\label{c1bounds3}
 \|Df(y_i)\|\;\leq\; \|Df(x_i)\|+ C_3|\Delta_{i+j,n}| \cdot  [d(c,\Delta_{i+j,n})]^{d-2}\ ,
\end{equation}
for $i=0,1,\ldots, k-1$.
\item[(v)] By the chain rule for first derivatives, we have 
\begin{equation}\label{c1bounds4}
 \|Df^k(y)\|\;\leq\; \left\|\prod_{i=0}^{k-1} Df(y_i)\right\|\;\leq\; \prod_{i=0}^{k-1} \|Df(y_i)\|\ .
\end{equation}
\item[(vi)] Using \eqref{c1bounds3} and \eqref{c1bounds4} we get
\begin{align}\label{c1bound5}
 \|Df^k(y)\|\;&\leq\; \prod_{i=0}^{k-1} \left(  \|Df(x_i)\|+ C_3|\Delta_{i+j,n}| \cdot  [d(c,\Delta_{i+j,n})]^{d-2}   \right) \nonumber\\
 &\leq\; \prod_{i=0}^{k-1} \|Df(x_i)\| \cdot \prod_{i=0}^{k-1} \left(1+ C_3\frac{|\Delta_{i+j,n}|}{\|Df(x_i)\|}  [d(c,\Delta_{i+j,n})]^{d-2}  \right)\ . 
\end{align}
\item[(vii)] But since $x_i$ is real (and $f$ preserves the real line), we have 
\begin{equation}\label{c1bounds6}
 \prod_{i=0}^{k-1} \|Df(x_i)\|\;=\; \left|\prod_{i=0}^{k-1} f'(x_i)\right|\;=\; \|Df^k(x)\|\ .
\end{equation}
Moreover, for each $i=0,1,\ldots, k$ we have 
\begin{equation}\label{c1bounds7}
 \|Df(x_i)\|\;=\;|f'(x_i)|\;\asymp\; \frac{|\Delta_{i+j+1,n}|}{|\Delta_{i+j,n}|}\ .
\end{equation}
\item[(viii)] Putting \eqref{c1bounds6} and \eqref{c1bounds7} back into \eqref{c1bound5}, we get
\begin{equation}\label{c1bounds8}
 \|Df^k(y)\|\;\leq\; \|Df^k(x)\|\cdot \prod_{i=0}^{k-1} \left(1+ C_4\frac{|\Delta_{i+j,n}|^2}{|\Delta_{i+j+1,n}|}  [d(c,\Delta_{i+j,n})]^{d-2} \right) \ .
\end{equation}
But now, using Lemma \ref{Snbis}, we see that the product in the right-hand side of \eqref{c1bounds8} is uniformly bounded, 
because
\begin{align}\label{c1bounds9} 
 \prod_{i=0}^{k-1}  &\left(1+  C_4\frac{|\Delta_{i+j,n}|^2}{|\Delta_{i+j+1,n}|}  [d(c,\Delta_{i+j,n})]^{d-2} \right)  \nonumber \\
 & \leq\;  \exp\left\{C_4\sum_{i=0}^{k-1} \frac{|\Delta_{i+j,n}|^2}{|\Delta_{i+j+1,n}|}  [d(c,\Delta_{i+j,n})]^{d-2}\right\} \nonumber \\
 &\leq\;\exp\left\{C_4 \sum_{i=1}^{q_n-1} \frac{|\Delta_{i,n}|^2}{|\Delta_{i+1,n}|}  [d(c,\Delta_{i,n})]^{d-2}\right\} \nonumber \\ 
 &=\exp\{C_4S_n^*\}\;\leq\; \exp\{B_2C_4\}\ . 
\end{align}

\item[(ix)] Hence we have proved that $\|Df^k(y)\|\leq C_5\|Df^k(x)\|$, for all $y\in U_{j,n}$ and all $x\in \Delta_{j,n}$. 
From \eqref{c1bounds1}, it follows that 
\begin{equation}\label{c1bounds10}
 \|Df^k(y)\|\leq C_6 \frac{|\Delta_{j+k,n}|}{|\Delta_{j,n}|}\ , \ \ \ \textrm{for all}\ y\in U_{j,n}\ .
\end{equation}
In particular, taking $j=1$ and $k=q_n-1$, we see that the first derivative of the map $f^{q_n-1}|_{U_{1,n}}:U_{1,n}\to V_{0,n}$ 
satisfies{\footnote{Recall that $\Delta_{q_n,n}=\Delta_{0,n}$.}}
\begin{equation}\label{c1bounds11}
 \|Df^{q_n-1}(y)\|\;\leq\; C_6 \frac{|\Delta_{0,n}|}{|\Delta_{1,n}|}\ , \ \ \ \textrm{for all}\ y\in U_{1,n}\ .
\end{equation}
\item[(x)] On the other hand, since $f$ has a critical point of order $d$ at $c=0$, the restriction $f|_{U_{0,n}}:U_{0,n}\to U_{1,n}$ 
satisfies $\|Df(y)\|\leq C_7|y|^{d-1}\leq C_8|\Delta_{0,n}|^{d-1}$ for all $y\in U_{0,n}$ (we are implicitly using step (i) here). Combining 
this fact with step (ix), \eqref{c1bounds11}, and using the chain rule, we see that the first derivative of the 
map 
$$f_{0,n}=f^{q_n}|_{U_{0,n}}= f^{q_n-1}|_{U_{1,n}}\circ f|_{U_{0,n}}:U_{0,n}\to V_{0,n}$$ 
satisfies 
\begin{equation}\label{c1bounds12}
 \|Df^{q_n}(y)\|\;\leq\; C_9 \frac{|\Delta_{0,n}|^d}{|\Delta_{1,n}|}\ , \ \ \ \textrm{for all}\ y\in U_{0,n}\ .
\end{equation}
But, again using that the critical point has order $d$, we have $|\Delta_{1,n}|\asymp |\Delta_{0,n}|^d$. 
Putting this information back in \eqref{c1bounds12}, we deduce that 
$$\|Df_{0,n}\|_{C^0(U_{0,n})}=\|Df^{q_n}\|_{C^0(U_{0,n})}\leq C_{10}\ .$$
Therefore $\|DR^nf\|_{C^0(U_n)}\leq C_{10}$ also, since $R^nf$ is a simply a linearly rescaled copy of $f_{0,n}$.
This shows that the successive renormalizations of $f$ around the critical point are indeed uniformly bounded in the $C^1$ topology, and the  
bounds are {\it beau\/}. 
\end{itemize}

\subsubsection*{Bounding the $C^2$ norms}
We now move to the task of bounding the {\it second derivatives\/} of the renormalizations of $f$. 
Here we use the chain rule for the second derivative of a (long) composition, as given by Lemma \ref{lemchainrule}. 
Once again, we break the proof into a series of (short) steps. 
\begin{itemize}
 \item[(xi)] Since $R^nf=\Lambda_n^{-1}\circ f_{0,n}\circ \Lambda_n$, with $\Lambda_n(z)=|\Delta_{0,n}|z$, we have
 \begin{equation}\label{c2bounds1}
  \|D^2R^nf\|_{C^0(U_n)}\;\leq\; |\Delta_{0,n}|\cdot\|D^2f_{0,n}\|_{C^0(U_{0,n})}\ .
 \end{equation}
We need to bound the norm on the right-hand side of \eqref{c2bounds1}. 
 \item[(xii)] Recall from step (x) the decomposition $f_{0,n}= f^{q_n-1}|_{U_{1,n}}\circ f|_{U_{0,n}}$.
 By the chain rule for second derivatives, for each $y\in U_{0,n}$ we have 
 \begin{equation}\label{c2bounds2}
  D^2f_{0,n}(y)= D^2f^{q_n-1}(f(y)) Df(y)^{\otimes 2} + Df^{q_n-1}(f(y)) D^2f(y)\ .
 \end{equation}
 Note from step (i) that $\|D^2f(y)\|\leq C_0|y-c|^{d-2} \leq C_{11}|\Delta_{0,n}|^{d-2}$. 
 Moreover, applying \eqref{c1bounds11} with $y$ replaced by $f(y)$, we have
 \begin{equation}\label{c2bounds3}
 \|Df^{q_n-1}(f(y))\|\;\leq\; C_6 \frac{|\Delta_{0,n}|}{|\Delta_{1,n}|}\ .
 \end{equation}
 These two estimates combined yield an upper bound for the matrix norm of the second summand in the right-hand side of \eqref{c2bounds2}, namely
 \begin{equation}\label{c2bounds4}
 \|Df^{q_n-1}(f(y))D^2f(y)\|\;\leq\; C_{12} \frac{|\Delta_{0,n}|^{d-1}}{|\Delta_{1,n}|}\ ,
 \end{equation}
where $C_{12}=C_6C_{11}$. 
 \item[(xiii)] It remains to bound the matrix norm of the first summand in the right-hand side of \eqref{c2bounds2}. Applying Lemma 
 \ref{lemchainrule} with $\phi=f$ and $k=q_n-1$ to any point $z\in U_{1,n}$, we have
 \begin{align}\label{c2bounds5}
  D^2f^{q_n-1}(z)=\,&  D^2f(f^{q_n-2}(z)) (Df^{q_n-2}(z))^{\otimes 2}  \\ 
   & + \sum_{j=1}^{q_n-2} Df^{q_n-j-1}(f^j(z)) D^2f(f^{j-1}(z))
  (Df^{j-1}(z))^{\otimes 2}\ ,\nonumber
 \end{align}
 Note that $\|D^2f(f^{q_n-2}(z))\|\leq C_0$, by step (i). 
 Since $f^{j-1}(z)\in U_{j,n}\subset \mathcal{O}$, it also follows from step (i) that 
 \[
 \|D^2f(f^{j-1}(z))\|\leq C_0|f^{j-1}(z)-c|^{d-2} \leq C_{13}[d(c,\Delta_{j,n})]^{d-2}\ ,
 \]
 for all $j\leq q_n$. 
 Using this information in \eqref{c2bounds5}, we get
 \begin{align}\label{c2bounds6}
  \|D^2f^{q_n-1}(z)\|\leq\; & C_0 \|Df^{q_n-2}(z)\|^2   \\ 
  & + C_{13}\sum_{j=1}^{q_n-2} \|Df^{q_n-j-1}(f^j(z))\|\, \|Df^{j-1}(z)\|^2 [d(c,\Delta_{j,n})]^{d-2}
  \ .\nonumber 
 \end{align}
 \item[(xiv)] We now need to bound the norms on the right-hand side of \eqref{c2bounds6}.
 Using the estimate \eqref{c1bounds10} given in step (ix), we have 
 \begin{equation}\label{c2bounds7}
  \|Df^{q_n-2}(z)\|\;\leq\; C_6\frac{|\Delta_{q_n-1,n}|}{|\Delta_{1,n}|}\ ,
 \end{equation}
 as well as 
 \begin{equation}\label{c2bounds8}
  \|Df^{q_n-j-1}(f^j(z))\|\;\leq\; C_6\frac{|\Delta_{q_n-1,n}|}{|\Delta_{j+1,n}|}\ ,
 \end{equation}
and 
\begin{equation}\label{c2bounds9}
  \|Df^{j-1}(z)\|\;\leq\; C_6\frac{|\Delta_{j,n}|}{|\Delta_{1,n}|}\ ,
 \end{equation}
 for all $j\leq q_n-1$. Putting \eqref{c2bounds7}, \eqref{c2bounds8} and \eqref{c2bounds9} back in \eqref{c2bounds6}, we get
\begin{equation}\label{c2bounds10}
 \|D^2f^{q_n-1}(z)\|\leq C_{14}\left[ \frac{|\Delta_{q_n-1,n}|^2}{|\Delta_{1,n}|^2} + 
 \sum_{j=1}^{q_n-2} \frac{|\Delta_{q_n-1,n}|}{|\Delta_{j+1,n}|}\frac{|\Delta_{j,n}|^2}{|\Delta_{1,n}|^2}[d(c,\Delta_{j,n})]^{d-2}\right] \ .
\end{equation}
\item[(xv)] Now we note that $|\Delta_{q_n-1,n}|\asymp |\Delta_{0,n}|$, by the real bounds.{\footnote{We have 
$|\Delta_{0,n}|=|f'(\xi)||\Delta_{q_n-1,n}|$ for some $\xi\in\Delta_{q_n-1,n}$, by the mean value theorem, so 
$|\Delta_{0,n}|\leq C_0|\Delta_{q_n-1,n}|$ (where $C_0$ is the constant of step (i)). An inequality in the opposite 
direction follows from the fact, due to Guckenheimer (and using \cite[Theorem IV.B]{dMvS} if $f$ is not symmetric), that when $f|_{I}$ has negative Schwarzian derivative, 
the renormalization interval containing the critical point is the largest among all renormalization intervals at its level. 
Here we have not assumed the negative Schwarzian property for $f$, but it can be proved that $R^nf|_{I}$ has this property 
for all sufficiently large $n$. For details, see \cite[p.~760]{deFariadeMeloPinto}.}}
Using this information in \eqref{c2bounds10}, we deduce that
\begin{equation}\label{c2bounds11}
 \|D^2f^{q_n-1}(z)\|\;\leq\; C_{15}\frac{|\Delta_{0,n}|}{|\Delta_{1,n}|^2}\left[ |\Delta_{0,n}| + 
 \sum_{j=1}^{q_n-2} \frac{|\Delta_{j,n}|^2}{|\Delta_{j+1,n}|}[d(c,\Delta_{j,n})]^{d-2}\right] \ .
\end{equation}
Applying Lemma \ref{Snbis}, we see that the sum inside square-brackets in the right-hand side of \eqref{c2bounds11} is bounded 
(by a {\it beau constant\/}). Hence we have established that
\begin{equation}\label{c2bounds12}
 \|D^2f^{q_n-1}(z)\|\;\leq\; C_{16}\frac{|\Delta_{0,n}|}{|\Delta_{1,n}|^2}\ .
\end{equation}
\item[(xvi)] Carrying the estimates \eqref{c2bounds4} and \eqref{c2bounds12} back into \eqref{c2bounds2}, we 
deduce that
\begin{equation}\label{c2bounds13}
 \|D^2f_{0,n}(y)\|\;\leq\; C_{17}\left( \frac{|\Delta_{0,n}|^{2d-1}}{|\Delta_{1,n}|^2}  +  \frac{|\Delta_{0,n}|^{d-1}}{|\Delta_{1,n}|}\right)
\end{equation}
This inequality is established for all $y\in U_{0,n}$. 
\item[(xvii)] Finally, combining \eqref{c2bounds13} with \eqref{c2bounds1}, we get
\[
 \|D^2R^nf\|_{C^0(U_n)}\;\leq\; C_{18}\left( \frac{|\Delta_{0,n}|^{2d}}{|\Delta_{1,n}|^2}  +  \frac{|\Delta_{0,n}|^{d}}{|\Delta_{1,n}|}\right)  \ . 
\]
Using once again the fact that $|\Delta_{1,n}|\asymp |\Delta_{0,n}|^{d}$, we deduce at last the inequality $\|D^2R^nf\|_{C^0(U_n)}\leq C_{20}$. 
Hence the successive renormalizations of $f$ are uniformly bounded in the $C^2$ topology, as claimed (and the bounds are {\it beau\/}). 
\end{itemize}
 This finishes the proof of Theorem \ref{c2bounds}. 
 
\begin{remark}\label{stoilowdecomp}
 If we consider the Stoilow decomposition $R^nf=\phi_n\circ g_n$ coming from Theorem \ref{complexbounds}(iv), 
 where $g_n:U_n\to V_n$ is a $d$-to-$1$ holomorphic 
 branched covering map, and $\phi_n:V_n\to V_n$ is an asymptotically holomorphic diffeomorphism, then 
 it is possible to prove, using similar estimates, that $\|\phi_n\|_{C^2(V_n)}$, $\|\phi_n^{-1}\|_{C^2(V_n)}$ and $\|g_n\|_{C^2(U_n)}$ are 
 uniformly bounded, and the bounds are {\it beau\/}. 
\end{remark}

\section{Controlling the distortion of hyperbolic metrics}\label{sec:hyperbolic}

This section is a conformal/quasiconformal intermezzo. Here we develop the distortion tools that will be used 
in the proof of Theorem \ref{control} in \S \ref{sec:recurrence}. 
We believe that these tools  -- especially those concerning the control of 
infinitesimal distortion of hyperbolic metric 
by an asymptotically conformal diffeomorphism, see Proposition \ref{prop52} (for self-maps
of the disk) and Theorem~\ref{exphypmetric} (for other domains) -- 
are of independent interest, and may find applications in other topics of study, such 
as Riemann surface theory. 

\subsection{Comparison of hyperbolic metrics}

We view any non-empty open set $Y\subset \mathbb{C}$ whose complement has at least two points as a {\it hyperbolic Riemann surface\/}. As such, 
$Y$ admits a 
conformal metric of constant negative curvature equal to $-1$, the so-called {\it hyperbolic\/} or {\it Poincar\'e\/} metric 
of $Y$. We denote by $\rho_Y(z)|dz|$ this metric; $\rho_Y(z)$ is the {\it Poincar\'e\/} density at $z\in Y$. 
Integrating this metric along a given rectifiable path $\gamma\subset Y$, we get its {\it hyperbolic length\/} $\ell_Y(\gamma)$. 
This gives rise to a distance $d_Y$ in the usual way: for any given pair of points $z,w\in Y$, we set 
$d_Y(z,w)=\inf \ell_Y(\gamma)$, where $\gamma$ ranges over all paths joining $z$ to $w$ (this will be equal to $\infty$ if 
$z$ and $w$ lie in distinct components of $Y$). We call $d_Y$ the {\it hyperbolic distance\/} 
of $Y$. Accordingly, given $E\subseteq Y$, we denote by $\mathrm{diam}_Y(E)$ the {\it hyperbolic diameter\/} of $E$. 
We also use the following notation: if $z\in Y$ and $v\in T_zY$ is a tangent vector to $Y$ at $z$, then 
we write $|v|_Y$ for the {\it hyperbolic length\/} of $v$ ({\it i.e.,\/} the length of $v$ in the above infinitesimal conformal metric). 

Thus, when $Y$ is the upper or lower half-plane, we have $\rho_Y(z)=|\mathrm{Im}\,z|^{-1}$. When 
$Y$ is the disk of center $z_0\in \mathbb{C}$ and radius $R>0$, we have
\begin{equation}\label{hypmet1}
 \rho_Y(z)\;=\; \frac{2R}{R^2-|z-z_0|^2}\ .
\end{equation}
In the case of the unit disk, one can easily compute that
\[
 d_{\mathbb{D}}(0,z)\;=\; \log{\frac{1+|z|}{1-|z|}}\ .
\]
This yields the following elementary estimate which will be used in \S \ref{sec:control} (see Remark \ref{rem:control}). 

\begin{lemma}\label{disttoroot}
 Let $0\in E\subset \mathbb{D}$ and $0<\delta\leq 1$. If $z\in \mathbb{D}$ is any point whose 
 distance to the boundary of $\mathbb{D}$ is at least $\delta$,
 and if $w\in E$, then
 \[
  d_{\mathbb{D}}(z,w) \leq \mathrm{diam}_{\mathbb{D}}(E) + \log{\frac{1}{\delta}}\ .
 \]

\end{lemma}

The well-known {\it Schwarz lemma\/} states that any holomorphic map $\varphi: X\to Y$ between two hyperbolic Riemann surfaces 
{\it weakly contracts\/} the underlying hyperbolic 
metrics. In other words, $|D\varphi(z)v|_Y \leq |v|_X$ for all $z\in X$ and every tangent vector $v\in T_zX$. If equality 
holds for some $z$ even at a single non-zero vector $v\in T_zX$, then $\varphi$ is a local isometry between (a component of) $X$ and 
(a component of) $Y$. In particular, if $X$ is connected and $X\subset Y$ is a strict inclusion, and $\varphi: X\to Y$ is the inclusion map, then 
$\varphi$ is a strict contraction of the hyperbolic metrics. This leads, in the case when $X$ is connected and $X\subset Y\subset \mathbb{C}$, 
to the strict monotonicity of Poincaré densities: $\rho_X(z) > \rho_Y(z)$ for all $z\in X$. 
The following comparison of Poincar\'e densities follows from monotonicity and will prove useful later.

\begin{lemma}\label{lem:hypmet}
 Let $Y\subseteq \mathbb{C}\setminus \mathbb{R}$ be an non-empty open set, and let $z,w\in Y$ be such that 
 $\mathrm{Re}\,z=\mathrm{Re}\,w$ and $|\mathrm{Im}\,z|\leq |\mathrm{Im}\,w|$. 
 If $z\in D(w,|\mathrm{Im}\,w|)\subseteq Y$, then 
 \begin{equation}\label{hypmet2}
  \frac{1}{|\mathrm{Im}\,z|} \;\leq\; \rho_Y(z)\;\leq\; 
  \frac{1}{|\mathrm{Im}\,z|}\left(1-\frac{1}{2}\frac{|\mathrm{Im}\,z|}{|\mathrm{Im}\,w|}\right)^{-1}\ .
 \end{equation}
\end{lemma}

\begin{proof}
 Look at the inclusions $D(w,|\mathrm{Im}\,w|)\subseteq Y\subseteq \mathbb{C}\setminus \mathbb{R}$ 
 and use \eqref{hypmet1} with $z_0=w$ and $R=|\mathrm{Im}\,w|$. 
\end{proof}

\subsection{Expansion of hyperbolic metric}

It so happens that contraction sometimes leads to expansion. If $\psi: X\to Y$ is a bi-holomorphic map between two hyperbolic Riemann 
surfaces and $X\subset Y$, then the inverse $\psi^{-1}$, viewed as a map from $Y$ {\it into\/} $Y$, can be written 
as a composition of $\psi^{-1}:Y\to X$ with the inclusion $X\subset Y$. The first map in the composition is an isometry between the 
underlying hyperbolic 
metrics, whereas the second map is a contraction. Therefore $\psi$ {\it expands\/} the hyperbolic metric of $Y$. 
In the present paper, we shall need a more quantitative version of this fact. 
This is given by the following lemma due to McMullen (see \cite{McM}).

\begin{lemma}\label{mclemma}
 Let $X, Y$ be hyperbolic Riemann surfaces with $X\subset Y$, and let $\psi:X\to Y$ be holomorphic univalent and onto.
 Then for all $x\in X$ and each tangent vector $v\in T_xX$ we have 
 \begin{equation}\label{McMeq1}
  |D\psi(x)v|_Y \;\geq\; \Phi(s_{X,Y}(x))^{-1}|v|_X\ ,
 \end{equation}
where $s_{X,Y}(x)=d_Y(x,Y\setminus X)$ and $\Phi(\cdot)$ is the 
universal function given by{\footnote{In \cite{McM} McMullen gives $\Phi(s)=2\frac{|t\log t|}{1-t^2}$, where 
$0\leq t<1$ is such that $s=d_{\mathbb{D}}(0,t)$. Eliminating $t$ yields \eqref{McMeq2}.}} 
\begin{equation}\label{McMeq2}
 \Phi(s)\;=\;\sinh{(s)} \log{\left( \frac{1+ e^{-s}}{1- e^{-s}} \right)}\ .
\end{equation}
\end{lemma}

We remark that $\Phi(s)$ is a continuous monotone increasing function with $\Phi(0)=0$ and $\Phi(\infty)=1$. 
Instead of \eqref{McMeq2}, we shall need merely the estimate 
\begin{equation}\label{McMeq3}
 \Phi(s)\;<\; 1 -\frac{1}{3}e^{-2s}\ .
\end{equation}
This estimate is valid provided $s>\frac{1}{2}\log{2}$, and is easily proved with the help of 
Taylor's formula. 

\subsection{Non-linearity and conformal distortion} 
We will also need certain well-known results concerning the geometric distortion of holomorphic univalent maps. 
For details and some background, we recommend \cite[\S 3.8]{dFdM}.

Let $\varphi: V\to \mathbb{C}$ be a holomorphic univalent map defined on an open set $V\subset \mathbb{C}$. 
Then we have Koebe's pointwise estimate on the {\it non-linearity\/} $\varphi''/\varphi'$; to wit, for every $z\in V$ 
we have 
\begin{equation}\label{Koebe0}
 \left|\frac{\varphi''(z)}{\varphi'(z)}\right| \;\leq\; \frac{4}{\mathrm{dist}(z,\partial V)}\ ,
\end{equation}
where $\mathrm{dist}(\cdot,\cdot)$ denotes euclidean distance. This form of pointwise control of the non-linearity 
of $\varphi$ has the following geometric
consequence. Suppose $D\subset V$ is a compact convex subset, and write 
\begin{equation}\label{Koebe1}
 N_{\varphi}(D)\;=\; \mathrm{diam}(D)\,\sup_{z\in D}{\left|\frac{\varphi''(z)}{\varphi'(z)}\right|}\ .
\end{equation}
Then for all $z, w\in D$ we have 
\begin{equation}\label{Koebe2}
 e^{-N_{\varphi}(D)} \leq \left|\frac{\varphi'(z)}{\varphi'(w)}\right| \leq e^{N_{\varphi}(D)}\ .
\end{equation}
When $D$ is not convex, we can still get an estimate like \eqref{Koebe2} by covering $D$ with small disks. 
The following result is by no means the sharpest of its kind, but it will be quite sufficient for our purposes. 

\begin{lemma}\label{koebebounds}
 Let $\varphi:V\to \mathbb{C}$ be holomorphic univalent, and let $W\subset V$ be a non-empty compact connected set. 
 Suppose $M>1$ is such that $1\le \mathrm{diam}(V)\leq M$ and $\mathrm{dist}(\partial V,\partial W)\geq M^{-1}$. 
 Also, let $z_0\in W$ be given. 
 Then the following assertions hold.
 \begin{enumerate}
 \item[(i)] There exists $K_1=K_1(M)>1$ such that, for all $z, w\in W$, we have 
\begin{equation}\label{Koebe3}
 \frac{1}{K_1} \leq \left|\frac{\varphi'(z)}{\varphi'(w)}\right| \leq K_1\ .
\end{equation}
In fact, we can take $K_1=e^{32\pi M^4}$. 
\item[(ii)] There exists $K_2=K_2(M)>0$ such that $\max\left\{\|\varphi'|_{W}\|_{C^0}, \|\varphi''|_{W}\|_{C^0}\right\}\leq K_2 |\varphi'(z_0)|$. 
\end{enumerate}
\end{lemma}
\begin{proof}
 Cover $W$ with a finite number $m$ of non-overlapping closed squares $Q_j$, $1\leq j\leq m$, each $Q_j$ having 
 the same side $\ell=(2\sqrt{2} M)^{-1}$, and take $m$ to be the smallest possible. Then $Q_j\cap W\neq \O$, the diameter 
 of $Q_j$ is $(2M)^{-1}$, and 
 $\mathrm{dist}(Q_j, \partial V)\geq (2M)^{-1}$, for each $1\leq j\leq m$. Since the total area of these squares cannot exceed 
 the area of $V$, which is less than  $\pi M^2$, we see that $m<8\pi M^4$. Moreover, from Koebe's estimate \eqref{Koebe1} we have 
 for each $j$
 \[
  N_\varphi(Q_j)\leq (2M)^{-1}\cdot \frac{4}{(2M)^{-1}} \;=\; 4\ .
 \]
Now, since $W$ is connected, given any pair of points $z,w\in W$, we can join them by a chain of 
 pairwise distinct squares $Q_{j_1}, Q_{j_2}, \ldots, Q_{j_n}$ such that $Q_{j_k}\cap Q_{j_{k+1}}\neq \O$, with $z\in Q_{j_1}$ and 
 $w\in Q_{j_n}$, say. Choose $z_k\in Q_{j_k}\cap Q_{j_{k+1}}$ for $k=1,2,\dots, n-1$, and set $z_0=z, z_n=w$. Use \eqref{Koebe2} to get
 \begin{align*}
  \left|\frac{\varphi'(z)}{\varphi'(w)}\right| &= \prod_{k=0}^{n-1} \left|\frac{\varphi'(z_k)}{\varphi'(z_{k+1})}\right| \\
   & \leq \exp{\left(\sum_{k=1}^{n} N_{\varphi}(Q_{j_k})\right)}\;\leq\; e^{4m}\ .
 \end{align*}
This establishes the upper bound in \eqref{Koebe3}; the lower bound is obtained in the same way, or simply interchanging $z$ and $w$.
Hence assertion (i) is proved. Assertion (ii) follows from assertion (i) and the inequality \eqref{Koebe0}.
\end{proof}

\subsection{Quasiconformality and holomorphic motions}
We need some non-trivial facts from the theory of quasiconformal mappings. Good references for what follows are \cite{A} and \cite{AIM}. 
Given a quasiconformal homeomorphism $\phi$, we write $\mu_\phi(z)$ for the Beltrami form of $\phi$ at $z$,
and $K_\phi(z)=(1+|\mu_\phi(z)|)/(1-|\mu_\phi(z)|)$ for the dilatation of $\phi$ at $z$. We also denote by $K_\phi$ 
the maximal dilatation of $\phi$, namely the supremum 
of $K_\phi(z)$ over all $z$ in the domain of $\phi$. 


\begin{lemma}\label{lem00}
 Let $\phi: \mathbb{C}\to \mathbb{C}$ be a $K$-quasiconformal homeomorphism. Then for each $z\in \mathbb{C}$ 
 and all $r>0$ and $s>0$ we have 
 \[
  \frac{\max_{|\zeta-z|=rs}|\phi(\zeta)-\phi(z)|}{\min_{|\zeta-z|=s}|\phi(\zeta)-\phi(z)|}\;\leq\; 
  e^{\pi K}\max\left\{ r^{K}\,,\, r^{1/K}\right\}\ .
 \]
\end{lemma}

For a proof of this lemma, see \cite[pp.~312-313]{AIM}.

\begin{lemma}\label{lem0}
 Let $\phi:\mathbb{D}\to \mathbb{C}$ be a quasiconformal embedding of the disk with $\phi(0)=0$, and let $0<r<1$. Then 
 the restriction $\phi|_{D(0,r)}$ admits a homeomorphic $K$-quasiconformal extension to the entire plane, 
 where $K = \frac{1+r}{1-r} K_\phi$. 
\end{lemma}

This lemma and its proof can be found in \cite[p.~310]{AIM}. We shall need also the following rather non-trivial 
result due to Slodkowski. Recall that a {\it holomorphic motion\/} of a set $E\subseteq \widehat{\mathbb{C}}$ 
is a map $F: \Delta\times E\to \widehat{\mathbb{C}}$, where $\Delta\subset \mathbb{C}$ is a disk, such that 
(i) for each $z\in E$, the map $t\mapsto F(t,z)$ is holomorphic in $\Delta$; (ii) for each $t\in \Delta$, the map 
$\varphi_t:E\to \widehat{\mathbb{C}}$ given by $\varphi_t(z)=F(t,z)$ is injective; (iii) for a certain $t_0\in \Delta$ 
we have $\varphi_{t_0}(z)=z$ for all $z\in E$. The point $t_0$ is called the {\it base point\/} of the motion.

\begin{theorem}\label{slodkowski}
 Let $F: \Delta\times E\to \widehat{\mathbb{C}}$ be a holomorphic motion of a set $E\subseteq \widehat{\mathbb{C}}$ 
 with base point $t_0\in\Delta$. Then there exists a continuous map $\widehat{F}: \Delta\times \widehat{\mathbb{C}}\to 
 \widehat{\mathbb{C}}$ with the following properties.
 \begin{enumerate}
  \item[(i)] The map $\widehat{F}$ is a holomorphic motion of $\widehat{\mathbb{C}}$ which extends $F$ (in the sense that $\widehat{F}(t,z)=F(t,z)$ 
  for all $z\in E$ and all $t\in \Delta$).
  \item[(ii)] For each $t\in \Delta$, the map $\psi_t(z)=\widehat{F}(t,z)$ is a global $K_t$-quasiconformal homeomorphism 
  with $K_t\leq \exp\{d_{\Delta}(t,t_0)\}$ (where $d_{\Delta}$ denotes the hyperbolic metric of $\Delta$).  
 \end{enumerate}

\end{theorem}

The following lemma contains a well-known result stating that every quasiconformal homeomorphism 
can be embedded in a holomorphic motion (see \cite[ch.~12]{AIM}). It will be used in combination with 
Slodkowski's theorem. 

\begin{lemma}\label{embed}
 Let $\psi: \mathbb{C}\to \mathbb{C}$ be a quasiconformal homeomorphism with $k=\|\mu_\psi\|_\infty\neq 0$, and 
 let $z_0\in \mathbb{C}$ be such that $\psi(z_0)=z_0$.
 \begin{enumerate}
  \item[(i)] There exists a holomorphic motion $\psi_t:\mathbb{C}\to \mathbb{C}$, $t\in \mathbb{D}$, such that $\psi_k=\psi$ and 
  $\psi_t(z_0)=z_0$ for all $t$.
  \item[(ii)] If $0<r_0<1$ and $M>1$ are such that $\psi(D(z_0,r_0))\subseteq D(z_0,Mr_0)$, then for all $0\leq r<1$ and 
  all $t$ with $|t|<\frac{1}{2}$ we have $\psi_t(D(z_0,r))\subseteq D(z_0,R)$, where 
  \begin{equation}\label{Rvalue}
   R\;=\; \frac{2Me^{6\pi}r^{1/3}}{kr_0^2}\ .
  \end{equation}

 \end{enumerate}
\end{lemma}

\begin{proof}
 We may assume that $z_0=0$ (otherwise we simply conjugate $\psi$ by the translation $z\mapsto z-z_0$ and work with the resulting 
 map, which fixes $0$). For each $t\in \mathbb{D}$, let $\varphi_t: \mathbb{C}\to \mathbb{C}$ be the unique solution to 
 the Beltrami equation
 \[
  \overline{\partial}\varphi_t\;=\;\frac{t}{k}\mu_\psi \partial\varphi_t\ ,
 \]
normalized so that $\varphi_t$ fixes $0,1$ and $\infty$. Define $\psi_t: \mathbb{C}\to \mathbb{C}$ by the formula
\begin{equation}\label{extholmotion}
 \psi_t(\zeta)\;=\;\left[ 1+ \frac{t}{k}\left( \psi(1)-1  \right)  \right] \varphi_t(\zeta)\ .
\end{equation}
Note that $\psi_t(0)=0$ for all $t$. Also, for $t=k$, we have $\psi_k(\zeta)=\psi(1)\varphi_k(\zeta)$, so $\psi_k(1)=\psi(1)$. Since the 
Beltrami form of $\psi_k$ is the same as the Beltrami form of $\varphi_k$, which is $\mu_\psi$, it follows from uniqueness
of normalized solutions to the Beltrami equation that $\psi_k=\psi$. This proves (i). 

Applying Lemma \ref{lem00} to $\phi=\varphi_t$, $z=0$ and $s=1$, we see that for all $0<r<1$
\[
 \max_{|\zeta|=r}|\varphi_t(\zeta)|\;\leq\; e^{\pi K_t}r^{1/K_t}\ ,
\]
where $K_t$ is the maximal dilatation of $\varphi_t$, which satisfies
\[
 K_t\;\leq\; \frac{1+|t|}{1-|t|}\ .
\]
In particular, since $K_t<3$ for all $t$ with $|t|<\frac{1}{2}$, we have
\begin{equation}\label{varphibound}
 \varphi_t(D(0,r))\;\subseteq\; D(0,e^{3\pi}r^{1/3}) 
\end{equation}
Let us now estimate the scaling factor multiplying $\varphi_t(\zeta)$ on the right-hand side of 
\eqref{extholmotion}. Applying Lemma \ref{lem00} with $\phi=\psi$, $z=0$, $s=r_0$ and $r=r_0^{-1}$, 
and taking onto account that the maximal dlatation of $\psi$ is less than $3$, we get
\begin{align*}
 \max_{|\zeta|=1}{|\psi(\zeta)|}\;&\leq\; e^{3\pi}\frac{1}{r_0^3}\min_{|\zeta|=r_0}{|\psi(\zeta)|} \\
 &\leq\; e^{3\pi}\frac{1}{r_0^3}(Mr_0)\;=\; \frac{Me^{3\pi}}{r_0^2}\ .
\end{align*}
In particular, $|\psi(1)-1|\leq 2Me^{3\pi}r_0^{-2}$, and therefore
\[
 \left|1+\frac{t}{k}\left(\psi(1)-1\right)\right|\;\leq\; \frac{2Me^{3\pi}}{kr_0^2}
\]
for all $t$ with $|t|\leq \frac{1}{2}$. 
Combining this fact with \eqref{varphibound}, it follows that for all such $t$ we have 
\[
 \max_{|\zeta|\leq r}|\psi_t(\zeta)|\;\leq\; \frac{2Me^{6\pi}r^{1/3}}{kr_0^2}\ .
\]
Therefore $\psi_t(D(0,r)) \subseteq D(0,R)$ for all $t$ with $|t|\leq \frac{1}{2}$ and all $0<r<1$, 
where $R$ is given by \eqref{Rvalue}. 
This proves (ii). 
\end{proof}

\subsection{Quasi-isometry estimates for almost conformal maps}

Our goal in this subsection is to make more precise a somewhat vague but intuitive assertion, namely that 
if a self-map of a hyperbolic domain (or Riemann surface) is almost conformal, then it is an almost isometry of 
the hyperbolic metric. For the sake of the dynamical applications we have in mind, what is needed is an infinitesimal version 
of this statement.

The desired infinitesimal quasi-isometry property will be presented in two versions.
In the first version we deal with the case when the quasiconformal map has small dilatation everywhere, and the 
quasi-isometry bounds we get are in terms of this global small dilatation. In the second version we
deal with the situation when the map is $K$-quasiconformal (with $K$ not necessarily small) but the 
quasi-isometry bounds we get are local, near any point $z\in \mathbb{D}$ where the dilatation is bounded by some 
fixed power of the distance 
between $z$ and $\partial\mathbb{D}$. This last version is precisely what we need when studying the metric distortion 
properties of maps which are asymptotically holomorphic. Both versions are first established for quasiconformal diffeomorphisms of the unit disk, 
but at the end of this subsection we show how to transfer these results to the kind of simply-connected regions that matter to us. 

First, let us introduce some notation. 
We denote by $\rho_{\mathbb{D}}(z)=2(1-|z|^2)^{-1}$ the Poincar\'e density of the unit disk, as before. 
We also denote by $\Delta_z\subset \mathbb{D}$ the closed euclidean disk $\{\zeta:\;|\zeta -z|\leq \frac{1}{2}(1-|z|)\}$. 
Given a $C^2$ map $\phi:{\mathbb{D}}\to {\mathbb{D}}$, we denote by $m_\phi(z)$ the $C^2$ norm of 
$\phi|_{\Delta_z}$. We write $J_\phi(z)=\det{D\phi(z)}$ for the euclidean Jacobian of $\phi$ at $z$, and 
\[
 J_{\phi}^{h}(z)\;=\; J_{\phi}(z)\left(\frac{\rho_{\mathbb{D}}(\phi(z))}{\rho_{\mathbb{D}}(z)}\right)^2
\]
for the {\it hyperbolic\/} Jacobian of $\phi$ at $z$. 

\begin{proposition} \label{thm1}
For each $0<\theta< 1$, there exists a universal continuous function $A_{\theta}:(1,\infty)\times \mathbb{R}^+\to \mathbb{R}^+$ for which the following 
holds. Let $0<\epsilon<1$ and $\alpha>1$ be given, and suppose $\phi: {\mathbb{D}}\to {\mathbb{D}}$ is 
a $C^2$ quasiconformal diffeomorphism with $K_\phi\leq 1+\epsilon$. If $z\in \mathbb{D}$ is such that 
\begin{equation}\label{alphainverse}
 \alpha^{-1}\;\leq\; \frac{\rho_{\mathbb{D}}(\phi(z))}{\rho_{\mathbb{D}}(z)}\;\leq\; \alpha\ ,
\end{equation}
then 
\[
 J_{\phi}^{h}(z)\;\leq\; 1+ A_{\theta}(\alpha,m_\phi(z))\epsilon^{1-\theta}\ .
\]
\end{proposition}

The proof, given later in this subsection, will use the following three lemmas. 

\begin{lemma}\label{lem1}
 Let $z\in \mathbb{D}$ and let $0< r<1-|z|$. Then
 \begin{equation}\label{eq0}
  \mathrm{mod}(\mathbb{D}\setminus D(z,r))\;\leq\; \log{\left(\frac{1-|z|^2+|z|r}{r}  \right)}\ .
 \end{equation}
\end{lemma}

\begin{proof}
 We may assume that $z$ is real and non-negative, say $z=x\in [0,1)$. Let $\varphi\in \mathrm{Aut}(\mathbb{D})$ be given by 
 \[
  \varphi(\zeta)\;=\; \frac{\zeta -x}{1-x\zeta}\ ,
 \]
and define
\[
 \alpha\;=\; \varphi(x-r)\;=\;\frac{-r}{1-x^2+rx}\ \ ;\ \ \beta\;=\; \varphi(x+r)\;=\;\frac{r}{1-x^2-rx}\ .
\]
Then $D_r'=\varphi(D(x,r))$ is a disk with diameter $(\alpha,\beta)\subset (-1,1)$. Since $|\alpha|\leq \beta$, we see 
that $D_r'\supseteq D(0, |\alpha|)$. Therefore
\begin{align*}
 \mathrm{mod}(\mathbb{D}\setminus D(x,r))\;&=\; \mathrm{mod}(\mathbb{D}\setminus D_r') \\
  &\leq\; \mathrm{mod}(\mathbb{D}\setminus D(0,|\alpha|))\;=\; \log{\frac{1}{|\alpha|}}\\
  &=\;\log{\frac{1-x^2+rx}{r}}\ ,
\end{align*}
and this finishes the proof.
\end{proof}

\begin{remark}\label{rem1}
 It follows from \eqref{eq0} that $\mathrm{mod}(\mathbb{D}\setminus D(z,r))\leq \log{\left(\dfrac{2}{r}\right)}$. This estimate 
 will be useful when $r$ is small compared to the distance from $z$ to $\partial \mathbb{D}$. If $r=\frac{1}{2}\delta (1-|z|)$ with 
 $0<\delta\leq 1$, then an easy manipulation of the right-hand side of \eqref{eq0} yields the estimate 
 $\mathrm{mod}(\mathbb{D}\setminus D(z,r))\leq \log{\left(\dfrac{5}{\delta}\right)}$. 
 This remark will be used in the proof of Lemma \ref{lem3} below. 
\end{remark}

\begin{lemma}\label{lem2}
 Let $\alpha>1$ and suppose $z,w\in \mathbb{D}$ are such that 
 \begin{equation}\label{eq1}
  \alpha^{-1}\;\leq\; \frac{\rho_{\mathbb{D}}(z)}{\rho_{\mathbb{D}}(w)}\;\leq\; \alpha\ ,
 \end{equation}
 Then there exists $\psi\in \mathrm{Aut}(\mathbb{D})$ with $\psi(z)=w$ such that the following inequalities 
 hold for all $\zeta\in \Delta_z$:
 \begin{multicols}{2}
 \begin{enumerate}
  \item[(i)] $\displaystyle{\frac{1}{2\alpha}\;\leq\; |\psi'(\zeta)|\leq 4\alpha^2}$\ ;
  \item[(ii)] $\displaystyle{|\psi''(\zeta)| \;\leq\; 16\alpha^3}$.
 \end{enumerate}
\end{multicols}
\end{lemma}

\begin{proof}
Write $a=|z|$ and $b=|w|$, so that $0\leq a,b<1$. We have $1-a^2=\rho_{\mathbb{D}}(z)^{-1}$ and 
$1-b^2=\rho_{\mathbb{D}}(w)^{-1}$, so \eqref{eq1} tells us that
\begin{equation}\label{eq2}
 \alpha^{-1}\;\leq\; \frac{1-a^2}{1-b^2}\;\leq\; \alpha\ .
\end{equation}
Let $\varphi\in \mathrm{Aut}(\mathbb{D})$ be the hyperbolic translation 
with axis $(-1,1)\subset \mathbb{D}$ such that $\varphi(a)=b$. Then
\[
 \varphi(\zeta)\;=\; \frac{\zeta-c}{1-c\zeta} \ ,
\]
where $c=(a-b)/(1-ab)\in (-1,1)$, as a simple calculation shows. Moreover, we have 
\begin{equation}\label{eq3}
 \varphi'(\zeta)\;=\;\frac{1-c^2}{(1-c\zeta)^2}\ ,
\end{equation}
as well as 
\begin{equation}\label{eq4}
 \varphi''(\zeta)\;=\;\frac{2c(1-c^2)}{(1-c\zeta)^3}\ ,
\end{equation}
Since $1-c^2\;=\; (1-a^2)(1-b^2)/(1-ab)^2$, and since $\min\{1-a^2,1-b^2\}\leq 1-ab\leq \max\{1-a^2,1-b^2\}$, it follows 
from \eqref{eq2} that 
\begin{equation}\label{eq5}
 \alpha^{-1}\;\leq\; 1-c^2\;\leq\; 1\ .
\end{equation}
Now, if $\zeta\in \Delta_a$, then $|\zeta|\leq (1+a)/2$. Hence
\[
 |1-c\zeta|\;\geq\; 1-|c|\left(\frac{1+a}{2}\right)\;=\; \frac{1-|c|}{2}+\frac{1-|c|a}{2}\;>\; \frac{1-|c|a}{2}\ .
\]
Here, there are two cases to consider. If $a\geq b$, then $c\geq 0$ and 
$1-|c|a=1-ca=(1-a^2)/(1-ab)$, so from \eqref{eq2} we deduce that $1-|c|a\geq {\alpha^{-1}}$.
If however $a< b$, then $c< 0$, and in this case we see that
\[
 1-|c|a\;=\; \frac{1-b^2+(b-a)^2}{1-ab}\;>\; \frac{1-b^2}{1-a^2}\;\geq\; \alpha^{-1}\ ,
\]
where once again we have used \eqref{eq2}. Thus, in either case we have
\begin{equation}\label{eq6}
 \frac{1}{2\alpha}\;\leq\;|1-c\zeta|\;<\;2 \ , \ \ \textrm{for all}\ \zeta\in \Delta_a\ .
\end{equation}
Using both \eqref{eq5} and \eqref{eq6} in \eqref{eq3} and\eqref{eq4}, we easily arrive at inequalities (i) and (ii) with 
$\varphi$ replacing $\psi$ (and $\Delta_a$ replacing $\Delta_z$). Finally, we define $\psi=R_b\circ \varphi\circ R_a$, 
where $R_a$ is the rigid rotation around 
$0$ with $R_a(z)=a$, and $R_b$ is the rigid rotation around $0$ with $R_b(b)=w$. Then $\psi(z)=w$, and since $R_a,R_b$ are euclidean isometries 
and $R_a(\Delta_z)=\Delta_a$,
the inequalities (i) and (ii) for $\psi$ follow from the corresponding inequalities for $\varphi$.  
\end{proof}

For our final lemma, we introduce further notation. Given a $C^2$ map $\phi:\mathbb{D}\to \mathbb{D}$, a point $z\in \mathbb{D}$ 
and $0<\delta\leq 1$, we denote 
by $m_\phi(z,\delta)$ the $C^2$ norm of the restriction of $\phi$ to the disk $\{\zeta:\,|\zeta-z|\leq \delta r_z\}$, 
where $r_z=\frac{1}{2}(1-|z|)$. In particular, $m_\phi(z,1)=m_\phi(z)$. 

\begin{lemma}\label{lem3}
 For each $0<\theta<1$ there exists a universal, continuous monotone function $B_{\theta}:\mathbb{R}^+\to \mathbb{R}^+$ such that the following holds. 
 Given $0<\epsilon<1$, let $\phi: {\mathbb{D}}\to {\mathbb{D}}$ be 
a $C^2$ quasiconformal diffeomorphism with $K_\phi\leq 1+\epsilon$, and suppose that $z\in \mathbb{D}$ is a fixed point of $\phi$. 
Then for each $0<\delta\leq 1$ we have
\begin{equation}\label{eq7}
 J_\phi^h(z)\;\leq\; 1+ B_{\theta}\left(\frac{m_\phi(z,\delta)}{\delta}\right)\epsilon^{1-\theta}\ .
\end{equation}

\end{lemma}

\begin{proof}
The basic geometric idea behind the proof is to use macroscopic estimates on the moduli of certain annuli in order to bound 
 a microscopic quantity, namely the hyperbolic Jacobian at $z$.
 Rotating the coordinate axes if necessary, we may also assume that $D\phi(z)=S\cdot T$, where 
 $S=\rho I=\left(\begin{matrix} \rho & 0\\0 & \rho\end{matrix}\right)$, for some $\rho>0$, and 
 $T=\left(\begin{matrix} \lambda & b\\ 0 & \lambda^{-1}\end{matrix}\right)$, where $\lambda\geq 1$ and $b\in \mathbb{R}$. 
 Here we obviously have $\rho^2=\det{D\phi(z)}=J_\phi(z)=J_\phi^h(z)$. 
 We shall prove the lemma only in the case when $b=0$ and $\lambda>1$. The cases when $b\neq 0$ and/or $\lambda=1$ 
 are similarly handled. Note that the linear map $D\phi(z)$ maps the circle of radius $1$ about the origin onto an ellipse with 
 major axis $\rho\lambda$ and minor axis $\rho/\lambda$. 
 Since $\phi$ is $(1+\epsilon)$-qc, we have $\lambda^2\leq 1+\epsilon$. In what follows, we assume that
 $\rho>\lambda +\epsilon$, as otherwise $\rho^2\leq (\lambda+\epsilon)^2\leq 1+6\epsilon$ and there is nothing to prove. 
 
 If $\zeta$ is such that $|\zeta-z|\leq \delta r_z$ we can write, using Taylor's formula and the fact that $\phi(z)=z$, 
 \begin{equation}\label{eq8}
 \phi(\zeta)\;=\;z+D\phi(z)\cdot (\zeta -z) + R_\phi(\zeta)\ ,
 \end{equation}
 where the remainder $R_\phi(\zeta)$ satisfies 
 $|R_\phi(\zeta)|\leq C|\zeta-z|^2$, with $C=C_0 m_\phi(z,\delta)>0$ (and $C_0>0$ an absolute constant). 
 Let us choose $0< r\leq \delta r_z$ so small that 
 \begin{equation}\label{eq9}
  \frac{\rho}{\lambda}r -Cr^2 \;>\; \frac{\rho}{\lambda+\epsilon}r\ .
 \end{equation}
For definiteness, we take 
\begin{equation}\label{eq10}
r=\min\left\{\delta r_z\;,\;\frac{\rho\epsilon}{C\lambda^2(\lambda+\epsilon)}\right\} \ .
\end{equation}
Then \eqref{eq8} and \eqref{eq9} tell us that $\phi$ maps the disk $D(z,r)$ onto a Jordan domain $V_r$ which contains that 
disk and also the round annulus $\Omega=\{\zeta:\,r<|\zeta-z|<\frac{\rho}{\lambda+\epsilon}r\}$. Setting $\Omega_0=V_r\setminus D(z,r)$, 
we have $\Omega_0\supseteq \Omega$, and so 
\begin{equation}\label{eq11}
 \mathrm{mod}(\Omega_0)\;\geq\; \mathrm{mod}(\Omega)\;=\; \log{\left( \frac{\rho}{\lambda+\epsilon} \right)}\ .
\end{equation}
Consider the images of $\Omega_0$ under the forward iterates of $\phi$, {\it i.e.\/}, $\Omega_n=\phi^n(\Omega_0)$, $n\geq 0$. The annuli $\Omega_n$ 
are pairwise disjoint, and $\cup_{n=0}^{\infty} \Omega_n \subset \mathbb{D}\setminus D(z,r)$. By sub-additivity of the modulus, we have
\begin{equation}\label{eq12}
 \sum_{n=0}^{\infty} \mathrm{mod}(\Omega_n) \;\leq\;\mu_r \;=\; \mathrm{mod}(\mathbb{D}\setminus D(z,r))\ .
\end{equation}
Now, since $\phi$ is $(1+\epsilon)$-qc, we know that $\phi^n$ is $(1+\epsilon)^n$-qc, and therefore 
\begin{equation}\label{eq13}
 \mathrm{mod}(\Omega_n)\;\geq\; \frac{\mathrm{mod}(\Omega_0)}{(1+\epsilon)^n}\ .
\end{equation}
Putting together \eqref{eq11}, \eqref{eq12} and \eqref{eq13}, we get
\begin{equation}\label{eq14}
 \log{\left( \frac{\rho}{\lambda+\epsilon} \right)} \sum_{n=0}^{\infty} \frac{1}{(1+\epsilon)^n} \;\leq\;\mu_r \ .
\end{equation}
Applying Lemma \ref{lem1} and Remark \ref{rem1} to our $r$ as defined in \eqref{eq10}, we see that 

\begin{equation}
\mu_r\;\leq\;\left\{\begin{array}{ll}
\log{\left(\dfrac{5}{\delta} \right)}\ , & \ {\textrm{when}}\ r=\delta r_z \ ;\\
 {} & {}\\
\log{\left(\dfrac{2C\lambda^2(\lambda+\epsilon)}{\rho\epsilon} \right)}\ , &\ {\textrm{when}}\ r=\dfrac{\rho\epsilon}{C\lambda^2(\lambda+\epsilon)}\ .
\end{array}\right.
\end{equation}
Regardless of which of the two cases occur, we certainly have
\begin{equation}\label{eq15}
 \mu_r\;\leq\; \log{\left(\frac{10C\lambda^2(\lambda+\epsilon)}{\delta\rho\epsilon}\right)}\;<\; \log{\left(\frac{60C}{\delta\epsilon}\right)}\ ,
\end{equation}
where in the last step we have used that $\lambda^2(\lambda+\epsilon)<6$ and $\rho>1$. 
Combining \eqref{eq14} and \eqref{eq15}, we deduce that
\begin{align}\label{eq16}
 \log{\left( \frac{\rho}{\lambda+\epsilon} \right)}\;&\leq\; \frac{\epsilon}{1+\epsilon}\log{\left(\frac{60C}{\delta\epsilon}\right)} \nonumber \\
 & <\; \epsilon \log{\left(\frac{60C}{\delta}\right)} + \epsilon \log{\frac{1}{\epsilon}}
\end{align}
Since $0<\epsilon<1$, we have $\epsilon<\epsilon^{1-\theta}$ and $\epsilon^{\theta}\log{\dfrac{1}{\epsilon}}\leq (\theta e)^{-1}$. Using these facts 
in \eqref{eq16}, we get
\begin{align}
 \rho\;&\leq\; (\lambda+\epsilon)\exp\left\{\left(\frac{1}{\theta e}+ \log{\frac{60C}{\delta}}\right)\epsilon^{1-\theta}\right\}\\ 
  &\leq\; 1+ \left(2+180e^{{1}/{\theta e}}\frac{C}{\delta}\right)\epsilon^{1-\theta}\ , 
\end{align}
where we have used that $\lambda+\epsilon\leq 1+2\epsilon$. From this, and the fact that $C=C_0 m_\phi(z,\delta)$, it readily follows that
\[
 J_\phi^h(z)\;=\; \rho^2\leq 1+ 3\left(2+180e^{{1}/{\theta e}}C_0\frac{m_\phi(z,\delta)}{\delta}\right)^2 \epsilon^{1-\theta} \ .
\]
This proves \eqref{eq7}, provided we take $B_\theta(t)=3\left(2+ 180e^{{1}/{\theta e}}C_0t\right)^2$. 
\end{proof}

We are now ready for the proof of the first main result of this subsection.

\begin{proof}[Proof of Proposition \ref{thm1}]
 The idea, of course, is to reduce the required estimate to the case treated in Lemma \ref{lem3}. Let $\psi\in \mathrm{Aut}(\mathbb{D})$ be the conformal 
 automorphism given by Lemma \ref{lem2}, with $\psi(z)=w=\phi(z)$. Then the diffeomorphism $F=\psi^{-1}\circ \phi:\mathbb{D}\to \mathbb{D}$ 
 has a fixed point at $z$. Since $\psi^{-1}$ is an isometry of the hyperbolic metric, we certainly have $J_F^h(z)=J_\phi^h(z)$. We would like to 
 estimate $J_F^h(z)$ using Lemma \ref{lem3}. For this, we need an estimate on the $C^2$ norm of the composition $\psi^{-1}\circ \phi$ in a suitable disk 
 around $z$. By Koebe's one-quarter theorem, $\psi(\Delta_z)$ contains the disk
 \[
  D\;=\;\left\{ \zeta:\; |\zeta -w|<\frac{1}{4}|\psi'(z)|\cdot r_z    \right\}\ .
 \]
Since we know from Lemma \ref{lem2}(i) that $|\psi'(z)|\geq (2\alpha)^{-1}$, it follows that $\psi(\Delta_z)\supset D(w,R)$, where 
$R=r_z/8\alpha$. Now let us define 
\[\delta= \frac{1}{8\alpha m_\phi(z)}\ \ \ \textrm{and}\ \ \  M=\sup_{\zeta\in \Delta_z}|D\phi(\zeta)| \leq m_\phi(z)\ .\] 
Then we have $\phi(D(z,\delta r_z))\subset D(w,M\delta r_z)\subseteq D(w,R)\subset \psi(\Delta_z)$. We can now estimate 
the $C^2$ norm of $F$ restricted to the disk $D(z,\delta r_z)$, {\it i.e.\/} we can estimate $m_F(z,\delta)$, with the help 
of Lemma \ref{lem2}. We do this by means of the following two steps. 
\begin{enumerate}
\item[(i)] By the chain rule for first derivatives, we have $DF=D\psi^{-1}\circ \phi\cdot D\phi$. Since $\psi^{-1}$ is holomorphic, 
for each $\zeta\in D(z,\delta r_z)$ we have
\begin{equation}\label{eq17}
\|D\psi^{-1}(\phi(\zeta))\|\leq |(\psi^{-1})'(\phi(\zeta))| = |\psi'(\psi^{-1}\circ\phi(\zeta))|^{-1} \leq 2\alpha\ .
\end{equation}
Hence the $C^0$ norm of $DF$ in $D(z,\delta r_z)$ is bounded by $2\alpha m_\phi(z)$.
\item[(ii)] By the chain rule for second derivatives, we have
\begin{equation}\label{eq18}
 D^{2}F \;=\; (D^2\psi^{-1}\circ \phi)\cdot (D\phi \otimes D\phi) + D\psi^{-1}\circ \phi\cdot D^2\phi\ .
\end{equation}
Again, since $\psi^{-1}$ is holomorphic, a simple calculation shows that
\[
 (\psi^{-1})''\;=\; -\frac{\psi''\circ \psi^{-1}}{(\psi'\circ \psi^{-1})^3}\ .
\]
Therefore, for each $\zeta\in D(z,\delta r_z)$ we have, with the help of Lemma \ref{lem2}, 
\begin{equation}\label{eq19}
 \|D^2\psi^{-1}(\phi(\zeta))\|\;\leq\; |(\psi^{-1})''(\phi(z))|\;\leq\; 128\alpha^6\ .
\end{equation}
Using \eqref{eq17}, \eqref{eq19} and the fact that $\|D\phi\otimes D\phi\|\leq \|D\phi\|^2$ in \eqref{eq18}, we deduce that 
the $C^0$ norm of $D^2F$ in the disk $D(z,\delta r_z)$ is bounded by $(128\alpha^6+2\alpha)m_\phi(z)<130\alpha^6m_\phi(z)$. 
 \end{enumerate}
 
From steps (i) and (ii) above we deduce that $m_F(z,\delta)\leq 130\alpha^6m_\phi(z)$. Therefore, applying Lemma \ref{lem3}
for $F$ yields
\[
 J_\phi^h(z)\;=\;J_F^h(z)\;\leq\; 1+ B_{\theta}\left(\frac{m_F(z,\delta)}{\delta}\right)\epsilon^{1-\theta} 
 \leq 1+ B_{\theta}\left(1040\alpha^7(m_\phi(z))^2\right)\epsilon^{1-\theta} \ .
\] 
This completes the proof of our theorem, provided we take $A_{\theta}(s,t)=B_{\theta}(1040s^7t^2)$.

\end{proof}

\begin{proposition}\label{prop52}
 For each $0<\theta< 1$, there exists a universal continuous function 
 $C_{\theta}:(1,\infty)\times (1,\infty)\times \mathbb{R}^+\times \mathbb{R}^+\to \mathbb{R}^+$ 
 for which the following 
holds. Let $\alpha>1$ and $\beta>1$ be given, and suppose $\phi: {\mathbb{D}}\to {\mathbb{D}}$ is 
a $C^2$ quasiconformal diffeomorphism. If $z\in \mathbb{D}$ is such that 
\begin{equation}\label{alphaishere}
 \alpha^{-1}\;\leq\; \frac{\rho_{\mathbb{D}}(\phi(z))}{\rho_{\mathbb{D}}(z)}\;\leq\; \alpha\ ,
\end{equation}
and 
\begin{equation}\label{localdilat}
 \sup_{\zeta\in \Delta_z} |\mu_\phi(\zeta)|\leq b_0(1-|z|)^{\beta}\ ,
\end{equation}
then 
\begin{equation}\label{hyperjacobian}
 J_{\phi}^{h}(z)\;\leq\; 1+ C_{\theta}(\alpha,\beta,b_0,m_\phi(z))(1-|z|)^{\beta(1-\theta)}\ .
\end{equation}
\end{proposition}

\begin{proof}
We present the proof of the required estimate under the additional assumption that $z$ is a fixed-point of $\phi$. The general case can be reduced to 
this one by post-composing $\phi$ with a suitable conformal automorphism of the unit disk, and proceeding just as in the 
proof of Proposition \ref{thm1}, {\it mutatis mutandis\/}. For the sake of clarity of exposition, we divide the proof into a series of steps.
\begin{enumerate}
 \item[(i)] First we introduce some notation. Throughout the proof we denote by $c_0,c_1,\ldots$ positive constants that are either absolute or depend on 
 the given constants $\alpha, \beta, \break b_0, M$, where $M=m_\phi(z)$. Let us write $\epsilon=b_0(1-|z|)^{\beta}= (b_02^{\beta})r_z^{\beta}$. 
 Also, let $k_0=\sup_{\zeta\in \Delta_z}|\mu_\phi(\zeta)|\leq \epsilon$, and set $r_0=\epsilon r_z$. We may assume without loss of generality 
 that $\epsilon$ is small, say $\epsilon<1/{32}$.
 \item[(ii)] The restricted map $\phi|_{\Delta_z}: \Delta_z\to \mathbb{D}$ is a $\frac{1+k_0}{1-k_0}$-quasiconformal embedding. By Lemma \ref{lem0}, 
 the further restriction $\phi|_{D(z,r_0)}$ can be extended to a global quasiconformal homeomorphism $\psi:\mathbb{C}\to \mathbb{C}$ with 
 $k=\|\mu_\psi\|_{\infty}$ satisfying 
 \[
  \frac{1+k}{1-k}\leq \frac{1+\epsilon}{1-\epsilon}\cdot \frac{1+k_0}{1-k_0} \leq \left(\frac{1+\epsilon}{1-\epsilon}\right)^2\ .
 \]
 \item[(iii)] In particular, $k\leq 16\epsilon<\frac{1}{2}$ (by our assumption on $\epsilon$ in (i)). 
 We may assume that $k\neq 0$ (if this is not the case, it is easy to perturb 
 $\psi$ slightly in a neighborhood of infinity). By Lemma \ref{embed}(i), there exists a global holomorphic motion $\psi_t : \mathbb{C}\to \mathbb{C}$ 
 with $\psi_k=\psi$ and $\psi_t(z)=z$ for all $t\in \mathbb{D}$. Now choose $r_1>0$ so small that 
 \[
  R\;=\; \frac{2Me^{6\pi}}{k_0r_0^2}\cdot r_1^{1/3}\;<\; r_z\ .
 \]
For definiteness, take $r_1=c_1k^3r_z^{6\beta+9}$, where $c_1=b_0^6/(M^3e^{18\pi})$. Then, by Lemma \ref{embed}(ii), we have 
$\psi_t(D(z,r_1))\subset D(z,R)$ for all $t$ with $|t|<\frac{1}{2}$ (note that this includes the time $t=k$). 
\item[(iv)] We may now define, for each $t\in D(0,\frac{1}{2})$, the map $\widetilde{\psi_t}: D(z,r_1)\cup (\mathbb{C}\setminus \mathbb{D})\to \mathbb{C}$ 
by
\[
\widetilde{\psi_t}(\zeta)
\;=\;
\left\{\begin{array}{ll} 
{\psi_t(\zeta)} & \ \ \textrm{for}\ \ \zeta\in D(z,r_1)\ , \\
{}       & {} \\
{\zeta}      & \ \ \textrm{for}\ \ \zeta\in\mathbb{C}\setminus \mathbb{D}\ .
\end{array}\right.
\]
Since $D(z,R)\subset\mathbb{D}$, we have from step (iii) that $\psi_t(D(z,r_1))\cap \mathbb{C}\setminus \mathbb{D} =\O$. Hence 
$\widetilde{\psi_t}$, $|t|<\frac{1}{2}$, is a holomorphic family of injections, {\it i.e.,} a holomorphic motion of the set 
$D(z,r_1)\cup (\mathbb{C}\setminus \mathbb{D})$.
\item[(v)] Now apply Slodkowski's Theorem \ref{slodkowski} to get a global extension $\widehat{\psi}_t: \mathbb{C}\to \mathbb{C}$ 
of the motion $\widetilde{\psi_t}$, with time parameter $t$ in $D(0,\frac{1}{2})$. In particular, the map $\widehat{\psi}=\widehat{\psi}_k$ 
is $K$-quasiconformal with $K=\frac{1+2k}{1-2k}$, and it maps the unit disk onto itself. Moreover, we have 
\[
 \widehat{\psi}|_{D(z,r_1)}=\psi|_{D(z,r_1)}=\phi|_{D(z,r_1)}\ .
\]
Thus, $\widehat{\psi}$ is the desired modification of $\phi$ away from $z$. 
\item[(vi)] We are now in a position to use the same {\it annulus trick\/} we employed in the proof of Lemma \ref{lem3}. 
Let $\rho>0$, $\lambda>1$ and the absolute constant $C_0>0$ be as in the proof of that Lemma. In particular, 
$\rho^2=J_\phi^h(z)=J_{\widehat{\psi}}^h(z)$, and thus our goal is to bound $\rho$ from above. 
We have $\lambda\leq 1+\epsilon$, and we may assume that $\rho>\lambda+\epsilon$, otherwise there is nothing to prove. 
Now let $r_2>0$ be given by
\[
 r_2\;=\; \frac{\epsilon}{3C_0M}\;<\; \frac{\rho\epsilon}{C_0M\lambda^2(\lambda+\epsilon)}\ .
\]
Then for all $r\leq r_2$ the inequality \eqref{eq9} holds. Let us choose $r=\min\{r_1,r_2\}$. With this choice of $r$, 
using the Taylor expansion \eqref{eq8}
as in the proof of Lemma \ref{lem3} we see that $\Omega_0=\widehat{\psi}(D(z,r))\setminus D(z,r)=\phi(D(z,r))\setminus D(z,r)$ is a conformal annulus,  
with 
\begin{equation}\label{logmod}
\!\!\!\mod{(\Omega_0)}\geq \log{\frac{\rho}{\lambda+\epsilon}}\ . 
\end{equation}
\item[(vii)] Now define $\Omega_n=\widehat{\psi}^n(\Omega_0)$ for all $n\geq 0$, and note that 
\begin{equation}\label{modomega}
 \!\!\!\mod{(\Omega_n)}\;\geq\;\left(\frac{1-2k}{1+2k}\right)^n\!\!\!\!\!\!\mod{(\Omega_0)}\ .
\end{equation}
Since $\cup_{n\geq 0}\Omega_n \subset \mathbb{D}\setminus D(z,r)$, we deduce from \eqref{logmod} and \eqref{modomega} that
\begin{equation}\label{logrhoestimate}
 \log{\left(\frac{\rho}{\lambda+\epsilon}\right)} \sum_{n=0}^{\infty} \left(\frac{1-2k}{1+2k}\right)^n\;\leq\; \log{\frac{2}{r}}\ ,
\end{equation}
where we have used the estimate on $\!\!\!\mod{(\mathbb{D}\setminus D(z,r))}$ given by Lemma \ref{lem1} (and Remark \ref{rem1}). 
From \eqref{logrhoestimate} it follows that
\begin{equation}\label{loglesslog}
 \log{\left(\frac{\rho}{\lambda+\epsilon}\right)} \;\leq\; \frac{4k}{1+2k}\log{\frac{2}{r}}\;<\;4k\log{\frac{2}{r}}\ .
\end{equation}
\item[(vii)] But from our choices of $r_1$ and $r_2$, we see that $r=\min\{r_1,r_2\}=c_2k^3r_z^{6\beta+9}$, for some constant 
$c_2>0$. Hence
\[
 \log{\frac{2}{r}}\;\leq\; \log{\frac{2}{c_2}} +3\log{\frac{1}{k}}+(6\beta+9)\log{\frac{1}{r_z}}\ .
\]
Putting this back into \eqref{loglesslog} and using that $k\leq \mathrm{(const.)}r_z^\beta$, we deduce that, for each $0<\theta<1$,
\begin{align*}
 \log{\left(\frac{\rho}{\lambda+\epsilon}\right)} \;&\leq\; c_3k+c_4k\log{\frac{1}{k}} + c_5k\log{\frac{1}{r_z}} \\
 &\leq\; c_6r_z^{\beta(1-\theta)} + c_7r_z^{\beta}\log{\frac{1}{r_z}} \\
 &\leq\; c_8 r_z^{\beta(1-\theta)}\ .
\end{align*}
Here the constants $c_6, c_7, c_8$ depend on $M,\beta, b_0$ and also on $\theta$. 
From this it follows that 
\[
 \rho \;\leq\; 1+ c_9r_z^{\beta(1-\theta)}\ ,
\]
and therefore 
\[
 J_\phi^h(z)\;=\; \rho^2\;\leq\; 1 + c_{10}r_z^{\beta(1-\theta)}\ ,
\]
where the constant $c_{10}$ depends on $M,\beta, b_0$ and $\theta$.
\end{enumerate}
Hence we have established \eqref{hyperjacobian}, with $c_{10}$ playing the role of $C_\theta$, in the case when $z$ is a fixed-point of $\phi$. 
As we already remarked, the general case follows from this one by post-composition of $\phi$ with a suitable automorphism of the disk, 
using the same procedure given in the proof of Proposition \ref{thm1}. It is here, and only here, that \eqref{alphaishere} is used. 
Hence the final constant $C_\theta$ indeed depends on $M, \alpha, \beta, b_0$, and of course also on $\theta$. This finishes the proof. 
\end{proof}

As we informally said in the beginning of this subsection, our goal is to develop bounds on the infinitesimal distortion, 
by a self-map (diffeomorphism) of a hyperbolic Riemann surface, 
of the underlying hyperbolic metric in terms of the local quasiconformal distortion of the map. 
So far we have only shown how to bound in such terms the {\it hyperbolic Jacobian\/} of these maps. 
Can we use such estimates on the Jacobian to bound the {\it infinitesimal distortion of the hyperbolic metric\/}? 
The answer is yes, and the reason lies in the fact that there is a simple relationship between 
the two concepts. More precisely, let $\phi:Y\to Y$ be a quasiconformal diffeomorphism. Then 
for each $z\in Y$ and each non-zero tangent vector $v\in T_zY$, we have 
\begin{equation}\label{hypdistortjacobian}
 \frac{1}{K_\phi(z)}\,J_\phi^h(z)\;\leq \; \left( \frac{|D\phi(z)v|_Y}{|v|_Y}\right)^2\;\leq K_\phi(z)\,J_\phi^h(z)\ .
\end{equation}
This fact is classical (see for instance \cite[p.~17]{McM}).

\begin{theorem}\label{exphypmetric}
 Let $U,V\subset \mathbb{C}$ be Jordan domains, symmetric about the real axis, with $\overline{U}\subset V$, and let $Y=V\setminus \mathbb{R}$. 
 Let $\phi:V\to V$ be a $C^r$ diffeomorphism which is symmetric about the real axis, and write 
 \[ 
  M=\max\left\{ \mathrm{diam}(V),(\mathrm{dist}(\partial V,\partial U))^{-1}\,,\,\|\phi\|_{C^2}\,,\,\|\phi^{-1}\|_{C^2}\right\}>0
 \]
Then the following facts hold true for each $0<\theta<1$.
\begin{enumerate}
 \item[(i)] If $\phi$ is $(1+\delta)$-quasiconformal ($\delta>0$), then for each $z\in U\cap Y$ with $\phi(z)\in U\cap Y$ 
 and all non-zero tangent vectors $v\in T_zY$ we have
 \begin{equation}\label{exphypmetric1}
  \left(1+C_\theta\delta^{1-\theta}\right)^{-1} \leq \frac{|D\phi(z)v|_Y}{|v|_Y} \leq 1+C_\theta\delta^{1-\theta}\ ,
 \end{equation}
where $C_\theta>0$ depends only on $\theta$ and $M$.
\item[(ii)] If $\phi$ is asymptotically holomorphic of order $r$, so that $|\mu_\phi(z)|\leq b_0|\mathrm{Im}\,z|^{r-1}$ for all $z\in Y$, 
then for each $z\in U\cap Y$ with $\phi(z)\in U\cap Y$ 
 and all non-zero tangent vectors $v\in T_zY$ we have
 \begin{equation}\label{exphypmetric2}
  \left(1+C_\theta|\mathrm{Im}\,z|^{(r-1)(1-\theta)}\right)^{-1} \leq \frac{|D\phi(z)v|_Y}{|v|_Y} \leq 
  1+C_\theta|\mathrm{Im}\,z|^{(r-1)(1-\theta)} \,
 \end{equation}
 where $C_\theta>0$ depends only on $\theta$, $M$ and $b_0$. 
\end{enumerate}

\end{theorem}

\begin{proof} The hard work has already been done in Propositions \ref{thm1} and \ref{prop52}, and all 
we have to do is to show, with the help of \eqref{hypdistortjacobian}, how to reduce the present theorem to 
the situation in those auxiliary results. 
There is no loss of generality in assuming that $\phi$ preserves $Y^+=Y\cap \mathbb{C}^+$ 
(and therefore also $Y^{-}=Y\cap \mathbb{C}^{-}$). Also, it suffices to establish the upper estimates in 
\eqref{exphypmetric1} and \eqref{exphypmetric2}, since the lower estimates follow by replacing $\phi$ 
with its inverse. Moreover, by symmetry we only need to establish these upper estimates for 
points $z\in U\cap Y^+$. 

Let $(a,b)=V\cap \mathbb{R}$, and let $\varphi:V\to \widehat{\mathbb{C}}$ be a holomorphic univalent map with $\varphi(Y^+)=\mathbb{D}$, 
$\varphi(Y^{-})=\widehat{\mathbb{C}}\setminus \overline{\mathbb{D}}$, normalized so that $\varphi(a)=-1$, $\varphi(b)=+1$. 
Let $W^*=\bigcup_{\zeta\in \varphi(U^+)} \Delta_\zeta\subset \mathbb{D}$, and consider $W=\varphi^{-1}(W^*)\subset Y^+$. Note 
that $W\supset U^+$. By Lemma \ref{koebebounds} (ii), the $C^2$ norms of the restrictions $\varphi|_{W}$ and $\varphi^{-1}|_{\varphi(W^*)}$ are both 
bounded by a constant that depends only on $\mathrm{dist}(\partial V, \partial W)$, and it is not difficult (albeit a bit laborious) 
to see that this last distance is bounded by a constant that depends only on $M$. 
These bounds also imply that there exists a constant $K_1>1$ depending only on $M$ such that 
\begin{equation}\label{disttoboundary}
 \frac{1}{K_1}\left(1-|\varphi(z)|\right)\;\leq\; |\mathrm{Im}\,z|\;\leq\; K_1\left(1-|\varphi(z)|\right)
\end{equation}
for all $z\in W$. 

Now consider the $C^2$ diffeomorphism $\psi:\mathbb{D}\to \mathbb{D}$ given by 
$\psi=\varphi\circ \phi\circ \varphi^{-1}$. Note that, by the chain rule and the bounds on $\varphi$, $\varphi^{-1}$ stated above, the $C^2$ norm 
of $\psi|_{W^*}$ is also bounded by a constant that depends only on $M$. 

Given a point $z\in  Y^+$ and a vector $v\in T_zY^+\equiv T_zY$, let 
$\zeta=\varphi(z)\in \mathbb{D}$ and $w=D\varphi(z)v\in T_\zeta\mathbb{D}$. Since 
$\varphi$ yields an {\it isometry\/} between the hyperbolic metric of $Y^+$ ({\it i.e.,} of $Y$) and the hyperbolic 
metric of $\mathbb{D}$, we have 
$|v|_Y=|w|_{\mathbb{D}}$. Moreover, by the chain rule we have
\[
 |D\phi(z)v|_Y\;=\; |D\varphi^{-1}(\psi(\zeta))\,D\psi(\zeta)w|_Y\;=\; |D\psi(\zeta)w|_{\mathbb{D}}\ ,
\]
where in the last step we have used that $\varphi^{-1}$ yields an isometry between the hyperbolic metric of $\mathbb{D}$ and the 
hyperbolic metric of $Y^+$ (and therefore the derivative $D\varphi^{-1}(\psi(\zeta))$ is an infinitesimal isometry between 
corresponding tangent spaces). This shows that for each $z\in Y^+$ and each non-zero tangent vector $v\in T_zY$, we have 
\begin{equation}\label{phipsizzeta}
 \frac{|D\phi(z)v|_Y}{|v|_Y}\;=\; \frac{|D\psi(\zeta)w|_{\mathbb{D}}}{|w|_{\mathbb{D}}}\ .
\end{equation}
In  addition, since $\varphi$ and $\varphi^{-1}$ are conformal, we have that $\psi$ and $\phi$ have the same dilatation 
at corresponding points, {\it i.e.}, $K_\psi(\zeta)=K_\phi(z)$ for all $z\in Y^+$. Also, since 
$\varphi$ and $\varphi^{-1}$ are hyperbolic isometries, the hyperbolic Jacobians of $\psi$ and $\phi$ agree on corresponding 
points, {\it i.e.,} $J_\psi^h(\zeta)=J_\phi^h(z)$. 

Putting these facts together, we see that the assertions (i) and (ii) in the statement 
({\it i.e.,} the estimates in \eqref{exphypmetric1} and \eqref{exphypmetric2}) 
will be proved for $\phi$ as soon as the corresponding assertions for $\psi$ are proved. 
But assertion (i) for $\psi$ follows by putting together Proposition \ref{thm1} and \eqref{hypdistortjacobian}, 
whereas assertion (ii) for $\psi$ follows by putting together Proposition \ref{prop52} and \eqref{hypdistortjacobian}. 
To see why this is so, we need to check that, in each case, the hypotheses of the corresponding propositions 
are satified by $\psi$.

{\it Case (i)\/}. If $\phi$ is $(1+\delta)$-quasiconformal, as in (i),  then $\psi$ is $(1+\delta)$-quasiconformal as well. 
The hypotheses on $\phi$ imply that there exists a constant $K_2>1$ depending only on $M$ such that 
\begin{equation}\label{imaginaryratio}
 \frac{1}{K_2}\;\leq\; \frac{|\mathrm{Im}\,z|}{|\mathrm{Im}\,\phi(z)|}\;\leq\; K_2
\end{equation}
for all $z\in W$. Applying this with $z=\varphi^{-1}(\zeta)$ for $\zeta\in W^*$ and using \eqref{disttoboundary}, we deduce that 
there exists $K_3>1$ depending only on $M$ such that 
\[
 \frac{1}{K_3}\;\leq\; \frac{\rho_{\mathbb{D}}(\zeta)}{\rho_{\mathbb{D}}(\psi(\zeta))}\;\leq\; K_2
\]
for all $\zeta\in W^*$. This shows that the inequality \eqref{alphainverse} in the hypothesis of Proposition \ref{thm1} is 
satisfied for $\psi$. Moreover, we have for each $\zeta\in \varphi(U^+)$ we have $\Delta_\zeta\subset W^*$, and 
so, in the notation introduced before , $m_\psi(\zeta)\leq \|\psi|_{W^*}\|_{C^2}\leq K_4$, where $K_4>0$ is 
a constant that depends only on $M$. Hence all the hypotheses of Proposition \ref{thm1} are satisfied by $\psi$. 
It follows that, for each $0<\theta<1$, there exists a constant $K_\theta$ depending only on $\theta$ and $M$ such that 
\begin{equation}\label{jacobian}
 J_{\psi}^h(\zeta)\;\leq\; 1+ K_\theta \delta^{1-\theta}\ ,
\end{equation}
for all $\zeta \in \varphi(U^+)$. Combining \eqref{jacobian} with the general upper estimate in \eqref{hypdistortjacobian} 
(for $\psi$), we see that for each $0<\theta<1$ there exists a constant $C_\theta>0$ depending only on $\theta$ and $M$ such that
\begin{equation}\label{dpsi}
 \frac{|D\psi(\zeta)w|_{\mathbb{D}}}{|w|_{\mathbb{D}}}\;\leq 1+ C_\theta \delta^{1-\theta}\ ,
\end{equation}
for all $\zeta \in \varphi(U^+)$ and each non-zero tangent vector $w\in T_\zeta\mathbb{D}$. 
Putting \eqref{dpsi} together with \eqref{phipsizzeta} for $z=\varphi^{-1}(\zeta)\in U^+$ and $v=D\varphi^{-1}(\zeta)w\in T_zY^+$, 
we deduce the upper estimate in \eqref{exphypmetric1}, as desired.

\noindent{\it Case (ii)\/}. If $\phi$ is asymptotically holomorphic (near the real axis) then so is $\psi$ (near 
the boundary of the unit disk). Verifying the hypotheses of Proposition \ref{prop52} for $\psi$ in this case is similar to 
what was done in case (i), hence we omit the details.

\end{proof}

\begin{remark}
 In the application we have in mind, namely Theorem \ref{control} below, the diffeomorphism $\phi$ will be the 
 asymptotically holomorphic diffeomorphism appearing in the  
 Stoilow decomposition of a high renormalization of an (infinitely renormalizable) AHPL-map. For such maps, we can always 
 assume that the constant $b_0$ appearing in assertion (ii) is equal to one. The reason for this is 
 embedded in the proof of a slightly improved version of the complex bounds (see Theorem \ref{complexbounds} (iv)). 
\end{remark}

\section{Recurrence and expansion}\label{sec:recurrence}

This section contains a crucial step towards the proof of our Main Theorem (as stated in the introduction), namely Theorem 
\ref{control} below. 
We show that every AHPL-map arising as a deep renormalization of an infinitely renormalizable $C^r$ unimodal 
map with bounded combinatorics expands the hyperbolic metric of its co-domain minus the real axis. 
From this we deduce a few basic properties concerning the global dynamics of these AHPL-maps -- such as the fact that 
all of their periodic points are expanding. The expansion property proved here will lead 
to much stronger results in \S \ref{localconnectivity}, including, of course, the proof of the Main Theorem.

\subsection{Controlled AHPL-maps}\label{sec:control}

In order to establish the desired expansion property, we need to assume that our AHPL-maps 
satisfy certain {\it geometric constraints\/}. We call such maps {\it controlled AHPL-maps\/}. 
These geometric constraints may seem artificial, but the point is that they are always verified 
once we renormalize a given AHPL-map a sufficient number of times. 

Let us proceed with the formal definition. 
First, we need some notation. Given $z=x+iy\in \mathbb{C}\setminus \mathbb{R}$ and $\alpha>1$, let 
$z_\alpha=x+i\alpha y$. 

\begin{definition}\label{def:control}
 Let $\alpha, M>1$ and $0<\delta,\theta<1$ be real constants, and let $n_0\in \mathbb{N}$. 
 An AHPL-map $f:U\to V$ of class $C^r$, $r\geq 3$, is said to be 
 \emph{$(\alpha, \delta, \theta,M, n_0)$-controlled} if the following conditions are satisfied.
 \begin{enumerate}
  \item[(i)] We have $\mathrm{diam}(V)\leq M$ and $\mathrm{mod}(V\setminus U)\geq M^{-1}$;
  \item[(ii)] If $f=\phi\circ g$ is the Stoilow decomposition of $f$, with $\phi:V\to V$ a $C^r$-diffeomorphism 
  and $g:U\to V$ holomorphic, then $\|\phi\|_{C^2}, \|\phi^{-1}\|_{C^2}\leq M$;
  \item[(iii)] $\phi$ is $(1+\delta)$-quasiconformal on $V$; 
  \item[(iv)] The dilatation $\mu_\phi$ satisfies $|\mu_\phi(z)|\leq M|\mathrm{Im}\,z|^{r-1}$;
  \item[(v)] For all $z\in U_\alpha=U\cap \{w:\,|\mathrm{Im}\,w|\leq (\alpha M)^{-1}\}$, we have
  $D(z_\alpha,|\mathrm{Im}\,z_\alpha|)\subset Y=V\setminus \mathbb{R}$; 
  \item[(vi)] For all $z\in U\setminus \mathbb{R}$ we have $M^{-1}\leq |\mathrm{Im}\,z|/|\mathrm{Im}\,\phi(z)|\leq M$, 
  as well as $M^{-1}\leq \rho_Y(z)/\rho_Y(\phi(z))\leq M$;
  \item[(vii)] We have 
  \[
   \Phi(\mathrm{diam}_Y(U\setminus U_\alpha)+2n_0\log{M})\;<\; 1-C_\theta\delta^{1-\theta}\ ,
  \]
  where $\Phi$ is McMullen's universal function \eqref{McMeq2} and $C_\theta=C_\theta(M)$ is the constant appearing in Theorem \ref{exphypmetric} (i).
 \end{enumerate}
\end{definition}

\begin{remark}\label{rem:control}
 It is possible to prove, with the help of Lemma \ref{disttoroot} and the Riemann mapping theorem, 
 that $\mathrm{diam}_Y(U\setminus U_\alpha)\leq C+\log{\alpha}$ 
 for some positive constant $C=C(M)$. 
\end{remark}

The following result is a straightforward consequence of the complex bounds, as given by Theorem \ref{complexbounds}, 
together with the $C^2$ bounds, as given by Theorem \ref{c2bounds} and Remark \ref{stoilowdecomp}. 

\begin{theorem}\label{thm:61}
 For each positive integer $N$ there exists $M=M(N)>1$ such that the following holds. Let $f: U\to V$ 
 be an AHPL-map of class $C^r$, $r\geq 3$, whose restriction to the real line is an infinitely renormalizable unimodal map 
 with combinatorics bounded by $N$. Then for each $\alpha>1$ and $0<\theta<1$ and each $n_0\in \mathbb{N}$, there exist $0<\delta<1$ and 
 $n_1=n_1(f,\alpha, \theta, n_0)\in \mathbb{N}$ such that, for all $n\geq n_1$, the $n$-th renormalization $R^nf:U_n\to V_n$ 
 is an $(\alpha, \delta, \theta,M,n_0)$-controlled AHPL map. 
\end{theorem}

Now, we have the following main theorem.

\begin{theorem}\label{control}
Given $M>1$, $r>3$ and $0<\theta<1$ so small that $(r-1)(1-\theta)>2$, 
there exists $\alpha_0>1$ such that the following holds for all $\alpha>\alpha_0$. 
 Let $f: U\to V$ be an AHPL-map of class $C^r$ and assume that $f$ is $(\alpha, \delta, \theta,M, n_0)$-controlled 
 for some $0<\delta<1$ and some $n_0\in \mathbb{N}$. Suppose also that $r$, $\alpha$,  $\theta$ and $n_0$ are such that
 \begin{equation}\label{control0}
  r\;>\; 1 + \frac{4n_0\alpha}{(n_0-1)(1-\theta)(2\alpha-1)}\ .
 \end{equation}
 Then the following assertions hold true.
 \begin{enumerate}
  \item[(a)] There exists a constant $0<\eta<1$ such that $|Df^n(z)v|_Y\geq \eta |v|_Y$, for all 
  $z\in Y\cap U$ such that $f^i(z)\in Y$ for $0\leq i\leq n$ and all $v\in T_zY$. 
  \item[(b)] If $z$ is a point in the filled-in Julia set of $f$ and its $\omega$-limit set is not contained 
 in the real axis, we have $|Df^n(z)v|_Y/|v|_Y\to \infty$ as $n\to \infty$, for each non-zero tangent vector $v\in T_zY$. 
  \item[(c)] Every periodic orbit of $f$ is expanding.
  \item[(d)] The expanding periodic points are dense in the set of all recurrent points.
 \end{enumerate}

\end{theorem}

\begin{proof}

First we give an informal description of the argument. For a suitable constant $0<\lambda <1$, we partition the domain of $f=\phi\circ g$ 
into a sequence of {\it scales\/}, the $n$-th scale being the set of points in the domain (off the real axis) whose 
distance to the real axis is of the order $\lambda^n$. The rough idea then is that at each level the worst expansion of the hyperbolic metric of $Y$ 
by $g$ beats the best contraction of that metric by $\phi$. In this, we are aided by Theorem \ref{exphypmetric} and Lemma \ref{mclemma}. 
We warn the reader that, in what follows, whenever invoking Theorem \ref{exphypmetric}, we denote by $C_\theta$ the {\it largest\/} of the 
two constants with that name appearing in assertions (i) and (ii) of said theorem. 

Let us now present the formal proof. Let us assume we are given a large number $\alpha>1$. How large 
$\alpha$ must be will be determined in the course of the argument. 

To start with, note that by \eqref{hypmet2} in Lemma \ref{lem:hypmet} we have, for all $z\in U_\alpha$,
\begin{equation}\label{control1}
 \frac{1}{|\mathrm{Im}\,z|} \;\leq\; \rho_Y(z)\;\leq\; 
  \frac{1}{|\mathrm{Im}\,z|}\left(1-\frac{1}{2\alpha}\right)^{-1}\ .
\end{equation}
Let us fix for the time being a real number $0<\lambda <1$, which we will use to define the {\it scales\/} we mentioned above.
For definiteness, we take $\lambda =M^{-1}$. 
For each $n\geq 1$ we define
\[
 W_n\;=\; \left\{z\in U_\alpha\,:\, \frac{\lambda^n}{\alpha M}\leq |\mathrm{Im}\,z| < \frac{\lambda^{n-1}}{\alpha M}\right\}\ .
\]
Also, we set $W_0=U\setminus U_\alpha \subset Y$. Then we have, of course, $U\setminus \mathbb{R}= \bigcup_{n=0}^\infty W_n$. 
\smallskip
\smallskip

{\noindent \it Claim.\/ There exists a sequence of numbers $\xi_n>1$, $n\geq 0$, with $\xi_n\to 1$ as $n\to\infty$, having the following 
property: For each $z\in W_n$ and each tangent vector $v\in T_{z}Y$, we have
\begin{equation}\label{claim:control}
 |D(g\circ \phi)(z)v|_Y \;\geq\; \xi_n|v|_Y\ .
\end{equation} }
\begin{proof}[Proof of Claim]
In order to prove this claim, we analyse separately the expansion of the conformal map $g$ and the (possible) contraction of the quasi-conformal 
diffeomorphism $\phi$. We proceed through the following steps.

\begin{enumerate}
\item[(i)] Let $X\subset Y$ be the open set containing $\phi(z)$ such that $g$ maps $X$ univalently onto $Y$. 
Writing $w=D\phi(z)v\in T_{\phi(z)}Y$, and applying Lemma \ref{mclemma} together with the estimate \eqref{McMeq3}, we 
deduce that
\begin{equation}\label{control2}
 |Dg(\phi(z))\,w|_Y\;\geq\; \left(1 +\frac{1}{3}e^{-2s_{X,Y}(\phi(z))}\right)|w|_Y\ .
\end{equation}
Now we need to estimate $s_{X,Y}(\phi(z))$. 

\item[(ii)] Let us write $p=\phi(z)=x+iy$ and let $q=x+i(\alpha M)^{-1}\frac{y}{|y|}\in U\setminus U_\alpha$, which lies in the same vertical 
as $p$. There are two cases to consider:
\begin{enumerate}
 \item[(1)] We have $p\in X$ but $q\notin X$. In this case, we have $d_Y(p,Y\setminus X)\leq d_Y(p,q)$. Using 
 \eqref{control1}, we get
 \[
  s_{X,Y}(\phi(z))\leq d_Y(p,q)\leq \left(1-\frac{1}{2\alpha}\right)^{-1}\log{\frac{(\alpha M)^{-1}}{|\mathrm{Im}\,\phi(z)|}}\ .
 \]
But by property (vi) of Definition \ref{def:control} we have  $|\mathrm{Im}\,\phi(z)|\geq M^{-1}\lambda^{n}(\alpha M)^{-1}$. 
Hence 
\begin{equation}
 s_{X,Y}(\phi(z))\leq \left(1-\frac{1}{2\alpha}\right)^{-1}\left[ n\log{\frac{1}{\lambda}} +\log{M}\right]\ .
\end{equation}

 \item[(2)] We have $p\in X$ and $q\in X$. In this case we have
 \begin{align*}
  d_Y(p,Y\setminus X)& \leq d_Y(p,q) + d_Y(q,Y\setminus X) \\
                     & \leq d_Y(p,q) + \mathrm{diam}_Y(U\setminus U_\alpha)\ .
 \end{align*}
 Therefore
 \begin{equation}\label{control3}
  s_{X,Y}(\phi(z))\leq C_\alpha + \left(1-\frac{1}{2\alpha}\right)^{-1}\left[ n\log{\frac{1}{\lambda}} +\log{M}\right]\ ,
 \end{equation}
where $C_\alpha=\mathrm{diam}_Y(U\setminus U_\alpha)$. 
\end{enumerate}
Whichever case occurs, we see that \eqref{control3} always holds. Combining these facts with \eqref{control2} we deduce that 
\begin{equation}\label{control4}
 |Dg(\phi(z))w|_Y \geq \left(1 + K_1\lambda^{2n\left(1-\frac{1}{2\alpha}\right)^{-1}}\right)|w|_Y\ ,
\end{equation}
where $K_1=K_1(\alpha, M)$ is the constant given by
\begin{equation}\label{control4a}
 K_1\;=\; \frac{1}{3}e^{-2C_\alpha}\,\exp\left\{-2\left(1-\frac{1}{2\alpha}\right)^{-1}\log{M}\right\} < 1\ .
\end{equation}
This gives us a lower bound on the amount of expansion of the hyperbolic metric of $Y$ by the conformal map $g$ for points at level $n$. 

 \item[(iii)] Let us now bound the amount of contraction of the hyperbolic metric by the quasi-conformal diffeomorphism $\phi$ at $z\in W_n$. 
 First we assume that $n\geq n_0$. 
 Applying Theorem \ref{exphypmetric}(ii), we have for all $v\in T_zY$ the estimate
 \begin{equation}\label{control5}
  |D\phi(z)v|_Y \geq \left(1-C_\theta |\mathrm{Im}\, z|^{(r-1)(1-\theta)}\right)|v|_Y\ ,
 \end{equation}
But since $z\in W_n$, we know that $|\mathrm{Im}\, z|\leq (\alpha M)^{-1}\lambda^{n-1}$.
Carrying this information back into \eqref{control5}, 
we deduce that
\begin{equation}\label{control6}
 |D\phi(z)v|_Y \geq \left(1-K_2\lambda^{(n-1)(r-1)(1-\theta)}\right) |v|_Y\ ,
\end{equation}
where $K_2=K_2(\alpha, \theta, r, M)$ is the constant given by
\begin{equation}\label{control7}
K_2= C_\theta (\alpha M)^{(1-r)(1-\theta)} \ .
\end{equation}

\item[(iv)] Note that both constants $K_1$ and $K_2$ depend on $\alpha$. We claim that 
the ratio $K_2/K_1$ goes to zero as $\alpha\to \infty$. From \eqref{control4a} and \eqref{control7}, we see that
\[
 \frac{K_2}{K_1}\;<\; C_1 e^{2C_\alpha} \alpha^{(1-r)(1-\theta)} \ ,
\]
where $C_1\;=\; 3C_\theta M^{(1-r)(1-\theta)}M^4$ is independent of $\alpha$. 
By Remark \ref{rem:control}, we have $C_\alpha< C_2 + \log{\alpha}$, for some constant $C_2$ depending only on $M$. 
Hence 
\begin{equation}\label{control7aa}
 \frac{K_2}{K_1}\;<\; C_3 \alpha^{2-(r-1)(1-\theta)}\ ,
\end{equation}
where $C_3=C_1e^{2C_2}$. Since by hypothesis $(r-1)(1-\theta)>2$, it follows that the right-hand side of 
\eqref{control7aa} indeed goes to zero as $\alpha\to \infty$. 
Hence we assume from now on 
that $\alpha$ is so large that $2K_2<K_1$. 

\item[(v)] Thus, if for each $n\geq n_0$ we let $\xi_n$ be given by
\begin{equation}\label{control7a}
 \xi_n\;=\; \left(1 + K_1\lambda^{2n\left(1-\frac{1}{2\alpha}\right)^{-1}}\right)
 \left(1-K_2\lambda^{(n-1)(r-1)(1-\theta)}\right)\ ,
\end{equation}
then we have $|D(g\circ \phi)(z)v|_Y\geq \xi_n |v|_Y$ for all $z\in W_n$ and each $v\in T_zY$. Note that 
$\xi_n\to 1$ as $n\to \infty$, because $\lambda<1$. We still need to check that $\xi_n>1$ for all $n\geq n_0$. 
This will be true provided 
\begin{equation}\label{control8}
 K_1\lambda^{2n\left(1-\frac{1}{2\alpha}\right)^{-1}}\; >\; 2K_2\lambda^{(n-1)(r-1)(1-\theta)}\ ,
\end{equation}
for all $n\geq n_0$. Note that both sides of \eqref{control8} are indeed smaller than $1$, because from \eqref{control4a}
and step (iv) we have 
$2K_2<K_1<1$, and $\lambda <1$.
Extracting logarithms from both sides of \eqref{control8}, we get
\[
 \ \ 2n\left(1-\frac{1}{2\alpha}\right)^{-1}\!\!\!\log{\lambda} > (n-1)(r-1)(1-\theta)\log{\lambda} + \log{(2K_1^{-1}K_2)}\ .
\]
Dividing both sides of the above inequality by $(n-1)(1-\theta)\log{\lambda} <0$, we arrive at 
\begin{equation}\label{control9}
 r\;>\; 1 + \frac{2n}{(n-1)(1-\theta)\left(1-\dfrac{1}{2\alpha}\right)} + \frac{\log{(2K_1^{-1}K_2)}}{(n-1)(1-\theta)\log{\dfrac{1}{\lambda}}}
\end{equation}
But since $2K_1^{-1}K_2<1$ (by our choice of $\alpha$ at the end of step (iv)), the third term on the right-hand side of \eqref{control9} 
is negative and therefore can be safely ignored. Moreover, since $n\geq n_0$ we have $2n/(n-1)\leq 2n_0/(n_0-1)$. Therefore 
the inequality \eqref{control8} will hold for all $n\geq n_0$ provided 
\[
 r\;>\; 1+ \frac{2n_0}{(n_0-1)(1-\theta)\left(1-\dfrac{1}{2\alpha}\right)}\ .
\]
But this is nothing but \eqref{control0} in disguise! Hence we have established that the $\xi_n$'s given by \eqref{control7a} 
satisfy $\xi_n>1$, for all $n\geq n_0$. 

\item[(vi)] In order to establish the claim, it remains to analyse what happens when $z\in W_0\cup W_1\cup\cdots \cup W_{n_0-1}$. 
On the one hand, since $\phi$ is $(1+\delta)$-quasiconformal throughout, applying Theorem \ref{exphypmetric} for such $z$ and any $v\in T_zY$ yields the 
lower bound 
\begin{equation}\label{control10}
 |D\phi(z)v|_Y \geq \left(1-C_\theta \delta^{1-\theta}\right)|v|_Y\ .
\end{equation}
On the other hand, using the estimate \eqref{control3} above with $n=n_0$ we deduce that 
\[
 s_{X,Y}(\phi(z))\leq C_\alpha + 2(n_0-1)\log{\frac{1}{\lambda}} +2\log{M} = C_\alpha + 2n_0\log{M}\ .
\]
Therefore, by McMullen's Lemma \ref{mclemma}, we have for all $w\in T_{\phi(z)}Y$, 
\begin{align}\label{control11}
 |Dg(\phi(z))w|_Y \;&\geq\; \Phi(s_{X,Y}(\phi(z))^{-1}|w|_Y\\ 
 & \geq\; \Phi\left(C_\alpha + 2n_0\log{M}\right)^{-1}|w|_Y\ .
\end{align}
Combining \eqref{control10} and \eqref{control11} (with $w=D\phi(z)v$), we deduce that 
\[
|D(g\circ \phi)(z)v|_Y \geq \Phi\left(C_\alpha + 2n_0\log{M}\right)^{-1}\left(1-C_\theta \delta^{1-\theta}\right)\,|v|_Y\ .
\]
Hence we can take 
\[
\xi_0=\xi_1=\cdots =\xi_{n_0-1}=\Phi\left(C_\alpha + 2n_0\log{M}\right)^{-1}\left(1-C_\theta \delta^{1-\theta}\right)\;>\;1\ .
\]
This establishes \eqref{claim:control} for all $z\in W_n$, for all $n\geq 0$, and completes the proof of our claim. 
\end{enumerate}
\end{proof}

With the Claim at hand, we proceed to the proof of the assertions in the statement of our theorem. 
Let $z\in \mathcal{K}_f$ be a point whose iterates up to time $n>1$ stay off the real axis -- 
in other words, $f^i(z)\in Y$ for all $0\leq i\leq n$. 
Note that, since $f=\phi\circ g$, we have $f^n=\phi\circ (g\circ \phi)^{n-1}\circ g$. Write 
$z_1=g(z)$ and define inductively $z_{j+1}=g\circ \phi(z_j)$, for $j=1,\ldots, n-1$. Then for each 
non-zero tangent vector $v\in T_zY$, we have by the chain rule
\begin{equation}\label{control12}
 Df^n(z)v = D\phi(z_n)\left[\prod_{j=1}^{n-1} Dg(\phi(z_j))D\phi(z_j)\right] Dg(z)v \ . 
\end{equation}
Now, since the holomorphic map $g$ expands the hyperbolic metric of $Y$, we have that $|Dg(z)v|_Y>|v|_Y$. 
Moreover, the amount of possible contraction of the hyperbolic metric by the $(1+\delta)$-quasiconformal diffeomorphism $\phi$ 
is bounded from below.   
Indeed, we have $|D\phi(\zeta)w|_Y\geq (1-C_\theta \delta^{1-\theta})|w|_Y$ 
for all $\zeta\in Y$ and all $w\in T_{\zeta}Y$. Moreover, writing $v_1=Dg(z)v\in T_{z_1}Y$ and $v_{j+1}=D(g\circ \phi)(z_j)v_j\in T_{z_{j+1}}Y$ 
for $j=1,\ldots, n-1$, and applying the above Claim, we get
\[
 |v_{j+1}|_Y = |D(g\circ \phi)(z_j)v_j|_Y \;\geq\; \xi_{k_j}|v_j|_Y\ ,
\]
where $k_j\geq 0$ is the unique integer such that $z_j\in W_{k_j}$. Setting $\eta=1-C_\theta\delta^{1-\theta}<1$ and 
carrying these facts back into \eqref{control12}, 
we deduce that
\begin{equation}\label{control13}
 |Df^n(z)v|_Y\;>\; \eta\,\left[\prod_{k=1}^{\infty} \xi_k^{N_{k,n}(z)}\right] |v|_Y\ ,
\end{equation}
where $N_{k,n}(z)$ is the total number of $j$'s in the range $1\leq j\leq n-1$ such that $z_j\in W_{k}$ (in particular, 
the product appearing in the right-hand side is actually finite). This proves assertion (a). 
Now suppose that $z$ is such that its 
$\omega$-limit set accumulates at a point off the real axis, say $p\in Y$. This is the case, for instance,  if 
$z$ is a recurrent or periodic point for $f$. Then there exist $k\geq 0$ and a sequence $j_\nu\to \infty$ such that 
$z_{j_\nu}\to p$ as $\nu\to\infty$ and $z_{j_\nu}\in W_k$ for all $\nu$. But this tells us that $N_{k,n}(z)\to \infty$ as 
$n\to\infty$, and therefore, from \eqref{control13}, we deduce at last that $|Df^n(z)v|_Y/|v|_Y\to \infty$ as $n\to \infty$. 
This proves the desired expansion property stated in assertion (b), and it also proves assertion (c). 
Hence it remains to prove assertion (d). 

Let $z\in Y\cap \mathcal{K}_f$ be a recurrent point. Let $N\geq 1$ be such that $|Df^N(z)v|_Y\geq 3\eta^{-1}|v|_Y$ for all 
$v\in T_zY$, where $\eta$ is the constant of assertion (a). Such an $N$ exists because of assertion (b). By continuity 
of $\zeta\mapsto Df^N(\zeta)$, we can find $\epsilon_0>0$ such that $|Df^N(\zeta)v|_Y\geq 2\eta^{-1}|v|_Y$ for all 
$\zeta\in B_Y(z,\epsilon_0)$  and each $v\in T_{\zeta}Y$. Now, given $0<\epsilon<\frac{1}{4}\eta\epsilon_0$, choose 
$m>N$ such that $f^m(z)\in B_Y(z,\epsilon)$; this is possible because $z$ is recurrent. Write 
$\mathcal{O}=B_Y(f^m(z), 2\epsilon)\subset B_y(z,\epsilon_0)$, and let $\mathcal{O}'\subset Y$ be the component 
of $f^{-m}(\mathcal{O})$ that contains $z$. Then $f^m|_{\mathcal{O}'}: \mathcal{O}'\to \mathcal{O}$ is a diffeomorphism. 
By assertion (a), the inverse diffeomorphism $f^{-m}|_{\mathcal{O}}: \mathcal{O}\to \mathcal{O}'$ is Lipschitz with 
constant $\eta^{-1}$ in the hyperbolic metric of $Y$. Therefore 
\[
\mathcal{O}'\,\subset\, B_Y(z, \eta^{-1}\cdot(2\epsilon)) \,\subset\, B_Y(z, \epsilon_0)\ .
\]
Now that we know this fact, writing $f^m=f^{m-N}\circ f^N$ we see that, for all $\zeta \in \mathcal{O}'$ and
each non-zero 
$v\in T_{\zeta} Y$,
\begin{align*}
 \frac{|Df^m(\zeta)v|_Y}{|v|_Y}\;&=\; \frac{|Df^{m-N}(f^N(\zeta))Df^N(\zeta)v|_Y}{|Df^N(\zeta)v|_Y}\cdot \frac{|Df^N(\zeta)v|_Y}{|v|_Y} \\
& \geq \eta \cdot (2\eta^{-1})\;=\; 2\ .
\end{align*}
Equivalently, we have shown that $|Df^{-m}(\zeta)v|_Y\leq \frac{1}{2}|v|_Y$ for all $\zeta\in \mathcal{O}$ and each $v\in T_\zeta Y$. 
In other words, $f^{-m}|_{\mathcal{O}}: \mathcal{O}\to \mathcal{O}'$ is, in fact, a contraction of the hyperbolic metric 
of $Y$, with contraction constant $\frac{1}{2}$. In particular,
\[
 \mathcal{O}'= f^{-m}|_{\mathcal{O}}(\mathcal{O}) \subset B_Y(z,\epsilon) \Subset B_Y(f^m(z), 2\epsilon) =\mathcal{O}\ .
\]
This means that $f^{-m}|_{\mathcal{O}}$ maps the hyperbolic ball $\mathcal{O}$ strictly inside itself (and it is a contraction 
of the hyperbolic metric). Hence there exists $z_*\in \mathcal{O}'$ such that $f^m(z_*)=z_*$, and this periodic point is necessarily expanding, 
by assertion (c). Thus, we have proved that for each $\epsilon>0$ there exists an expanding periodic point $\epsilon$-close to $z$. 
This establishes assertion (d) and completes the proof of our theorem.

\end{proof}

It is worth pointing out that, combining Theorem \ref{control} with Theorem \ref{thm:61}, we already deduce the following 
simple properties of the dynamics of all sufficiently deep renormalizations of a given AHPL-map. Considerably stronger 
results will be proved in \S \ref{localconnectivity} below. 

\begin{corollary}
 Let $f: U\to V$ be an AHPL-map of class $C^r$, with $r> 3$, whose restriction to the real line is an infinitely renormalizable unimodal map 
 with bounded combinatorics. There exists $n_1=n_1(f)\in \mathbb{N}$ such that, for all $n\geq n_0$, the 
 $n$-th renormalization $f_n=R^nf: U_n\to V_n$ is an AHPL-map with the following properties.
 \begin{enumerate} 
  \item[(a)] Every periodic orbit of $f_n$ is expanding.
  \item[(b)] The expanding periodic points are dense in the set of all recurrent points.
  \item[(c)] There are no stable components of $\mathrm{int}(\mathcal{K}_{f_n})$ whose closures intersect the real axis. 
 \end{enumerate}
\end{corollary}

\begin{proof}
 Choose $0<\theta<1$, as well as $n_0\in \mathbb{N}$ and $\alpha>1$ large enough so that 
 \eqref{control0} holds true. This is possible because $r>3$. Then, by Theorem \ref{thm:61}, 
 there exists $n_1\in \mathbb{N}$ such that for all $n\geq n_1$, the $n$-th renormalization $f_n$ of $f$ is 
 an $(\alpha, \delta, \theta,M,n_0)$-controlled AHPL map, for some $0<\delta<1$. Hence assertions (a) and (b) 
 follow from the corresponding assertions in Theorem \ref{control}. To prove (c), suppose $\Omega\subset Y_n= V_n\setminus \mathbb{R}$ 
 is a stable component of $\mathrm{int}(\mathcal{K}_{f_n})$ such that $\overline{\Omega}\cap \mathbb{R}=\O$. 
 Let $p\geq 1$ be such that $f_n^p(\Omega)=\Omega$. Also, consider the decomposition of the domain of $f_n$ into {\it scales\/} 
 as in Theorem \ref{control}. Since $\overline{\Omega}\subset U_n\setminus \mathbb{R}\subset Y_n$ is compact, 
 it is contained in the union of finitely many scales. In each scale $f_n$ expands the hyperbolic metric of $Y_n$ by a definite amount.
 Hence so does $f_n^p$ on $\Omega$. But this is impossible, because $\Omega$ has finite hyperbolic area. 
\end{proof}

\section{Topological conjugacy to polynomials and
local connectivity of Julia sets}\label{localconnectivity}
In this section,
we will prove that a
$(\alpha, \delta,\theta, M,n_0)$-controlled 
AHPL-mapping $f:U\rightarrow V,$ which is infinitely 
renormalizable of bounded type,
is topologically
conjugate to a real polynomial in a neighbourhood of its
filled Julia set,
so that from the topological point
of view, the dynamics of these mappings are the same
as those of polynomials; in particular, such mappings do not have
wandering domains. We will also prove that 
the Julia set of such an AHPL-mapping is locally connected.
Specifically, we will assume that $f$ satisfies the
conditions of Theorem~\ref{control}. In particular,
we assume that
$f:U\rightarrow V$ is a $C^r$ asymptotically holomorphic polynomial-like
mapping that is
$(\alpha, \delta,\theta, M,n_0)$-controlled,
$$r>1+\frac{4n_0\alpha}{(n_0-1)(1-\theta)(2\alpha-1)},$$
and that the conclusions of Theorem~\ref{control} all hold.
By Theorems \ref{complexbounds} and \ref{thm:61}, for any $r>3,$
if $g$ is a $C^r$ mapping of the interval, which is infinitely
renormalizable of bounded type, then for any $n$ sufficiently
large, there is a renormalization, \mbox{$R^ng:U_n\rightarrow V_n$}
of $g,$ which is an AHPL-mapping that
satisfies these
assumptions.

\subsection{Dilatation and expansion}
The proof of the following lemma is implicit in 
the proof of Theorem~\ref{control}; 
it makes the lower bound in
Equation~(\ref{claim:control})
explicit.
\begin{lemma}\label{lem:dilexp}
Let $\xi_n$ be the constant defined in  
Equation~(\ref{control7a}).
There exists $N\geq n_0$ such that if
$n\geq N,$
then
\begin{equation}\label{eqn:dilexp}
1+M\Big(\frac{\lambda^{n-1}}{\alpha M}\Big)^{r-1}\leq \xi_n.
\end{equation}
\end{lemma}
\begin{proof}
It is sufficient to show that
\begin{equation*}
M\Big(\frac{\lambda^{n-1}}{\alpha M}\Big)^{r-1}\leq K_1\lambda^{2n(1-\frac{1}{2\alpha})^{-1}}-2K_2\lambda^{(n-1)(r-1)(1-\theta)},
\end{equation*}
see Equation~(\ref{control8}).
Factoring out $\lambda^{(n-1)(r-1)}$ on the right and cancelling it
with the same term on the left, this is equivalent to:
\begin{equation}
\frac{M}{(\alpha M)^{r-1}}\leq
K_1\lambda^{2n(1-\frac{1}{2\alpha})^{-1}-(n-1)(r-1)}
-2K_2\lambda^{-\theta(n-1)(r-1)}.
\label{eqn:63}
\end{equation}
Since $n>n_0$, we have that
\begin{equation}
4n_0\alpha (n-1)-4n\alpha(n_0-1) =
4\alpha(n_0(n-1)-n(n_0-1))>0
\label{eqn:7b}
\end{equation}
Now, since
$$
r\geq 1+\frac{4 n_0 \alpha}{(n_0-1)(1-\theta)(2\alpha-1)},
$$
we have that
\begin{equation}
(r-1)(1-\theta)(n-1)\geq \frac{4n_0\alpha(n-1)}{(n_0-1)(2\alpha-1)}.
\label{eqn:7r}
\end{equation}
So
\begin{eqnarray*}
(r-1)(1-\theta)(n-1)-2n(1-\frac{1}{2\alpha})^{-1}
&\geq& \frac{4n_0\alpha(n-1)}{(n_0-1)(2\alpha-1)}-
2n\frac{2\alpha}{2\alpha-1}\\
&=&\frac{4n_0\alpha(n-1)-4n\alpha(n_0-1)}{(n_0-1)(2\alpha-1)}\\
&>&0,
\end{eqnarray*}
where the first inequality follows from (\ref{eqn:7r}) and
the last inequality follows from (\ref{eqn:7b})
Thus we have
$$2n(1-\frac{1}{2\alpha})^{-1}-(n-1)(r-1)\leq-\theta(n-1)(r-1),$$
since both exponents on the right hand side of (\ref{eqn:63}):
$$2n(1-\frac{1}{2\alpha})^{-1}-\theta(n-1)(r-1)
\mbox{ and }
-\theta(n-1)(r-1)$$ are negative,
equation~(\ref{eqn:dilexp}) holds for $n$ sufficiently large.
\end{proof}

Let $$K_{f^n}(z)=\frac{1+|\mu_{f^n}(z)|}{1-|\mu_{f^n}(z)|},$$
be the \emph{quasiconformal 
distortion of} $f^n$ \emph{at} $z$.
A \emph{chain} of domains is a
sequence of domains $\{B_j\}_{j=0}^n$
where $B_j$ is a component of $f^{-1}(B_{j+1})$ for all
$j=0,1,2,\dots,n-1$ and $B_n$ is a domain in $\mathbb C$.
To a mapping $f^n:A\rightarrow B,$
we associate the chain
of domains $\{B_j\}_{j=0}^n$, where $B_n=B$ and 
$B_j=\mathrm{Comp}_{f^{j}(B)}f^{-(n-j)}(B)$ for $j=0,\dots,n-1$.

Recall that $W_k$ is the strip 
\[
 W_k\;=\; \left\{z\in U_\alpha\,:\, \frac{\lambda^k}{\alpha M}\leq |\mathrm{Im}\,z| < \frac{\lambda^{k-1}}{\alpha M}\right\}\ .
\]

\begin{corollary}\label{cor:dilexp}
For each $N\in\mathbb N$ there exists $c>0$
such that the following holds.
Let $A$ be an open domain in $\mathbb C$.
Suppose that $f^n:A\rightarrow B$ 
is onto and let $\{B_j\}_{j=0}^n$
be the chain with $B_0=A$ and $B_n=B$. 
Assume that for each $0\leq j\leq n$ that
$$\#\{k:B_j\cap W_k\neq\O\}\leq N.$$
Then
$$c\cdot \sup_{z\in A} \log K_{f^n}(z)\leq  \inf_{z\in A}\log|Df^n(z)v|_Y,$$
for each unit tangent vector $v\in T_zY$.
\end{corollary}
\begin{proof}
Let us express $f^n:B_0\rightarrow B_n$
as $\phi\circ (g\circ \dots \circ\ g \circ \phi )\circ g$.
For each $0\leq j < n$,
$g:B_j\rightarrow \phi^{-1}(B_{j+1})$.
Since $\phi$ is a 
$(1+\varepsilon(\delta))$-quasi-isometry 
in the 
hyperbolic metric on $Y$ where $\varepsilon(\delta)\rightarrow 0$ as
$\delta\rightarrow 0,$ we have that
there exists $N_1$, depending only on $N$, so that
$\phi^{-1}(B_{j})$ intersects at most $N_1$ strips
$W_k$.

For each $B_j,$
let $n_j$ be minimal so that
$\phi^{-1}(B_j)\cap W_{n_j}\neq\O$.
Then for any $g(z)\in \phi^{-1}(B_j)$, $1\leq j<n$, 
we have that 
$$\Bigg|\frac{\bar\partial(g\circ\phi)}{\partial(g\circ\phi)}(g(z))\Bigg|
=\Bigg|\frac{\bar\partial\phi}{\partial \phi}(g(z))]\Bigg|
\leq M\Big(\frac{\lambda^{n_{j}-1}}{\alpha M}\Big)^{r-1}.$$

By equation~(\ref{claim:control}) and Lemma~\ref{lem:dilexp}, we have that for all
$v\in T_zY,$  with $|v|_Y=1$,
$$|D(g\circ \phi)(z)|_Y\geq 1+M\Big(\frac{\lambda^{n_{j}-1+N_1}}{\alpha
  M}\Big)^{(r-1)}=1+M\lambda^{N_1(r-1)}\Big(\frac{\lambda^{n_{j}-1}}{\alpha M}\Big)^{r-1},$$
so that
$$|D(g\circ\phi)(z)v|_Y\geq
\Big(1+\lambda^{N_1(r-1)}\sup_{z\in B_j}\Bigg|\frac{\bar\partial(g\circ\phi)}{\partial(g\circ \phi)}(g(z)) \Bigg|\Big)|v|_Y.$$
Thus we have that
$$\inf_{z\in B_j}|D(g\circ\phi)(z)v|_Y\geq
\Big(1+\lambda^{N_1(r-1)}\sup_{z\in B_j}\Bigg|\frac{\bar\partial(g\circ\phi)}{\partial(g\circ \phi)}(g(z)) \Bigg|\Big)|v|_Y.$$

For each $i$, let 
$$k_i=\#\{j:B_j\cap W_i\neq\O,
\mbox{ and for all }i'<i, B_j\cap W_{i'}=\O\},$$
and let us reindex the $B_j$
as follows: 
For each $i\in\mathbb N\cup\{0\}$, 
let $B_{i_{0}},\dots, B_{i_{k_i}}$ be an enumeration
of all $B_j$ so that $B_j\cap W_i\neq\O$
and for all $0\leq i'<i,$ $B_j\cap W_{i'}=\O$.
Notice that $n=\sum_{i=0}^\infty k_i$.

By the chain rule and Theorem~\ref{c2bounds},
we have that there exists a constant $c_1>0$ so that
$$
\inf_{z\in B_0} |Df^n(z)v|_Y  \geq  c_1
\prod_{i=0}^\infty \prod_{j=0}^{k_i}(1+ \lambda^{N_1(r-1)}\sup_{z\in
  B_{i_j}}
|\mu_f(z))|)$$
Now, there exists a constant $c_2>0$ such that

$$\log \prod_{i=0}^\infty \prod_{j=0}^{k_i}(1+
\lambda^{N_1(r-1)}  \sup_{z\in B_{i_j}}|\mu_f(z)|)
 = \sum_{i=0}^\infty\sum_{j=0}^{k_i}\log (1+ \lambda^{N_1(r-1)}\sup_{z\in
  B_{i_j}}|\mu_f(z)|)$$
\begin{eqnarray*}
& \geq & c_2 \sum_{i=0}^\infty\sum_{j=0}^{k_i} \lambda^{N_1(r-1)}\sup_{z\in
  B_{i_j}}|\mu_f(z)|\\
& =& c_2\frac{\lambda^{N_1(r-1)}}{2}  \sum_{i=0}^\infty\sum_{j=0}^{k_i} 
\Big(\sup_{z\in B_{i_j}}\big(|\mu_f(z)|-(-|\mu_f(z)|)\big)\Big)\\
&\geq& c_2\frac{\lambda^{N_1(r-1)}}{2}  \sum_{i=0}^\infty\sum_{j=0}^{k_i} 
\sup_{z\in B_{i_j}}\log\Big(\frac{1+|\mu_f(z)|}{1-|\mu_f(z)|}\Big)\\
&=&c_2\frac{\lambda^{N_1(r-1)}}{2}\log\prod_{i=0}^\infty
\prod_{j=0}^{k_i}\sup_{z\in B_{i_j}}
\Big(\frac{1+|\mu_f(z)|}{1-|\mu_f(z)|}\Big).
\end{eqnarray*}

Hence there exists a constant $c$
so that, 
$$\inf_{z\in B_0} \log |Df^n(z)v|_Y\geq c\cdot \log \sup_{z\in B_0}K_{f^n}(z).$$
\end{proof}

\subsection{Puzzle pieces}
Let us construct external rays for $f$.
These will allow us to construct Yoccoz puzzle
pieces for $f$ where the 
role of equipotentials is played by
the curves $f^{-i}\partial V$.
To construct these rays, we use a method analogous to the one used by 
Levin-Przytycki in \cite{LP}
to construct external rays for
holomorphic
polynomial-like maps.

First,
we associate to $f$
an external map, $h_f$ as follows:
Let 
$X_0=V$ and 
for $i\in\mathbb N$,
set $X_{i+1}=f^{-1}(X_i)$.
Notice that since 
$U\Subset V$,
$f:U\rightarrow V$ is 
a branched covering of $V$, ramified at 
a single point, 0, and $f^i(0)\in U$ for all $i$, we
have that $X_i=f^{-i}(V)$ is a
connected and simply connected topological disk for all
$i\in\mathbb N\cup\{0\}$,
and
$X_{i+1}\Subset X_i.$
Let $M=\mod(V\setminus \mathcal K_f),$ and
let 
$$\phi: D(0,e^M)\setminus\overline{\mathbb D}
\rightarrow V\setminus \mathcal K_f $$
be the uniformization of $V\setminus \mathcal K_f$ by
a round annulus.
Let $D_i=\phi^{-1}(X_{i}),$
we have that each annulus  $D_i\setminus\overline D_{i+1}$
is mapped as a $d$-to-1 covering map onto 
$D_{i-1}\setminus D_i$ by $h_f=\phi^{-1}\circ f\circ \phi$.
The mapping $h_f$ extends continuously to $\partial\mathbb D$, 
and by Schwarz reflection, 
$h_f$ can be defined as 
a mapping between annuli $W'\subset W$, each with the
same core curve, $\partial \mathbb D$.
We have that $h_f$ is a
$C^3$
 expanding mapping of $S^1$ (see the proof of 
\cite{CvST}~Lemma~10.17)
and that the dilatation of $h_f$ on
$W'$ is the same as the dilatation of $f$.
Foliate $W\setminus W'$
by $C^{r},$ $h_f$ invariant rays, 
connecting $\partial W'$ and $\partial W$.
and pull them back by $h_f$. 
We obtain a foliation by $C^r$ rays of
$W'\setminus \partial \mathbb D$ that is continuous on
$W'$. 
Pulling back
this foliation of $W'$ by $\phi$, we obtain a foliation of $V\setminus
\mathcal K_f.$
The leaves of this foliation are the \emph{external rays of} $f$.

\begin{remark} 
\label{page:rays} Observe that since
$h_f|S_1$ is a degree $d$ expanding mapping of the
circle, it is topologically conjugate to
$z\mapsto z^d$ on a neighbourhood of $S^1$.
Consequently, one can carry out this construction simultaneously
for two mappings 
$f:U\rightarrow V$ and $\tilde f:\tilde U\rightarrow\tilde V$
to obtain a mapping
$H:V\rightarrow\tilde V$ such that $H\circ F(z)=\tilde F\circ H(z)$
for any $z\in U$ contained in an equipotential or ray.
\end{remark}
\medskip

For each $z\in V\setminus \mathcal K_f,$ we let $R_z$ denote the ray through $z$.
Let us parameterize $R_z$ by $R_z(t), t\geq 0,$ such that for each $n\in\mathbb N$
we have that $R_z(n)$ is the unique point on $R_z$ that passes through
$\partial X_n$.
We say that a ray $R_z$ \emph{lands} at a point $p$ if
$\lim_{t\rightarrow\infty} R_z(t)=p$.

To prove that certain rays land, we will need the following lemma.
\begin{lemma}\cite[Lemma 2.3]{BeL}\label{lem:bl23}
Let $\Omega\subset\mathbb C$ be a hyperbolic region. Let $\gamma_n:[0,1]\rightarrow\Omega$ be a family of curves with
uniformly bounded hyperbolic length and such that $\gamma_n(0)\rightarrow\partial\Omega.$ Then $\mathrm{diam}(\gamma_n)\rightarrow 0.$
\end{lemma}

\begin{lemma}\label{lem:raysland}
If $R_z$ accumulates on a real repelling periodic
point $p$, then $R_z$ lands at $p$.
\end{lemma}
\begin{proof} Compare \cite[Lemma 2.1]{LP} and \cite{BeL}.
Suppose that $p$ is a real repelling periodic point of period $s$.
Then one can repeat the proof of linearization near repelling
periodic points of holomorphic maps to prove that
there exists a neighbourhood $B$ of $p$ such that
$f^s$ is conjugate to $z\mapsto\lambda z$ near $p$, where
$\lambda=Df^s(p)$, see \cite{Mil1}.

Let $R_z([n-1,n])$ be the segment of the ray connecting
$\partial X_{n-1}$ and $\partial X_n.$ Let us show that
$\mathrm{diam}(R_z([n-1,n]))\rightarrow 0$ as $n\rightarrow\infty$.
By Lemma~\ref{lem:bl23}, and since $\phi$ is an isometry in
the hyperbolic metric, it is sufficient to show that
the curves $\phi^{-1}(R_z([n-1,n]))$ have uniformly
bounded hyperbolic lengths. 
This follows from the fact that $\|Dh_f(z)\|>1$ in the hyperbolic metric for $z$
sufficiently close to $\partial\mathbb D$, which was proved in the
proof of  \cite[Lemma 10.17]{CvST}.
Thus we have that $\mathrm{diam}(R_z([n-1,n]))\rightarrow 0$ as
$n\rightarrow\infty.$ So there exists $n_0\in\mathbb N$ such that 
for all $n\geq n_0$, we have that
$R_{z}([n,n+1])\subset (f|_B)^{-s(n-n_0)}(B).$
Since $f^s|_B$ is
qc-conjugate to $z\mapsto\lambda z$ with $\lambda>1$
in a neighbourhood of 0, we have that
$\cap_{n=n_0}^{\infty}(f|_B)^{-s(n-n_0)}(B)=\{p\}.$
So the only accumulation point of the ray is $p$. 
\end{proof}

We define puzzle pieces for $f$ as follows.
Let us index the renormalizations 
$R^nf:U_n\rightarrow V_n$ of $f$ by
$f_n:U_n\rightarrow V_n$,
so that $f_n=f^{q_n}|_{U_n}.$
Let $I_n=\mathcal K_{f_n}\cap\mathbb R$
denote the invariant interval 
for $f_n$.
Let $\tau:I_0\rightarrow I_0$
be the even, dynamical, symmetry about the
even critical point at 0.
Let $\beta_n\in\partial I_n$
be the orientation preserving
fixed point of $f_n$ in $\partial I_n.$
By real-symmetry, there exist two rays,
labeled $R_{\beta_n}$ and $R_{\beta_n}'$
that land at $\beta_n$.
Let $R_{\tau(\beta_n)}$ and $R_{\tau(\beta_n)}'$
denote the preimages under $f^{q_n}$ of 
$R_{\beta_n}$ and $R_{\beta_n}'$, respectively,
which land at $\tau(\beta_n)$.
For each $n\in\mathbb N$,
the \emph{initial configuration of puzzle pieces
at level} $n$ are the components of 
$V\setminus (R_{\beta_n}\cup R_{\beta_n}'\cup R_{\tau(\beta_n)}
\cup R_{\tau(\beta_n)}'\cup\{\beta_n,\tau(\beta_n)\})$.
We denote this union of puzzle pieces by
$\mathcal{Y}^{(n)}_0$. Given an initial configuration,
$\mathcal{Y}^{(n)}_0$, for
 $j\in\mathbb N\cup\{0\}$, we define
$\mathcal Y_{j}^{(n)}$ to be the union of the connected components
of $f^{-j}(\mathcal{Y}^{(n)}_0)$. 
Given any $z\in \mathcal K_f$, we let $Y^{(n)}_j(z)$ denote the component of
 $\mathcal Y_{j}^{(n)}$ that contains $z$, and we let
$Y^{(n)}_j=Y^{(n)}_j(0)$ be the component that contains the critical point.

\begin{lemma}\label{lem:smallcrit}
For each $n\in\mathbb N,$ there exists $j$, so that $\mathcal K_{f_n}\subset
Y^{(n)}_j \subset U_n$.
\end{lemma}
\begin{proof}
For all $j\in\mathbb N$,
$\mathcal K_{f_n}\subset \overline{Y}^{(n)}_j$.
Let $q_n$ be the period of the renormalization
$f_n$ of $f$. Let $K_j=\mathrm{comp}_0 f^{-q_n j}(Y_0^{(n)})$.
Since $K_j\subset K_{j-1}$ and $f^{s_n}:K_{j}\rightarrow K_{j-1}$,
and $\cap_{j=0}^\infty \overline{K}_j$ is a compact connected set,
we have that $\mathcal K_{f_n}\subset \cap_{j=0}^\infty \overline{K}_j\subset U_n$.
\end{proof}

\begin{proposition}\label{prop:shrinkingpp}
Suppose that $z\in \mathcal K_f$. Then there exist arbitrarily small
neighbourhoods $P$ of $z$ such that $P$ is a union of puzzle pieces.
\end{proposition}
\begin{proof}
Observe that Lemma~\ref{lem:smallcrit} implies that there
are arbitrarily small puzzle pieces containing the critical point of
$f$. Let us start by spreading this information 
throughout the filled Julia set of $f$. Let $z\in \mathcal K_f$.

\medskip
\textit{Case 1:
Assume that
$0\in \omega(z)$.} 
For each $n$, let $C_n\subset U_n$ be the puzzle piece given by
Lemma~\ref{lem:smallcrit}.
Let $r_n$ be minimal so that $f^{r_n}(x)\in C_n$ and
let $C^0_n=\mathrm{comp}_xf^{-r_n}(C_n)$.
Each $C_n$ is contained in the topological disk, $\Gamma_n$, bounded by
the core curve $\gamma_n$ of the annulus $V_n\setminus\overline U_n$.
By Theorem~\ref{complexbounds}, there exists $C>0$ such that
for all $n\in\mathbb N,$ we have that 
$\mathrm{mod}(V_n\setminus \overline U_n)\geq C^{-1}.$ 
Thus the domain $\Gamma_n$ is a 
$K=K(C)$-quasidisk.
Let $V_n^0=\mathrm{Comp}_x f^{-r_n}(V_n)$, and
$\Gamma_n^0=\mathrm{Comp}_x f^{-r_n}(\Gamma_n).$
It is not hard to see that $f^{r_n}:V_n^0\rightarrow V_n$ is
a diffeomorphism: Suppose that there exists 
$0<j<r_n$ so that  $f^j(V^0_n)\owns 0,$
but $f^j(V_n^0)$ is not contained in $U_n$,
so that $f^j(V_n^0)\cap\partial U\neq\O.$ 
Since $f^{s_n}:U_n\rightarrow V_n$ is
a first return mapping to $V_n$,
for all $k\in\mathbb N$, $f^{j+ks_n}(V_n^0)$
intersects both $\mathcal K_{f_n}$ and $\partial V_n,$
and we have that there exists no $j_1\in\mathbb N$ such that
$f^{j_1}(V_n^0)=V_n.$ Thus we have that if for some $j,$
$f^j(V^0_n)\owns 0,$ then $f^j(V_n^0)\subset U_n$, but then
since for all $k\in\mathbb N$, $f^{-ks_n}(C_n)\cap (V_n\setminus C_n)=\O,$
$j=r_n$, and so $f^{r_n}:V_n^0\rightarrow V_n$ is
a diffeomorphism.

\medskip
\textit{Case 1a:
Suppose that $0\in\omega(z)$ and
$z\in\mathbb R\cap \mathcal K_f$.} Then, by the complex bounds,
we have that there exists $K>1$ for each $n$,
the mapping $f^{r_n}:V_n^0\rightarrow V_n$ 
is a diffeomorphism with
quasiconformal distortion bounded by $K$.
Hence there exists $m>0$ depending only on $K$ and
$M$ such that for all $n$,
$\mod(\Gamma^0_n\setminus\overline{\Gamma}_{n+1}^0)>m$.
Thus the puzzle pieces $C_n^0$ have diameters converging to 
$0$.

\medskip
\noindent \textit{Case 1b:
Suppose that $0\in\omega(z)$, $\omega(z)\subset \mathbb R$,
and for all $j$, $f^j(z)\notin\mathbb R$.}
We consider the case when the mappings 
$f^j$ 
have uniformly bounded quasiconformal distortion
near
$z,$ and the case when 
they have unbounded quasiconformal distortion near $z,$ separately. 
First, suppose that there exists $K_x\geq 1$ 
such that for each $n$ the mapping $f^{r_n}:C^0_n\rightarrow C_n$
extends to a mapping from $V_n^0$ 
onto $V_n$ with quasiconformal distortion bounded by $K_x$.
We have that each $\Gamma_n^0$ is a $K_1$-quasidisk,
for some $K_1>1$ depending on $x$, 
and there exists a constant $m>0$ such that for all $n$,
$\mod(\Gamma_n^0\setminus \Gamma_{n+1}^0)\geq m$,
and so the puzzle pieces $C_n^0$ shrink to $z$.

Suppose now that the quasiconformal distortion of $f^{r_n}:V_n^0\rightarrow V_n$
tends to infinity as $n$ tends to infinity.
For each $n$, let
$\{V_n^j\}_{j=0}^{r_n}$ be the chain with
$V_n^{r_n}=V_n$ and
$V_n^0=\mathrm{Comp}_z f^{-r_n}(V_n),$ and
let
$\{\Gamma_n^j\}_{j=0}^{r_n}$ be the chain with
$\Gamma_n^{r_n}=\Gamma_n$ and
$\Gamma_n^0=\mathrm{Comp}_z f^{-r_n}(\Gamma_n).$
For all $n$ sufficiently large,
there exists $0\leq j_n<r_n$ maximal so that the set
$V_n^{j_n}=\mathrm{Comp}_{f^{j_{n}}(z)}f^{-(r_n-j_n)}(V_n)$ does not
intersect the real line (see case 2a below). Let 
$\Gamma_n^{j_n}=\mathrm{Comp}_{f^{j_{n}}(z)}f^{-(r_n-j_n)}(\Gamma_n).$
Since $\partial\Gamma_n^{j_n}$ is the core curve 
$V_n\setminus U_n$, and $f^{(r_n-j_n)}$ has bounded quasiconformal distortion, 
we have that 
there exists $m_1>0$ such that
$\mathrm{mod}(V_n^{j_n}\setminus\Gamma_n^{j_n})>m_1.$
But this implies that there exists $m_2>0$ such that
$\mathrm{dist}(\partial
V_n^{j_n},\Gamma_n^{j_n})>m_2\mathrm{diam}(\Gamma_n^{j_n})$,
which immediately gives us that
there exists $m_3>0$ so that
$\mathrm{dist}(\Gamma_n^{j_n},\mathbb R)>
m_3\mathrm{diam}(\Gamma_n^{j_n}).$
It follows that there exists $\xi>0$ such that
for all $n$,
$\mathrm{diam}_Y(\Gamma_n^{j_n})<\xi.$

Let us inductively choose a subsequence
$V_{n_i}$
 of
of the levels $V_n$ so that the landing maps
from $V_{n_i}^{j_{n_i}}$ to $V_{n_{i+1}}^{j_{n_{i+1}}}$ all have definite expansion.
Let $\eta\in(0,1)$ be the constant from
Theorem~\ref{control},
so that $|Df^i(z)v|_Y\geq\eta|v|_Y$.
Then we have that if $X,$ a component of
$f^{-i_0}(V_{n}^{j_n}),$ is
a pullback of $V_{n}^{j_n}$ 
such that the 
quasiconformal distortion of $f^{i_0}|_X$
is bounded by
$2(1+\delta)/\eta$, then there exists
$N\in\mathbb N$ such that for each $i\leq i_0$,
for each element $X_i=f^i(X)$ in the
chain associated to the pullback, $X_i$
intersects at most $N$ of the strips 
$W_k$.
Let $c>0$ be the
constant associated to $N$
from Corollary~\ref{cor:dilexp}.
Let $k_0>0$ be minimal so that
$$
\sup_{z\in \Gamma^0_{k_0}} K_{f^{j_{k_0}}}(z)^c\geq \frac{2}{\eta}.
$$
Let $0\leq j_{k_0}'<j_{k_0}$ be maximal so that
$$
\sup_{z\in \Gamma^{j_{k_0}'}_{k_0}} K_{f^{j_{k_0}-j_{k_0}'}}(z)^c\geq \frac{2}{\eta}.
$$
Then, since $f$ is $(1+\delta)$-qc,
we have that
$$
\sup_{z\in \Gamma_{k_0}^{j_{k_0}'}} K_{f^{j_{k_0}-j_{k_0}'}}(z)^c\leq (1+\delta)\frac{2}{\eta}.
$$
Thus by Corollary~\ref{cor:dilexp} we have that 
$$\mathrm{diam}_Y(\Gamma_{k_0}^{j_{k_0}'})\leq 
\frac{\eta}{2}\mathrm{diam}_Y (\Gamma^{j_{k_0}}_{k_0})\leq \frac{\eta\xi}{2},$$
and by Theorem~\ref{control}, we have that 
$$\mathrm{diam}_Y (\Gamma^0_{k_0})<\frac{\xi}{2}.$$

We now repeat the argument:
let $k_1>k_0$ be minimal so that
so that 
$$
\sup_{z\in \Gamma^{j_{k_0}}_{k_0}} K_{f^{-(j_{k_1}-j_{k_0})}}(z)^c\geq
\frac{2}{\eta},
$$
and let 
$0\leq j_{k_1}'<j_{k_1}$ be maximal so that
$$
\sup_{z\in \Gamma^{j_{k_1'}}_{k_1}} K_{f^{j_{k_1}-j_{k_1}'}}(z)^c\geq \frac{2}{\eta}.
$$
Then, since $f$ is $(1+\delta)$-qc,
we have that
$$
\sup_{z\in \Gamma_{j_{k_1}}^{k_1'}} K_{f^{j_{k_1}-j_{k_1}'}}(z)^c\leq (1+\delta)\frac{2}{\eta}.
$$
Again by Corollary~\ref{cor:dilexp} we have that 
$$\mathrm{diam}_Y(\Gamma_n^{j_{k_1}'})\leq 
\frac{\eta}{2}\mathrm{diam}_Y(\Gamma^{j_{k_1}}_{k_1})\leq \frac{\eta\xi}{2},$$
and by Theorem~\ref{control}, we have that 
$$\mathrm{diam}_Y (\Gamma^{j_{k_0}}_{k_1})<\frac{\xi}{2}.$$
Combining this with the first step, we have that
$$\mathrm{diam}_Y(\Gamma_{k_1}^0)<\xi/4.$$
If the quasiconformal distortion of $f^n$ diverges at $x$, we see that 
we can repeat this argument infinitely many times
to obtain a nest of puzzle pieces
$\{C^0_{k_i}\}$
 about $z$ such that $\mathrm{diam}_Y(C^0_{k_i})\rightarrow 0$.

\medskip
Combining Cases (1a) and (1b), we have that for all $z$ such that
$0\in\omega(z)$, that there are arbitrarily small puzzle pieces $P\owns
z.$ Now we treat the cases when $0\notin\omega(z)$.

\medskip
\noindent\textit{Case 2a: Suppose that there exists
$n\in\mathbb N$ such that
$\omega(z)\subset\mathbb R\setminus V_n.$}
Let $\mathcal{Y}^{(n)}_0$, be the initial configuration
of puzzle pieces at level $n$.
Let $x_0\in\omega(z),$ then,
since the real traces of puzzle pieces shrink
to points, there exist $m_0>0$ and a union of (closed)
puzzle pieces of $\mathcal{Y}^{(n)}_{m_0}$, 
denoted by $Q_0$, such that
$Q_0\cap\omega(0)=\O$
and $x_0\in\mathrm{int}(Q_0)$.
Let $\mathcal{Y}^{(n)}_j(x_0)$ denote
the closure of the set of puzzle pieces
$P$ in $\mathcal{Y}^{(n)}_j$ with
$x_0\in\overline{P}.$
Let $Q=\cap_{j=0}^\infty \mathcal{Y}^{(n)}_j(x_0).$

Let us show that $Q=\{x_0\}$.
If $\mathrm{diam}(Q)>0,$
then, since $\cup_n f^n(Q)$ is a bounded set,
there exists $C>0$, 
$x\in Q$ and a vector $v\in T_x\mathbb C$
such that $|Df^{k_i}(x)v|<C.$
If $\omega(x)$ is not contained in 
the real-line, then in a small
neighbourhood of $x$, 
the hyperbolic metric on $Y$ is comparable to 
the Euclidean metric, but now 
$|Df^{k_i}(x)v|<C$
contradicts Theorem~\ref{control} (b).
So we can assume that
$\omega(x)\subset \mathbb R,$
but then $\omega(x)$ is contained in the
hyperbolic set of points that avoid
$V_n$, and we have that
$|Df^{k_i}(x)v|\rightarrow\infty$
for any $v\in T_x\mathbb C,$
and so $\mathrm{diam}(Q)=0$.
Let us point out that this argument shows that 
if $z\in\mathbb R$ is contained in a hyperbolic set,
then for any $n$ sufficiently big,
$\mathrm{diam}(\mathcal{Y}^{(n)}_{j}(z))\rightarrow 0$ as
$j\rightarrow\infty$, and indeed that
$J_f$ is locally
connected at any point in $J_f\cap\mathbb R$ that is contained in 
a hyperbolic set.

Suppose that for all $j\in\mathbb N\cup\{0\},$
$f^j(z)\notin \mathbb R.$
Let $r_0$ be the first return time of
$x_0$ to $Q_0$, and let 
$Q_1=\mathrm{Comp}_{x_0}f^{-r_0}(Q_0)$.
Inductively define $Q_{i+1}$ by taking
$r_i$ to be the first return time of $x_0$ to
$Q_i$ and setting $Q_{i+1}=\mathrm{Comp}_{x_0}f^{-r_i}(Q_i).$
Let $\varepsilon>0$ be so small that if
$z=x+iy$ satisfies
$\mathrm{dist}(z,\mathbb R)<\varepsilon$
and $z\notin V_n$, then $\mathrm{dist}(x,0)>\mathrm{diam}(V_n)/2$.
Since $x_0\in\omega(z)$, there exist 
$n_i\rightarrow\infty$ with the property that 
$n_i$ is minimal with $f^{n_i}(z)\in Q_i$.
It is sufficient to show that there exists
a constant $c>0$ so that for all $i$,
$\|Df^{n_i}(z)\|\geq c.$ 
Fix some $i\in\mathbb N$.
Let 
$j_0\geq n_0$ be minimal so that
$\mathrm{dist}(f^{j_0}(z),\mathbb R)>\varepsilon,$
and let $j_1\leq n_i$ be maximal so that
$\mathrm{dist}(f^{j_1}(z),\mathbb R)>\varepsilon,$
then there exists a constant 
$c_1>0$ so that
$$\|Df^{n_i}(z)\|\geq
c_1\eta\|Df^{n_i-j_1}(f^{j_1}(z))\|\|Df^{j_0-n_0}(f^{n_0}(z))\|\|Df^{n_0}(z)\|.$$
Thus it suffices to bound
$\|Df^{n_i-j_1}(f^{j_1}(z))\|$ and $\|Df^{j_0-n_0}(f^{n_0}(z))\|$ from below.
Let $z_0=f^{n_0}(z)$ and define $z_i=f^i(z_0)$, $x_i=f^i(x_0)$.
Then there exist constants $c_2,c_3$ so that
$$\|Df^{j_0-n_0}(z_0)\|
\geq c_2\prod_{i=0}^{j_0-n_0}\|Df(x_i)\|
\prod_{i=0}^{j_0-n_0}(1-c_3\frac{|z_i-x_i|}{\|Df(x_i)\|}).$$
By our choice of $\varepsilon$, and since $x_0$ is contained in a hyperbolic Cantor set,
we have that there exists a constant $c_4>0$ and $\Lambda>1$
so that
$$\sum_{i=0}^{j_0-n_0}\frac{|z_i-x_i|}{\|Df(x_i)\|}\leq 
\frac{1}{2\mathrm{diam}(V_n)}\sum_{i=0}^{j_0-n_0}|z_i-x_i|
\leq \frac{1}{2\mathrm{diam}(V_n)}\frac{c_4\varepsilon}{1-\Lambda^{-1}}.$$
Thus we have that $\|Df^{j_0-n_0}(z_0)\|$ is bounded from below.
The proof that $\|Df^{n_i-j_1}(f^{j_1}(z))\|$ is bounded from 
below is similar.

\medskip
\noindent\textit{Case 2b: Suppose that $\omega(z)\not\subset\mathbb R$.}
Let $z_0$ be an accumulation point of $\omega(z)$ that is not
contained in $\mathbb R$.
Since the real puzzle pieces shrink to points, there exist
$n$ and $m$ and a union $Q$ of puzzle pieces in
$\mathcal{Y}^{(n)}_m$ and a sequence $k_i\rightarrow\infty$
such that
$Q\cap\mathbb R=\O,$ and
$f^{k_i}(z)\in Q$ for all $i$. 
By Theorem~\ref{control} (b),
we have that
$$\mathrm{diam}(\mathrm{Comp}_{f^{k_0}(z)}(f^{-(k_i-k_0)}(Q)))\rightarrow
0\mbox{
as }i\rightarrow \infty.$$
Thus by Theorem~\ref{control} (a), 
$$\mathrm{diam}(\mathrm{Comp}_{z}(f^{-k_i}(Q)))\rightarrow 0\mbox{
as }i\rightarrow\infty.$$
\end{proof}

Proposition~\ref{prop:shrinkingpp} has several important consequences.

\begin{corollary}\label{CorMain}
Suppose that $f\in C^r$ is an asymptotically holomorphic polynomial-
like mapping, which is
$(\alpha, \delta,\theta, M,n_0)$-controlled, and that
$$r>1+\frac{4n_0\alpha}{(n_0-1)(1-\theta)(2\alpha-1)}.$$
Then the following hold:
\begin{enumerate}
\item $\mathcal J_f=\mathcal K_f$.
\item $f:U\rightarrow V$ is topologically conjugate to a polynomial
mapping in a neighbourhood of its Julia set. In particular,
$f:U\rightarrow V$ has no wandering domains.
\item $\mathcal J_f$ is locally connected.
\end{enumerate}
\end{corollary}
\begin{proof}
\textit{(1).} To see that $\mathcal J_f=\mathcal K_f$ observe that for
each $z\in \mathcal K_f$, 
there are arbitrarily small puzzle pieces containing
$z$, so $z$ is a limit of points whose orbits eventually land in
$V\setminus U$. Thus $z\in \mathcal J_f$.
In particular, $\mathcal K_f$ has empty interior.

\medskip

\noindent\textit{(2).} Let us now show that
$f:U\rightarrow V$ is topologically conjugate to a polynomial
mapping in a neighbourhood of its Julia set.
Let $I\subset U\cap\mathbb R$ denote the invariant interval 
for $f$. Since $f|_{I}$ has negative Schwarzian derivative, 
there exists a real polynomial $p$
with a critical point of the same degree as the
critical point of $f$
such that $f$ is topologically conjugate
to $p$ on $I.$
Let $h:I\rightarrow I$ be the continuous mapping
 such that $h\circ f|_I=p\circ h.$
Let $\tilde V$ be a domain containing $\mathcal J_p$ that is bounded by
some level set of the Green's function for $p$.
Let $\tilde U=p^{-1}(\tilde V)$. 

Let $H_0:V\rightarrow\tilde V$ be a 
homeomorphism such that
\begin{itemize}
\item for each $z\in\partial U,$ $H_0\circ F(z) =p\circ
H_{0}(z)$,
\item for each 
$z\in \cup_{n}(R_{\beta_n}\cup R_{\tau(\beta_n)}),$
we have that $H_0(z)\circ f=p\circ H_0(z)$,
and 
\item $H_0|_I=h$.
\end{itemize}
See Remark~\ref{page:rays}
for a description of how to construct such an $H_0$. 

Given that $H_i$ is defined, define $H_{i+1}$ by
$H_{i}\circ f=p\circ H_{i+1}$. Since each $H_i$ is conjugacy on
$J$ between $f$ and $p$ that maps that critical value
of $f$ to the critical value of $p$, this pullback 
is always well-defined and continuous.
Observe that for each $z\in U\setminus \mathcal K_f$,
$H_i$ eventually stabilizes.
Let $H:V\rightarrow\tilde V$ be a limit of the $H_i$.
To see that $H$ is continuous, take any $z\in U$
and let $\{z_n\}$ be a sequence of points such that 
$z_n\rightarrow z$. If $z\notin \mathcal K_f$, then 
there exists a neighbourhood $W$ of $z$ and $i_0\in\mathbb N,$
large, such that for all $i\geq i_0$ and $w\in W$
$H_i(w)=H_{i_0}(w).$ Hence $H(z_n)\rightarrow H(z)$.
So suppose that $z\in \mathcal K_f$,
then since the nests of puzzle pieces about $z$ and $H(z)$ both shrink
to points and $H$ maps puzzle pieces for $f$ to corresponding
puzzle pieces for $p$, $H(z_n)\rightarrow H(z)$.
Also, since for each $z\in U\setminus \mathcal K_f$,
$H_i$ eventually stabilizes, $H:U\rightarrow\tilde U$ 
satisfies $H\circ F(z)=p\circ H(z)$ for all
$z\in U\setminus \mathcal K_f$ and since $\mathcal K_f$ has empty 
interior, we have that
$H$ is a conjugacy between $f$ and $p$ on $U$.

\medskip

\noindent\textit{(3)}. Finally, let us show that $\mathcal J_f$ is locally connected.
Let $z\in \mathcal J_f$, and let $B$ be any
open set that contains $z$, by
Proposition~\ref{prop:shrinkingpp},
there exists a 
neighbourhood $Q\subset B$
of $z$, such that $Q$ is a union of 
puzzle pieces. Since $\mathcal J_f\cap P$
is connected for any puzzle piece
$P$, we have that $\mathcal J_f\cap Q$ is connected too.
\end{proof}

Let us remark that since $f$ is topologically conjugate to a
polynomial,
we obtain that the repelling periodic points of $f$ are
dense in $\mathcal J_f$. We also point out that this implies that
$f$ has no wandering domains, but that this fact can be 
deduced immediately from the fact that
the puzzle pieces shrink to points.

\section*{Acknowledgements}

We would like to thank Dennis Sullivan and Davoud Cheraghi for their
general comments, and Genadi Levin for his keen remarks concerning the proof of 
Lemma \ref{lem:raysland}.


\begin{thebibliography}{[ABD]}




\bibitem{A} 
\newblock L.~Ahlfors, 
\newblock \emph{Lectures on Quasiconformal Mappings\/}. 
\newblock Van Nostrand, 1966. 

\bibitem{AH} 
\newblock J.M. Anderson and A. Hinkkanen, 
\newblock Quasiconformal self-mappings with smooth boundary values, 
\newblock {\em Bull. London Math. Soc.} {\bf 26}(6) (1994),  549--556. 

\bibitem{AIM}
\newblock K.~Astala, T.~Iwaniec and G.~Martin,
\newblock \emph{Elliptic Partial Differential Equations and Quasiconformal Mappings in the Plane\/}.
\newblock Princeton Mathematical Series {\bf{48}}, Princeton University Press, 2009.

\bibitem{AK}
\newblock A. Avila and R. Krikorian,
\newblock Monotonic cocycles.
\newblock {\em   Invent. Math} {\bf 202}(1) (2015)  271--331.

\bibitem{AL}
\newblock A. Avila and M. Lyubich,
\newblock The full renormalization horseshoe for unimodal maps of higher degree: exponential contraction along hybrid classes.
\newblock {\em  Publ.~Math.~IHES} {\bf 114}(1) (2011) 171--223. 


\bibitem{BeL} 
\newblock A.M. Benini and M.  Lyubich, 
\newblock Repelling periodic points and landing of rays for post-singularly bounded exponential maps. 
\newblock{\em Ann. Inst. Fourier (Grenoble)} {\bf  64}(4) (2014),  1493--1520. 

\bibitem{BL}
\newblock A.M. Blokh and M.Yu. Lyubich, 
\newblock Nonexistence of wandering intervals and structure of topological attractors of one-dimensional dynamical systems. II. The smooth case, \newblock {\em Ergod. Th. \& Dynam. Sys.}, {\bf 9} (1989), 751--758.


\bibitem{Car}
\newblock L. Carleson, 
\newblock  On mappings, conformal at the boundary. 
\newblock {\em J. Analyse Math.} {\bf 19} (1967), 1--13.

\bibitem{ChSh}
\newblock D. Cheraghi and M. Shishikura, 
\newblock Satellite renormalization of quadratic polynomials.
\newblock {\tt{arXiv:1509.07843}}

\bibitem{CvS} 
\newblock T.~Clark and  S.~van Strien,
\newblock Quasisymmetric rigidity in dimension one.
\newblock Manuscript 2018. 


\bibitem{CvST} 
\newblock T.~Clark, S.~van Strien and S.~Trejo,
\newblock Complex bounds for real maps, 
\newblock {\em Comm. Math. Phys}. {\bf 355}(3) (2017),  1001--1119. 


 \bibitem{DH}
 \newblock A. Douady and J.H. Hubbard, 
 \newblock On the dynamics of polynomial-like mappings. 
 \newblock {\em Ann. Sci. \'Ecole Norm. Sup.} (4) {\bf 18}(2) (1985) 287--343.
 
\bibitem{Dyn} 
\newblock E. Dyn'kin, 
\newblock Estimates for asymptotically conformal mappings,
\newblock {\em Ann. Acad. Sci. Fenn. Math.} {\bf 22}(2) (1997),  275--304.

\bibitem{edsonwelington1} 
\newblock E.~de Faria and W.~de Melo, 
\newblock Rigidity of critical circle mappings I, 
\newblock {\em J. Eur. Math. Soc.}, {\bf 1} (1999), 339--392.

\bibitem{edsonwelington2} 
\newblock E.~de Faria and W.~de Melo, 
\newblock Rigidity of critical circle mappings II, 
\newblock {\em J. Amer. Math. Soc.}, {\bf 13} (2000), 343--370.

\bibitem{dFdM} 
\newblock E.~de Faria and W.~de Melo,
\newblock \emph{Mathematical Tools for One-dimensional Dynamics},
\newblock Cambridge Studies in Advanced Mathematics {\bf{115}},
\newblock Cambridge University Press, 2008. 

\bibitem{deFariadeMeloPinto} 
\newblock E.~de Faria, W.~de Melo and A.~Pinto, 
\newblock Global hyperbolicity of renormalization for $C^r$ unimodal mappings, 
\newblock {\em Ann. of Math.} {\bf 164} (2006), 731--824. 

\bibitem{dFG} 
\newblock E.~de Faria and P.~Guarino, 
\newblock Real bounds and Lyapunov exponents. 
\newblock \emph{Discrete and Continuous Dynamical Systems A} {\bf 36} (2016), 1957--1982.

%

\bibitem{GaSu1} 
\newblock F.~Gardiner and D.~Sullivan,
\newblock Symmetric structures on a closed curve,
\newblock {\em Amer.~J.~Math.} {\bf{114}} (1992), 683--736. 

\bibitem{GaSu2} 
\newblock F.~Gardiner and D.~Sullivan,
\newblock Lacunary series as quadratic differentials in conformal dynamics,
\newblock {\em Contemp.~Math.} {\bf{169}} (1994), 307--330. 

\bibitem{GdM}
\newblock P. Guarino and W. de Melo,
\newblock Rigidity of smooth critical circle maps.
\newblock {\em J.~Eur.~Math.~Soc} {\bf 19}(6) (2017), 1729--1783.

\bibitem{GdMM}
\newblock P. Guarino, M. Martens and W. de Melo,
\newblock Rigidity of critical circle maps.
\newblock {\tt  arXiv:1511.02792}


\bibitem{GS1} 
\newblock J. Graczyk and G. Swiatek, 
\newblock Polynomial-like property for real quadratic polynomials, 
\newblock {\em Topology Proc.} {\bf  21} (1996), 33--112.


\bibitem{GS2}
\newblock J. Graczyk and G. Swiatek, 
\newblock Generic hyperbolicity in the logistic family.
\newblock {\em Ann. of Math.} (2) {\bf 146} (1997), 1--52.

\bibitem{GSS} 
\newblock J. Graczyk, D. Sands and G. Swiatek, 
\newblock Decay of geometry for unimodal maps: negative Schwarzian case. 
\newblock {\em Ann. of Math.} (2) {\bf 161}(2) (2005),  613--677. 

\bibitem{Gu}
\newblock J. Guckenheimer, 
\newblock Sensitive dependence to initial conditions for one-dimensional maps. 
\newblock {\em Comm. Math. Phys.} {\bf 70}(2) (1979)  133--160.
 
\bibitem{Hub}
\newblock J.H. Hubbard, 
\newblock Local connectivity of Julia sets and bifurcation loci: three theorems of J.-C. Yoccoz. 
\newblock Topological methods in modern mathematics (Stony Brook, NY, 1991), 467--511, Publish or Perish, Houston, TX, 1993.

\bibitem{InSh}
\newblock H. Inou and M. Shishikura, 
\newblock The renormalization for parabolic fixed points and their perturbation, 
\newblock Preprint available at https://www.math.kyoto-u.ac.jp/~mitsu/pararenorm/, 2006. 

\bibitem{Ka} 
\newblock J. Kahn, 
\newblock A priori bounds for some infinitely renormalizable maps: I. bounded primitive combinatorics, 
\newblock Preprint, IMS at Stony Brook, 2006/05, 2006. 

\bibitem{KL} 
\newblock J. Kahn and M. Lyubich, 
\newblock A priori bounds for some infinitely renormalizable  quadratics. II. Decorations, 
\newblock {\em Ann. Sci. {\'E}c. Norm. Sup{\'e}r.} (4) {\bf 41} (2008), no. 1, 57--84. 

\bibitem{KL2} 
\newblock J. Kahn and M. Lyubich, 
\newblock  A priori bounds for some infinitely renormalizable quadratics. III. Molecules. 
\newblock {\em Complex dynamics}, 229--254, A K Peters, Wellesley, MA, 2009.

\bibitem{KT}
\newblock K. Khanin and A. Teplinsky,
\newblock Robust rigidity for circle diffeomorphisms with singularities. 
\newblock {\em Invent. Math.} {\bf 169}(1) (2007),  193--218. 

\bibitem{KvS}
\newblock O. Kozlovski and S. van Strien, 
\newblock Local connectivity and quasi-conformal rigidity of non-renormalizable polynomials. 
\newblock {\em Proc. Lond. Math. Soc.} (3) {\bf 99} (2009)(2), 275--296. 
  
\bibitem{KSS1}
\newblock  O. Kozlovski, W. Shen and S. van Strien,
\newblock Rigidity for real polynomials. 
\newblock {\em Ann. of Math.} (2) {\bf 165} (2007)(3), 749--841.

\bibitem{KSS2}
\newblock  O. Kozlovski, W. Shen and S. van Strien,
\newblock Density of hyperbolicity. 
\newblock {\em Ann. of Math.} (2) {\bf 166} (2007)(1), 145--182.

\bibitem{Lanford}
\newblock O. E. Lanford,
\newblock A computer assisted proof of the Feigenbaum conjectures.
\newblock {\em Bull. Amer. Math. Soc.} {\bf 6} (1982), 427--434.   

\bibitem{LP} 
\newblock G. Levin and F. Przytycki, 
\newblock External rays to periodic points,
\newblock {\em Israel Journal of Mathematics} {\bf 94} (1995), 29--57.

\bibitem{LvS}
\newblock G.  Levin and S. van Strien, 
\newblock Local connectivity of the Julia set of real polynomials. 
\newblock {\em Ann. of Math.} (2) {\bf  147}(3) (1998), 471--541.
   
\bibitem{Ly1}
\newblock M.Yu. Lyubich, 
\newblock Non-existence of wandering intervals and structure of topological attractors of one dimensional dynamical systems: 1. The case of negative Schwarzian derivative.
\newblock {\em Ergod. Th. \& Dynam. Sys.} {\bf 9} (1989), 737--749.

\bibitem{Ly1b}
\newblock M. Lyubich, 
\newblock Teichm\"uller space of Fibonacci maps.
\newblock {\em ArXiv} 9311213v1. 

\bibitem{Ly2}
\newblock M. Lyubich,
\newblock Dynamics of quadratic polynomials. I, II, 
\newblock {\em Acta Math.} {\bf 178} (1997), 185--247, 247--297.

\bibitem{Ly3}
\newblock M. Lyubich,
\newblock Feigenbaum-Coullet-Tresser universality and Milnor's hairiness conjecture. 
\newblock {\em Ann. of Math.} (2), {\bf 149}(2) (1999), 319--420.

\bibitem{LyY}
\newblock M. Lyubich and M. Yampolsky, 
\newblock Dynamics of quadratic polynomials: complex bounds for real maps, 
\newblock {\em Ann. Inst. Fourier} (Grenoble) {\bf 47} (1997), 1219--1255.

\bibitem{Martens}
\newblock M. Martens,
\newblock The periodic points of renormalization. 
\newblock \emph{Ann. of Math.} (2) {\bf 147}(3) (1998), 543--584.

\bibitem{MMS}
\newblock M. Martens, W. de Melo and S. van Strien,
\newblock Julia-Fatou-Sullivan theory for real one-dimensional dynamics. 
\newblock \emph{Acta Math.} {\bf 168}(3-4) (1992), 273--318.

\bibitem{manton} 
\newblock J.~Manton,
\newblock Differential Calculus, tensor products and the importance of notation,
\newblock {\tt arXiv:1208.0197v2}

\bibitem{McM} 
\newblock C.~McMullen,
\newblock Renormalization and 3-manifolds which fiber over the circle,
\newblock \emph{Ann. of Math. Studies} {\bf{142}}, 
Princeton University Press, 1996.

\bibitem{dMvS}
\newblock W.~de Melo and S.~van Strien, 
\newblock A structure theorem in one-dimensional dynamics. 
\newblock \emph{Ann. of Math.} (2) {\bf 129}(3) (1989), 519--546. 

\bibitem{demelovanstrien} 
\newblock W.~de Melo and S.~van Strien, 
\newblock {\em One-dimensional Dynamics}, 
\newblock Springer-Verlag, New York, 1993.



\bibitem{Mil1}
\newblock J. Milnor,
\newblock Dynamics in One Complex Variable. 
\newblock {\em Annals of Mathematics Studies} {\bf 160}, Princeton University Press. 2006.


\bibitem{Mil2}
\newblock  J. Milnor, 
\newblock Local connectivity of Julia sets: expository lectures, 
\newblock The Mandelbrot set, theme and variations, London Mathematical Society Lecture Note Series {\bf 274}
(Cambridge University Press, Cambridge, 2000) 67--116.

\bibitem{Shen}
\newblock W. Shen, 
\newblock On the metric properties of multimodal interval maps and $C^2$ density of Axiom A. 
\newblock {\em Invent. Math} {\bf 156} (2004) (2), 301--403. 

\bibitem{Smania1}
\newblock D. Smania, 
\newblock Complex bounds for multimodal maps: bounded combinatorics. 
\newblock {\em Nonlinearity} {\bf 14} (2001) (5), 1311--1330.

\bibitem{Smania2}
\newblock D. Smania,
\newblock Phase space universality for multimodal maps.
\newblock {\em  Bulletin of the Brazilian Mathematical Society} {\bf  36} (2) (2005), 225--274.


\bibitem{Smania3}
\newblock D. Smania,
\newblock On the hyperbolicity of the period-doubling fixed point.
\newblock {\em Transactions of the American Mathematical Society} {\bf 358} (4) (2006), 1827--1846.


\bibitem{Smania4}
\newblock D. Smania,
\newblock Solenoidal attractors with bounded combinatorics are shy.
\newblock {\tt  arXiv:1603.06300}.

\bibitem{Sor}
\newblock D. S\"orensen, 
\newblock Infinitely renormalizable quadratic polynomials, with non-locally connected Julia set,
\newblock {\em J. Geom. Anal.} {\bf 10} (2000), no. 1, 169--206.

\bibitem{vSV}
\newblock S. van Strien and E. Vargas, 
\newblock Real bounds, ergodicity and negative Schwarzian for multimodal maps.
\newblock  {\em J. Amer. Math. Soc.} {\bf 17} (2004) (4), 749--782.



\bibitem{sullivan0} 
\newblock D.~Sullivan, 
\newblock Quasiconformal homeomorphisms and dynamics. I. Solution of the Fatou-Julia problem on wandering domains. 
\newblock {\em Ann. of Math.} (2) {\bf 122} (1985) (3), 401--418.
                                                                                                                                                                                                                                                                                                                                                                                                                                                                                                                                                                                                                                                                                                                                                                                                                                                                                                                                                                                                                                                                                                                                                                                                                                                                                                                                                                                                                                                                                                                                                                                                                                                                                                   
\bibitem{sullivan} 
\newblock D.~Sullivan, 
\newblock Bounds, quadratic differentials, and renormalization conjectures, 
\newblock {\em AMS Centennial Publications}, {\bf{2}}, Mathematics into the Twenty-first Century, 1988. 


\bibitem{yam}
\newblock M.  Yampolsky, 
\newblock Hyperbolicity of renormalization of critical circle maps. 
\newblock {\em Publ. Math. Inst. Hautes Etudes Sci.} {\bf 96} (2002), 1--41.
 
 
 
\bibitem{Yoc}
\newblock J.-C. Yoccoz, 
\newblock On the local connectivity of the Mandelbrot set.  Unpublished, 1990.


\end{thebibliography}
\end{document}